\documentclass{article}
\usepackage{amssymb,amsmath,amssymb}
\usepackage{amsthm}
\usepackage{amsfonts}
\usepackage{graphics,psfrag, subcaption}
\usepackage{pgf,tikz}
\usetikzlibrary{arrows}
\usepackage{hyperref}

\usepackage[multiple]{footmisc} 

\DeclareMathOperator{\val}{val}
\DeclareMathOperator{\rep}{rep}
\DeclareMathOperator{\pref}{pref}

\newcommand\U{\mathcal{U}^\beta}
\newcommand\Llim{\mathcal{L}^\beta}
\newcommand\Aut{\mathcal{A}_\beta}
\newcommand\Lang{L_{U_\beta}}
\newcommand\Alp{A_{U_\beta}}
\newcommand\AEns{\mathcal{A}^\beta}

\theoremstyle{plain}
\newtheorem{theorem}{Theorem}
\newtheorem{lemma}[theorem]{Lemma}
\newtheorem{corollary}[theorem]{Corollary}
\newtheorem{prop}[theorem]{Proposition}

\theoremstyle{definition}
\newtheorem{definition}[theorem]{Definition}
\newtheorem{example}[theorem]{Example}
\newtheorem{remark}[theorem]{Remark}

\def\undertilde#1{\mathord{\vtop{\ialign{##\crcr
$\hfil\displaystyle{#1}\hfil$\crcr\noalign{\kern1.5pt\nointerlineskip}
$\hfil\tilde{}\hfil$\crcr\noalign{\kern1.5pt}}}}}

\usepackage{pgf}
\usepackage{tikz}
\usetikzlibrary{arrows,automata, shapes, positioning}
\usetikzlibrary{decorations}
\usetikzlibrary{snakes}
\usepackage{verbatim}

\newcommand{\seqnum}[1]{\href{http://oeis.org/#1}{\underline{#1}}}

\begin{document}
\begin{center}
\vskip 1cm{\LARGE\bf Convergence of Pascal-Like Triangles in Parry--Bertrand Numeration Systems}  \vskip 1cm
\large
Manon Stipulanti\footnote{Corresponding author.}\\
University of Li\`ege \\ 
Department of Mathematics \\ 
All\'ee de la D\'ecouverte 12 (B37) \\
4000 Li\`ege, Belgium \\
\href{mailto:M.Stipulanti@uliege.be}{\tt M.Stipulanti@uliege.be} \\

\end{center}

\vskip .2 in

\begin{abstract}
We pursue the investigation of generalizations of the Pascal triangle based on binomial coefficients of finite words. 
These coefficients count the number of times a finite word appears as a subsequence of another finite word. 
The finite words occurring in this paper belong to the language of a Parry numeration system satisfying the Bertrand property, i.e., we can add or remove trailing zeroes to valid representations.  
It is a folklore fact that the Sierpi\'{n}ski gasket is the limit set, for the Hausdorff distance, of a convergent sequence of normalized compact blocks extracted from the classical Pascal triangle modulo~$2$. In a similar way, we describe and study the subset of $[0, 1] \times [0, 1]$ associated with the latter generalization of the  Pascal triangle modulo a prime number.
\end{abstract}

\bigskip
\hrule
\bigskip



\bigskip
\hrule
\bigskip


\section{Introduction}\label{sec:intro}
Several generalizations and variations of the Pascal triangle exist and lead to interesting combinatorial, geometrical or dynamical properties~\cite{BNS, BS, vonHPS, JRV, LRS1}.
This paper is inspired by a series of papers based on generalizations of Pascal triangles to finite words~\cite{LRS1, LRS2, LRS3, LRS4}. 

\subsection{Binomial coefficients of words and Pascal-like triangles}
In this short subsection, we briefly introduce the concepts we use in this paper. 
For more definitions, see section~\ref{sec:background}.
A {\em finite word} is a finite sequence of letters belonging to a finite set called the {\em alphabet}. 
The \emph{binomial coefficient} $\binom{u}{v}$ of two finite words $u$ and $v$ is the number of times $v$ occurs as a subsequence of $u$ (meaning as a ``scattered'' subword). 

Let $A$ be a totally ordered alphabet, and let $L \subset A^*$ be an infinite language over $A$. 
We order the words of $L$ by increasing genealogical order and we write $L = \{ w_0 < w_1 < w_2 < \cdots \}$.
Associated with the language $L$, we define a Pascal-like triangle $\mathrm{P}_L : \mathbb{N} \times \mathbb{N} \to \mathbb{N}$ represented as an infinite table. 
The entry $\mathrm{P}_L(m,n)$ on the $m$th row and $n$th column of $\mathrm{P}_L$ is the integer $\binom{w_m}{w_n}$.

\subsection{Previous work}

Let $b$ be an integer greater than $1$.
We let $\rep_b(n)$ denote the (greedy) base-$b$ expansion of $n\in\mathbb{N}\setminus\{0\}$ starting with a non-zero digit. 
We set $\rep_b(0)$ to be the empty word denoted by $\varepsilon$. 
We let 
$$L_b=\{1,\ldots,b-1\}\{0,\ldots,b-1\}^*\cup\{\varepsilon\}$$ 
be the set of base-$b$ expansions of the non-negative integers. 
In~\cite{LRS1}, we study the particular case of $L=L_b$. 
The increasing genealogical order thus coincides with the classical order in $\mathbb{N}$. 
For example, see Table~\ref{tab:tp} for the first few values\footnote{Some of the objects discussed here are stored in Sloane's \emph{On-Line Encyclopedia of Integer Sequences}~\cite{Sloane}. See sequences \seqnum{A007306},
\seqnum{A282714},
\seqnum{A282715},
\seqnum{A282720},
\seqnum{A282728},
\seqnum{A284441},
and \seqnum{A284442}.} of $\mathrm{P}_2$. 
\begin{table}[h!t]
$$\begin{array}{r|cccccccc}
&\varepsilon&1&10&11&100&101&110&111\\
\hline
 \varepsilon&\mathbf{1} & 0 & 0 & 0 & 0 & 0 & 0 & 0 \\
 1&\mathbf{1} & \mathbf{1} & 0 & 0 & 0 & 0 & 0 & 0 \\
 10&1 & 1 & 1 & 0 & 0 & 0 & 0 & 0 \\
 11&\mathbf{1} & \mathbf{2} & 0 & \mathbf{1} & 0 & 0 & 0 & 0 \\
 100&1 & 1 & 2 & 0 & 1 & 0 & 0 & 0 \\
 101&1 & 2 & 1 & 1 & 0 & 1 & 0 & 0 \\
 110&1 & 2 & 2 & 1 & 0 & 0 & 1 & 0 \\
 111&\mathbf{1} & \mathbf{3} & 0 & \mathbf{3} & 0 & 0 & 0 & \mathbf{1} \\
\end{array}$$
\caption{The first few values in the generalized Pascal triangle $\mathrm{P}_2$ (\seqnum{A282714}).}
    \label{tab:tp}
\end{table}
Clearly, $\mathrm{P}_b$ contains several subtables corresponding to the usual Pascal triangle. 
For instance, it contains $(b-1)$ copies of the usual Pascal triangle obtained when only considering words of the form $a^m$ with $a\in\{1,\ldots,b-1\}$ and $m\ge 0$ since $\binom{a^m}{a^n}=\binom{m}{n}$. In Table~\ref{tab:tp}, a copy of the classical Pascal triangle is written in bold. 

Considering the intersection of the lattice $\mathbb{N}^2$ with $[0, 2^n ] \times [0, 2^n ]$, the first $2^n$ rows and columns of the usual Pascal triangle modulo~$2$ $(\binom{i}{j} \bmod{2})_{0 \le i,j<2^n}$ provide a coloring of this lattice. 
If we normalize this compact set by a homothety of ratio $1/2^n$, we get a sequence of subsets of $[0, 1] \times [0, 1]$ which converges, for the Hausdorff distance, to the Sierpi\'{n}ski gasket when $n$ tends to infinity.
In the extended context described above, the case when $b=2$ gives similar results and the limit set, generalizing the Sierpi\'{n}ski gasket, is described using a simple combinatorial property called $(\star)$~\cite{LRS1}. 

Inspired by~\cite{LRS1}, we study the sequence $(S_b(n))_{n\ge 0}$ which counts, on each row $m$ of $\mathrm{P}_b$, the number of words of $L_b$ occurring as subwords of the $m$th word in $L_b$, i.e., $S_b(m)=\#\{n\in\mathbb{N} \mid \mathrm{P}_b(m,n)>0 \}$. 
This sequence is shown to be $b$-regular~\cite{LRS2, LRS4}. 
We also consider the summatory function $(A_b(n))_{n\ge 0}$ of the sequence $(S_b(n))_{n\ge 0}$ and study its behavior~\cite{LRS3, LRS4}. 

So far, the setting is the one of integer bases. As a first extension, we handle the case of the Fibonacci numeration system, i.e., with the language $L_F=\{\varepsilon\}\cup 1\{0,01\}^*$~\cite{LRS2, LRS3}. 
It turns out that the sequence $(S_F(n))_{n\ge 0}$ counting the number of words in $L_F$ occurring as subwords of the $n$th word in $L_F$ has properties similar to those of $(S_b(n))_{n\ge 0}$. 
Finally, the summatory function $(A_F(n))_{n\ge 0}$ of the sequence $(S_F(n))_{n\ge 0}$ has a behavior similar to the one of $(A_b(n))_{n\ge 0}$.

\subsection{Our contribution}

The Fibonacci numeration system belongs to an extensively studied family of numeration systems called \textit{Parry--Bertrand numeration systems}, which are based on particular sequences $(U(n))_{n\ge 0}$ (the precise definitions are given later). 
In this paper, we fill the gap between integer bases and the Fibonacci numeration systems by extending the results of~\cite{LRS1} to every Parry--Bertrand numeration system. 
First, we generalize the construction of Pascal-like triangles to every Parry--Bertrand numeration system. 
For a given Parry--Bertrand numeration system based on a particular sequence $(U(n))_{n\ge 0}$, we consider the intersection of the lattice $\mathbb{N}^2$ with $[0, U(n) ] \times [0, U(n) ]$. 
Then the first $U(n)$ rows and columns of the corresponding generalized Pascal triangle modulo~$2$ provide a coloring of this lattice regarding the parity of the corresponding binomial coefficients. 
If we normalize this compact set by a homothety of ratio $1/U(n)$, we get a sequence in $[0, 1] \times [0, 1]$ which converges, for the Hausdorff distance, to a limit set when $n$ tends to infinity.
Again, the limit set is described using a simple combinatorial property extending the one from~\cite{LRS1}. 

Compared to the integer bases, new technicalities have to be taken into account to generalize Pascal triangles to a large class of numeration systems. 
The numeration systems occurring in this paper essentially have two properties. 
The first one is that the language of the numeration system comes from a particular automaton. 
The second one is the Bertrand condition which allows to delete or add ending zeroes to valid representations. 

This paper is organized as follows. 
In Section~\ref{sec:background}, we collect necessary background. 
Section~\ref{sec:star} is devoted to a special combinatorial property that extends the $(\star)$ condition from~\cite{LRS1}. 
This new condition allows us to define a sequence of compact sets, which is shown to be a Cauchy sequence in Section~\ref{sec:AEns}. 
In Section~\ref{sec:limite}, using the property of the latter sequence, we define a limit set which is the analogue of the Sierpi\'{n}ski gasket in the classical framework.
We show that the sequence of subblocks of the generalized Pascal triangle modulo~$2$ in a Parry--Bertrand numeration converges to this new limit set. 
As a final remark, we consider the latter sequence of compact sets modulo any prime number. 


\section{Background and particular framework}\label{sec:background}

We begin this section with well-known definitions from combinatorics on words; see, for instance,~\cite{Rigo1}.
Let $A$ be an alphabet, i.e., a finite set.  
The elements of $A$ are called \textit{letters}. 
A finite sequence over $A$ is called a \textit{finite word}. 
The length of a finite word $w$, denoted by $|w|$, is the number of letters belonging to $w$. 
The only word of length $0$ is the empty word $\varepsilon$. 
The set of finite words over the
alphabet $A$ including the empty word (resp., excluding the empty word) is denoted by $A^*$ (resp., $A^+$). 
The set of words of length $n$ over $A$ is denoted by $A^n$. 
If $u$ and $v$ are two finite words belonging to $A^*$, the \textit{binomial coefficient} $\binom{u}{v}$ of $u$ and $v$ is the number of occurrences of $v$ as a subsequence of $u$, meaning as a scattered subword.  
The sequences over $A$ indexed by $\mathbb{N}$ are the \emph{infinite words} over $A$.
If $w$ is a finite non-empty word over $A$, we let $w^\omega:=www\cdots$ denote the infinite word obtained by concatenating infinitely many copies of $w$. 
If $L \subset A^*$ is a set of finite words and $u\in A^*$ is a finite word, we let $u^{-1}.L$ denote the set of words $\{ v \in A^* \mid uv \in L \}$.
Let $A$ be totally ordered. 
If $u, v \in A^*$ are two words, we say that \emph{$u$ is less than $v$ in the genealogical order} and we write $u < v$ if either $|u| < |v|$, or if $|u| = |v|$ and there exist words $p, q, r \in A^*$ and letters $a, b\in A$ with
$u = paq$, $v = pbr$ and $a < b$. By $u \le v$, we mean that either $u < v$, or $u = v$.

In the first part of this section, we gather two results on binomial coefficients of finite words and integers.
For a proof of the first lemma, we refer the reader to~\cite[Chap.~6]{Lot}. 

\begin{lemma}\label{lem:lothaire-bin}
Let $A$ be a finite alphabet. Let $u,v \in A^*$ and let $a,b \in A$. Then we have
    $$\binom{ua}{vb}=\binom{u}{vb}+\delta_{a,b}\binom{u}{v}$$
where $\delta_{a,b}$ is equal to $1$ if $a=b$, $0$ otherwise.
\end{lemma}

Let us also recall Lucas' theorem relating classical binomial coefficients modulo a prime number $p$ with base-$p$ expansions. 
See~\cite[p. 230]{Lucas} or~\cite{Fine}. 
Note that in the following statement, if the base-$p$ expansions of $m$ and $n$ are not of the same length, then we pad the shortest with leading zeroes.

\begin{theorem}\label{thm:lucas}
Let $m$ and $n$ be two non-negative integers and let $p$ be a prime number. 
If 
$$m=m_kp^k + m_{k-1}p^{k-1} + \cdots + m_1 p + m_0$$ and $$n=n_kp^k + n_{k-1}p^{k-1} + \cdots + n_1 p + n_0$$ with $m_i,n_i\in\{0,\ldots,p-1\}$ for all $i$, then the following congruence relation holds $$\binom{m}{n}\equiv \prod _{i=0}^k \binom {m_i}{n_i} \bmod p,$$
using the following convention: $\binom {m}{n}= 0$ if $m < n$.
\end{theorem}

In the last part of this section, we introduce the setting of particular numeration systems that are used in this paper: the Parry--Bertrand numeration systems. 
First of all, we recall several definitions and results about representations of real numbers. 
For more details, see, for instance,~\cite[Chap.~2]{BR},~\cite[Chap.~7]{Lot2} or~\cite{Rigo2}.

\begin{definition}\label{def:beta-exp}
Let $\beta \in \mathbb{R}_{>1}$ and let $A_\beta=\{0, 1, \ldots, \lceil \beta \rceil -1 \}$.
Every real number $x \in [0, 1)$ can be written as a series
$$
x = \sum_{j=1}^{+\infty} c_j \beta^{-j}
$$
where $c_j \in A_\beta$ for all $j\ge 1$, and where $ \lceil \cdot \rceil$ denotes the \emph{ceiling function} defined by $ \lceil x \rceil = \inf \{ z\in \mathbb{Z} \mid z \ge x \}$. 
The infinite word $c_1 c_2 \cdots$ is called a \emph{$\beta$-representation} of $x$. 
Among all the $\beta$-representations of $x$, we define the \emph{$\beta$-expansion} $d_\beta(x)$ of $x$ obtained in a greedy way, i.e., for all $j\ge 1$, we have  $c_j \beta^{-j} + c_{j+1} \beta^{-j-1} + \cdots  < \beta^{-j+1}$.
We also make use of the following convention: if $w=w_n\cdots w_0$ is a finite word (resp., $w=w_1 w_2 \cdots$ is an infinite word) over $A_\beta$, the notation $0.w$ has to be understood as the real number $\sum_{j=0}^{n} w_j \beta^{j-n-1}$ (resp., $\sum_{j=1}^{+\infty} w_j \beta^{-j}$); it actually corresponds to the value of the word $w$ in base $\beta$. 

In an analogous way, the \emph{$\beta$-expansion} $d_\beta(1)$ of $1$ the following infinite word over $A_\beta$
$$
d_\beta(1) := 
\left\{
    \begin{array}{cl}
        (\beta-1)^\omega, & \mbox{if } \beta \in \mathbb{N}; \\
         ( \lceil \beta \rceil -1 ) d_\beta ( 1 - (\lceil \beta \rceil -1)/ \beta), & \mbox{otherwise.}
    \end{array}
\right.
$$ 
In other words, if $\beta$ is not an integer, the first digit of the $\beta$-expansion of $1$ is $\lceil \beta \rceil -1$ and the other digits are derived from the $\beta$-expansion of $1 - (\lceil \beta \rceil -1)/ \beta$.

Let $d_\beta (1) = (t_n )_{n\ge 1}$ be the $\beta$-expansion of $1$. 
Observe that $t_1 = \lceil \beta \rceil -1$. 
We define the \emph{quasi-greedy $\beta$-expansion} $d_\beta^*(1)$ of 1 as follows. 
If $d_\beta (1)= t_1 \cdots t_m$ is finite, i.e., $t_m \neq 0$ and $t_j = 0$ for all $j > m$, then $d_\beta^*(1) = (t_1 \cdots t_{m-1} (t_m - 1))^\omega$, otherwise $d_\beta^*(1) = d_\beta(1)$.
\end{definition} 

A real number $\beta > 1$ is a \emph{Parry number} if $d_\beta(1)$ is ultimately periodic. If $d_\beta(1)$ is finite, then $\beta$ is called a \emph{simple Parry number}.
In this case, Proposition~\ref{pro:auto} gives an easy way to decide if an infinite word is the $\beta$-expansion of a real number~\cite{Parry}. 
For more details, see, for instance,~\cite[Chap.~7]{Lot2}. 
First, let us recall the definition of a deterministic finite automaton. 

\begin{definition}
A deterministic finite automaton (DFA),
over an alphabet $A$ is given by a $5$-tuple $\mathcal{A} = (Q, q_0 , A, \delta, F )$ 
where $Q$ is a finite set of states, $q_0\in Q$ is the initial state, $\delta : Q \times A \mapsto Q$ is the transition function and $F \subset Q$ is the set of final states (graphically represented by two concentric circles). 
The map $\delta$ can be extended to $Q \times A^*$ by setting $\delta(q, \varepsilon) = q$ and $\delta(q, wa) = \delta(\delta(q, w), a)$ for all $q \in Q$,
$a \in A$ and $w \in A^*$.
We also say that a word $w$ is \emph{accepted} by the automaton if $\delta(q_0,w)\in F$. 
\end{definition}

\begin{prop}\label{pro:auto}
Let $\beta \in \mathbb{R}_{>1}$ be a Parry number. 
\begin{itemize}
\item[(a)] Suppose that $d_\beta (1)= t_1 \cdots t_m$ is finite, i.e., $t_m \neq 0$ and $t_j = 0$ for all $j > m$. 
Then an infinite word is the $\beta$-expansion of a real number in $[0,1)$ if and only if it is the label of a path in the automaton $\Aut=(\{a_0,\ldots,a_{m-1}\}, a_0, A_\beta, \delta, \{a_0,\ldots,a_{m-1}\})$ depicted in Figure~\ref{fig:ParrySimple}.
\item[(b)] Suppose that $d_\beta (1)= t_1 \cdots t_m (t_{m+1}\cdots t_{m+k})^\omega$ where $m,k$ are taken to be minimal. 
Then an infinite word is the $\beta$-expansion of a real number $[0,1)$ if and only if it is the label of a path in the automaton $\Aut=(\{a_0,\ldots,a_{m+k-1}\}, a_0, A_\beta, \delta, \{a_0,\ldots,a_{m+k-1}\})$ depicted in Figure~\ref{fig:ParryNonSimple}.
\end{itemize}
\end{prop}

\begin{figure}
\vspace{-2cm}
\begin{subfigure}[b]{1\textwidth}
\centering
\begin{tikzpicture}
\tikzstyle{every node}=[shape=circle,fill=none,draw=black,minimum size=20pt,inner sep=2pt]
\node[accepting](a) at (0,0) {$a_0$};
\node[accepting](b) at (2,0) {$a_1$};
\node[accepting](c) at (4,0) {$a_2$};
\node[accepting](d) at (12,0){$a_{m-2}$};
\node[accepting](e) at (14,0) {$a_{m-1}$};

\tikzstyle{every node}=[shape=circle,fill=none,draw=none,minimum size=10pt,inner sep=2pt]
\node(a1) at (0,1) {$0,\ldots,t_1-1$};
\node(a2) at (1,-0.75) {$0,\ldots,t_2-1$};
\node(a3) at (6,0) {$\ldots$};
\node(a4) at (8,0) {$\ldots$};
\node(a5) at (10,0) {$\ldots$};
\node(a6) at (3,0.9) {$0,\ldots,t_3-1$};
\node(a7) at (1,0.2) {$t_1$};
\node(a8) at (3,0.2) {$t_2$};
\node(a8) at (5,0.2) {$t_3$};
\node(a9) at (11,0.2) {$t_{m-2}$};
\node(a10) at (13,0.2) {$t_{m-1}$};
\node(a11) at (11,2) {$0,\ldots,t_m-1$};
\node(a12) at (-1,0) {};

\tikzstyle{every path}=[color =black, line width = 0.5 pt]
\tikzstyle{every node}=[shape=circle,minimum size=5pt,inner sep=2pt]
\draw [->] (a12) to [] node [] {}  (a);

\draw [->] (a) to [loop above] node [] {}  (a);
\draw [->] (a) to [] node [] {}  (b);

\draw [->] (b) to [] node [above] {}  (c);
\draw [->] (b) to [bend left=30] node [above] {}  (a);

\draw [->] (c) to [] node [above] {}  (a3);
\draw [->] (c) to [bend right=30] node [above] {}  (a);

\draw [->] (a5) to [] node [above] {}  (d);
\draw [->] (d) to [] node [above] {}  (e);

\draw [->] (e) to [bend right=30] node [above] {}  (a);
;
\end{tikzpicture}
\caption{The case when $d_\beta(1)$ is finite.}
\label{fig:ParrySimple}
\end{subfigure}

\begin{subfigure}[b]{1\textwidth}
\centering
\begin{tikzpicture}
\tikzstyle{every node}=[shape=circle,fill=none,draw=black,minimum size=20pt,inner sep=2pt]
\node[accepting](a) at (0,0) {$a_0$};
\node[accepting](b) at (2,0) {$a_1$};
\node[accepting](c) at (4,0) {$a_2$};
\node[accepting](d) at (12,0){$a_{m-2}$};
\node[accepting](e) at (14,0) {$a_{m-1}$};
\node[accepting](f) at (14,-2.5) {$a_{m}$};
\node[accepting](g) at (5.5,-2.5) {$a_{m+k-1}$};

\tikzstyle{every node}=[shape=circle,fill=none,draw=none,minimum size=10pt,inner sep=2pt]
\node(a1) at (0,1) {$0,\ldots,t_1-1$};
\node(a2) at (2,-0.6) {$0,\ldots,t_2-1$};
\node(a3) at (6,0) {$\ldots$};
\node(a4) at (8,0) {$\ldots$};
\node(a5) at (10,0) {$\ldots$};
\node(a6) at (3,0.9) {$0,\ldots,t_3-1$};
\node(a7) at (1,0.2) {$t_1$};
\node(a8) at (3,0.2) {$t_2$};
\node(a8) at (5,0.2) {$t_3$};
\node(a9) at (11,0.2) {$t_{m-2}$};
\node(a10) at (13,0.2) {$t_{m-1}$};
\node(a11) at (11,2) {$0,\ldots,t_m-1$};
\node(a12) at (-1,0) {};
\node(a13) at (14.25,-1.25) {$t_{m}$};
\node(a14) at (12,-2.5) {$\ldots$};
\node(a15) at (10,-2.5) {$\ldots$};
\node(a16) at (8,-2.5) {$\ldots$};
\node(a17) at (13,-2.3) {$t_{m+1}$};
\node(a18) at (7,-2.3) {$t_{m+k-1}$};
\node(a19) at (12,-4.1) {$0,\ldots,t_{m+1}-1$};
\node(a20) at (10,-0.9) {$t_{m+k}$};
\node(a21) at (3.7,-1.6) {$0,\ldots,t_{m+k}-1$};

\tikzstyle{every path}=[color =black, line width = 0.5 pt]
\tikzstyle{every node}=[shape=circle,minimum size=5pt,inner sep=2pt]
\draw [->] (a12) to [] node [] {}  (a);

\draw [->] (a) to [loop above] node [] {}  (a);
\draw [->] (a) to [] node [] {}  (b);

\draw [->] (b) to [] node [above] {}  (c);
\draw [->] (b) to [bend left=25] node [above] {}  (a);

\draw [->] (c) to [] node [above] {}  (a3);
\draw [->] (c) to [bend right=30] node [above] {}  (a);

\draw [->] (a5) to [] node [above] {}  (d);
\draw [->] (d) to [] node [above] {}  (e);

\draw [->] (e) to [bend right=30] node [above] {}  (a);
\draw [->] (e) to [] node [above] {}  (f);

\draw [->] (f) to [] node [above] {}  (a14);
\draw [->] (f) to [bend left=40] node [above] {}  (a);

\draw [->] (a16) to [] node [above] {}  (g);
\draw [->] (g) to [bend left=30] node [above] {}  (f);
\draw [->] (g) to [bend left=15] node [above] {}  (a);
;
\end{tikzpicture}
\caption{The case when $d_\beta(1)$ is ultimately periodic but not finite.}
\label{fig:ParryNonSimple}
\end{subfigure}
\caption{The automaton $\Aut$ in function of the ultimately periodic word $d_\beta(1)$.}
\end{figure}
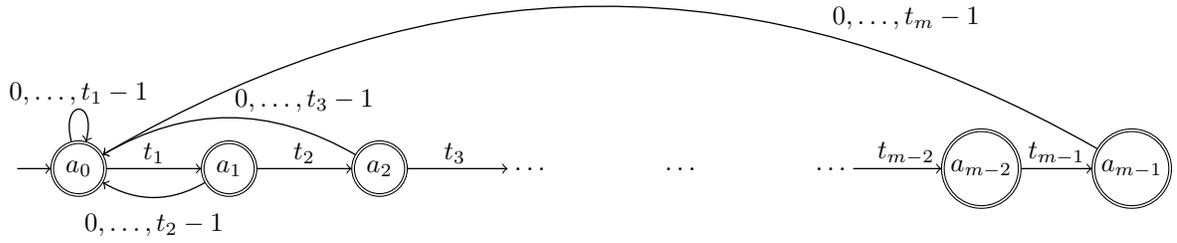
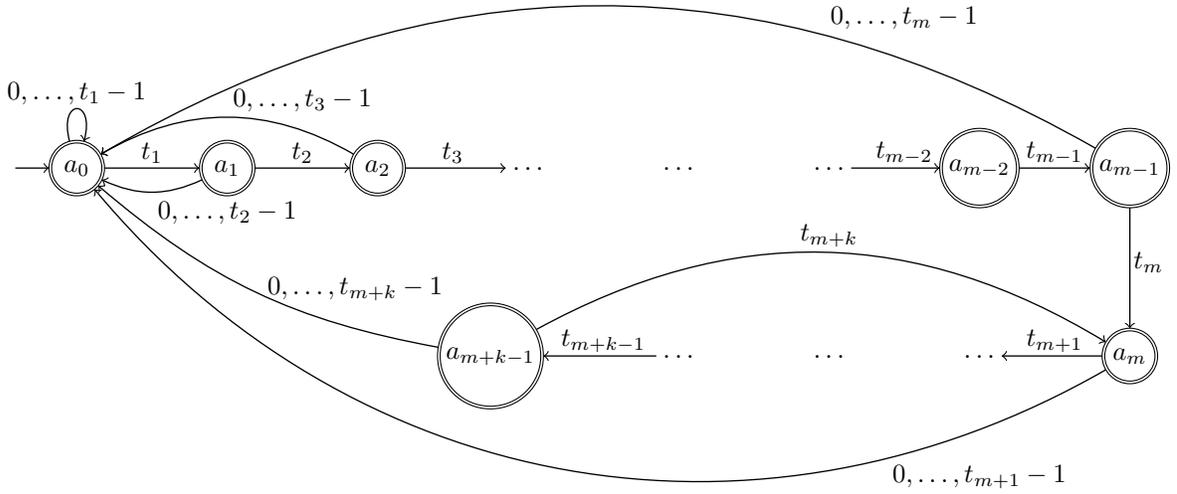

Let us illustrate the previous proposition. For other examples, see, for instance,~\cite{CRRW}.

\begin{example}\label{ex:integer}
If $\beta \in \mathbb{R}_{>1}$ is an integer, then $d_\beta(1)=d^*_\beta(1)=(\beta  -1)^\omega$. The automaton $\Aut$ consists of a single initial and final state $a_0$ with a loop of labels $0, 1, \ldots, \beta-1$.
\end{example}

\begin{example}\label{ex:phi-et-phi-sqr}
Consider the golden ratio $\varphi$. Since $1 = 1/\varphi + 1/\varphi^2$, we have $d_\varphi(1)=11$ and $d^*_\varphi(1)=(10)^\omega$. 
It is thus a Parry number.
The automaton $\mathcal{A}_\varphi$ is depicted in Figure~\ref{fig:Aut-phi}. 

The square $\varphi^2$ of the golden ratio is again a Parry number with $d_{\varphi^2}(1)=d^*_{\varphi^2}(1)=21^\omega$.
The automaton $\mathcal{A}_{\varphi^2}$ is depicted in Figure~\ref{fig:Aut-phi-sqr}. 

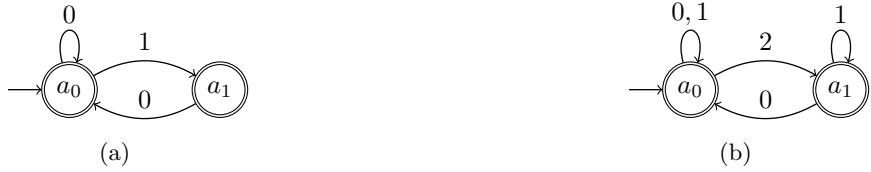
\begin{figure}
\vspace{-2cm}
\begin{subfigure}[b]{0.5\textwidth}
\centering
\begin{tikzpicture}
\tikzstyle{every node}=[shape=circle,fill=none,draw=black,minimum size=20pt,inner sep=2pt]
\node[accepting](a) at (0,0) {$a_0$};
\node[accepting](b) at (2,0) {$a_1$};

\tikzstyle{every node}=[shape=circle,fill=none,draw=none,minimum size=10pt,inner sep=2pt]
\node(a0) at (-1,0) {};
\node(a1) at (0,1) {$0$};

\tikzstyle{every path}=[color =black, line width = 0.5 pt]
\tikzstyle{every node}=[shape=circle,minimum size=5pt,inner sep=2pt]
\draw [->] (a0) to [] node [] {}  (a);
\draw [->] (a) to [loop above] node [] {}  (a);

\draw [->] (a) to [bend left] node [above] {$1$}  (b);
\draw [->] (b) to [bend left] node [above] {$0$}  (a);
;ParrySimple
\end{tikzpicture}
\caption{}
\label{fig:Aut-phi}
\end{subfigure}%
\begin{subfigure}[b]{0.5\textwidth}
\centering
\begin{tikzpicture}
\tikzstyle{every node}=[shape=circle,fill=none,draw=black,minimum size=20pt,inner sep=2pt]
\node[accepting](a) at (0,0) {$a_0$};
\node[accepting](b) at (2,0) {$a_1$};

\tikzstyle{every node}=[shape=circle,fill=none,draw=none,minimum size=10pt,inner sep=2pt]
\node(a0) at (-1,0) {};
\node(a1) at (0,1) {$0,1$};
\node(a2) at (2,1) {$1$};

\tikzstyle{every path}=[color =black, line width = 0.5 pt]
\tikzstyle{every node}=[shape=circle,minimum size=5pt,inner sep=2pt]
\draw [->] (a0) to [] node [] {}  (a);
\draw [->] (a) to [loop above] node [] {}  (a);
\draw [->] (b) to [loop above] node [] {}  (b);

\draw [->] (a) to [bend left] node [above] {$2$}  (b);
\draw [->] (b) to [bend left] node [above] {$0$}  (a);
;
\end{tikzpicture}
\caption{}
\label{fig:Aut-phi-sqr}
\end{subfigure}
\caption{The automaton $\mathcal{A}_{\varphi}$ (on the left) and the automaton $\mathcal{A}_{\varphi^2}$ (on the right).}
\end{figure}
\end{example}

With every Parry number is canonically associated a linear numeration system. 
Let us recall the definition of such numeration systems. 

\begin{definition}\label{def:lin-num-syst}
Let $U = (U(n))_{n\ge 0}$ be a sequence of integers such that $U$ is increasing, $U(0)=1$ and $\sup_{n \geq 0} \frac{U(n+1)}{U(n)}$ is bounded by a constant. 
We say that $U$ is a \emph{linear numeration system} if $U$ satisfies a linear recurrence relation, i.e., there exist $k\ge 1$ and $a_0, \ldots, a_{k-1} \in \mathbb{Z}$ such that
\begin{align}
\forall n \ge 0, \quad U(n+k) = a_{k-1}\, U(n+k-1) + \cdots + a_0\, U(n)
 \label{eq:eq-rec-U}. 
\end{align}
Let $n$ be a positive integer. 
By successive Euclidean divisions, there exists $\ell\ge 1$ such that 
$$ n = \sum_{j=0}^{\ell -1} c_j\, U(j)
$$
where the $c_j$'s are non-negative integers and $c_{\ell-1}$ is non-zero. 
The word $c_{\ell -1}\cdots c_0$ is called the \emph{normal $U$-representation} of $n$ and is denoted by $\rep_U(n)$. 
In other words, the word $c_{\ell-1}\cdots c_0$ is the greedy expansion of $n$ in the considered numeration system. 
We set $\rep_U(0):=\varepsilon$. 
Finally, we refer to $L_U:=\rep_U(\mathbb{N})$ as the \emph{language of the numeration} and we let $A_U$ denote the minimal alphabet such that $L_U \subset A_U^*$.  
If $d_{r}\cdots d_0$ is a word over an alphabet of digits, then its \emph{$U$-numerical value} is
$$\val_U(d_{r}\cdots d_0):=\sum_{j=0}^{r} d_j\, U(j).$$
Observe that, if $\val_U (d_{r}\cdots d_0) = n$, then the word $d_r \cdots d_0$ is a
$U$-representation of $n$ (but not necessarily its normal $U$-representation).
\end{definition}

\begin{definition}\label{def:Parry-Bertrand-num-syst}
Let $\beta \in \mathbb{R}_{>1}$ be a Parry number. We define a particular linear numeration system $U_\beta:=(U_\beta(n))_{n\ge 0}$ associated with $\beta$ as follows. 

If $d_\beta (1)= t_1 \cdots t_m$ is finite ($t_m \neq 0$), then we set $U_\beta(0):=1$, $U_\beta(i):= t_1 U_\beta(i-1) + \cdots + t_i U_\beta(0) + 1$ for all $i\in \{1,\ldots,m-1\}$ and, for all $n\ge m$,
\begin{align*}
U_\beta(n)&:= t_1 U_\beta(n-1) + \cdots + t_m U_\beta(n-m).
\end{align*}

If $d_\beta (1)= t_1 \cdots t_m (t_{m+1}\cdots t_{m+k})^\omega$ ($m,k$ are minimal), then we set
$U_\beta(0):=1$, $U_\beta(i):= t_1 U_\beta(i-1) + \cdots + t_i U_\beta(0) + 1$ for all $i\in \{1,\ldots,m+k-1\}$ and, for all $n\ge m+k$,
\begin{align*}
U_\beta(n):=& t_1 U_\beta(n-1) + \cdots + t_{m+k} U_\beta(n-m-k) + U_\beta(n-k) \\ 
&- t_1 U_\beta(n-k-1) - \cdots - t_m U_\beta(n-m-k). 
\end{align*}
\end{definition}

The linear numeration system $U_\beta$ from Definition~\ref{def:Parry-Bertrand-num-syst} has an interesting property: it is a Bertrand numeration system. 

\begin{definition}\label{def:bertrand}
A linear numeration system $U = (U(n))_{n\ge 0}$ is a \emph{Bertrand numeration system} if, for all $w \in A^+_U$, $w \in L_U \Leftrightarrow w0 \in L_U $.
\end{definition}

Bertrand proved that the linear numeration system $U_\beta$ associated with the Parry number $\beta$ from Definition~\ref{def:Parry-Bertrand-num-syst} is the unique linear numeration system associated with $\beta$ that is also a Bertrand numeration system~\cite{Bertrand}. 
In that case~\cite{Bertrand}, any word $w$ in the set $0^*\Lang$ of all normal $U_\beta$-representations with leading zeroes is the label of a path in the automaton $\Aut$ from Proposition~\ref{pro:auto}. 

Finally, every Parry number is a Perron number~\cite[Chap.~7]{Lot2}. 
A real number $\beta >1$ is a \textit{Perron number} if it is an algebraic integer whose conjugates have modulus less than $\beta$. 
Numeration systems based on Perron numbers are defined as follows and have the property~\eqref{eq:conv-beta-U}, which is often used in this paper. 

\begin{definition}\label{def:Perron-num-syst}
Let $U = (U(n))_{n\ge 0}$ be a linear numeration
system. 
Consider the characteristic polynomial of the recurrence~\eqref{eq:eq-rec-U} given by $P(X) = X^k - a_{k-1} X^{k-1} - \cdots - a_1 X - a_0$.
If $P$ is the minimal polynomial of a Perron number $\beta \in \mathbb{R}_{>1}$, we say that $U$ is a \emph{Perron numeration system}. 
In this case, the polynomial $P$ can be factored as 
$$
P(X) = (X-\beta) (X-\alpha_2) \cdots (X-\alpha_k) 
$$
where the complex numbers $\alpha_2,\ldots ,\alpha_k$ are the conjugates of $\beta$, and, for all $j > 1$, we have $|\alpha_j | < \beta$. 
Using a well-known fact regarding recurrence relations, we have
$$
U(n) = c_1 \beta^n + c_2 \alpha_2^n + \cdots + c_k \alpha^n_k \quad \forall n \ge 0 
$$
where $c_1, \ldots, c_k$ are complex numbers depending on the initial values of $U$.
Since $|\alpha_j | < \beta$ for all $j > 1$, we have  
\begin{align}
\lim\limits_{n \rightarrow +\infty} \frac{U(n)}{\beta^n} = c_1 \label{eq:conv-beta-U}. 
\end{align}
\end{definition}

\begin{remark}
Note that if two Perron numeration systems are associated with the same Perron number, then these two systems only differ by the choice of the initial values $U(0),\ldots, U(k-1)$. 
The choice of those initial values is of great importance.  See, for instance, Example~\ref{ex:phi}.
\end{remark}

\begin{example}
The usual integer base system is a special
case of a Perron--Bertrand numeration system. 
\end{example}

\begin{example}\label{ex:phi}
The golden ratio $\varphi$ is a Perron number whose minimal polynomial is $P(X)=X^2-X-1$. 
A Perron--Bertrand numeration system associated with $\varphi$ is the Fibonacci numeration system based on the Fibonacci numbers $(F(n))_{n\ge 0}$ defined by $F(0)=1$, $F(1)=2$ and $F(n+2)=F(n+1)+F(n)$. 
If we change the initial conditions and set $F'(0)=1$, $F'(1)=3$ and $F'(n+2)=F'(n+1)+F'(n)$, we again get a Perron numeration associated with $\varphi$ which is not a Bertrand numeration system. 
Indeed, $2$ is a greedy representation, but not
$20$ because $\rep_{F'} (\val_{F'} (20)) = 102$.
\end{example}

The particular setting of this paper is the following one: we let $\beta \in \mathbb{R}_{>1}$ be a Parry number and we constantly use the special Parry--Bertrand numeration $U_\beta$ from Definition~\ref{def:Parry-Bertrand-num-syst}. From Definition~\ref{def:beta-exp} and Definition~\ref{def:lin-num-syst}, the alphabet $\Alp$ is the set $\{0, 1, \ldots, \lceil \beta \rceil -1 \}$ and the language of the system of numeration $U_\beta$ is $\Lang \subset A_{U_\beta}^*$ (which is defined using the automaton $\Aut$ from Proposition~\ref{pro:auto}).
To end this section, we prove a useful lemma about binomial coefficients of words ending with blocks of zeroes. 

\begin{lemma}\label{lem:bin-zeroes}
For all non-empty words $u,v\in \Lang$ and all $k\in\mathbb{N}$, we have
$$
\binom{u0^k}{v0^k} = \sum_{j=0}^k \binom{k}{j} \binom{u}{v0^j}.
$$ 
\end{lemma}
\begin{proof}
We proceed by induction on $k\in\mathbb{N}$. 
If $k=0$, the result is obvious. 
Suppose that the result holds true for all non-empty words $u,v \in \Lang$ and for $0,\ldots,k$. 
We show that it still holds true for all non-empty words $u,v \in \Lang$ and $k+1$. 
Using Lemma~\ref{lem:lothaire-bin}, we first have 
$$
\binom{u0^{k+1}}{v0^{k+1}} = \binom{u0^k}{v'0^k} + \binom{u0^k}{v0^k}
$$
where $v'=v0 \in \Lang$ since $U_\beta$ is a Parry--Bertrand numeration system. 
By induction hypothesis, we get
\begin{eqnarray*}
\binom{u0^{k+1}}{v0^{k+1}} 
& =& \sum_{j=1}^{k+1} \binom{k}{j-1} \binom{u}{v0^j} +  \sum_{j=0}^k \binom{k}{j} \binom{u}{v0^j} \\
& =& \binom{k+1}{k+1} \binom{u}{v0^{k+1}} + \sum_{j=1}^k \left( \binom{k}{j-1} + \binom{k}{j} \right) \binom{u}{v0^j} + \binom{k+1}{0} \binom{u}{v} \\
&=& \sum_{j=0}^{k+1} \binom{k+1}{j} \binom{u}{v0^j}.
\end{eqnarray*}
\end{proof}


\section{The $(\star)$ condition}\label{sec:star}
We let $w_n=\rep_{U_\beta}(n)$ denote the $n$th word of the language $\Lang$ in the genealogical order. 
The generalized Pascal triangle $\mathrm{P}_{U_\beta} : \mathbb{N} \times \mathbb{N} \to \mathbb{N} : (i,j) \mapsto \binom{w_i}{w_j}$ is represented as an infinite table\footnote{Using the notation $\binom{u}{v}$, the rows (resp., columns) of $\mathrm{P}_{U_\beta}$ are indexed by the words $u$ (resp., $v$).} whose entry on the $i$th row and the $j$th column is the binomial coefficient  
$\binom{w_i}{w_j}$. 
For instance, when $\beta=\varphi$, the first few values in the generalized Pascal triangle $\mathrm{P}_{U_\varphi}$ are given in Table~\ref{tab:tpFib} below.
Considering the intersection of the lattice $\mathbb{N}^2$ with $[0,U_\beta(n)]\times [0,U_\beta(n)]$, the first $U_\beta(n)$ rows and columns of the generalized Pascal triangle $\mathrm{P}_{U_\beta}$ modulo~$2$ 
$$\left(\binom{w_i}{w_j}\bmod{2}\right)_{0\le i,j< U_\beta(n)}$$ 
provide a coloring of this lattice, leading to a sequence of compact subsets of $\mathbb{R}^2$. 
If we normalize these sets respectively by a homothety of ratio $1/U_\beta(n)$, we define a sequence $(\U_n)_{n\ge 0}$ of subsets of $[0,1]\times [0,1]$.

\begin{definition}\label{def:Un}
Let $Q:=[0,1]\times[0,1]$. Consider the sequence $(\U_n)_{n\ge 0}$ of sets in $[0,1]\times [0,1]$ defined for all $n\ge 0$ by 

$$
\U_n:= \frac{1}{U_\beta(n)}  \bigcup \left\{(\val_{U_\beta}(v),\val_{U_\beta}(u))+Q\mid u,v\in \Lang, \binom{u}{v}\equiv 1\bmod{2}\right\}\subset [0,1]\times [0,1].
$$
Each $\U_n$ is a finite union of squares of size $1/U_{\beta}(n)$ and is thus compact.
\end{definition}

\begin{example}
When $\beta=\varphi$ is the golden ratio, the first values in the generalized Pascal triangle $\mathrm{P}_{U_\varphi}$ are given in Table~\ref{tab:tpFib}.
\begin{table}[h!t]
$$\begin{array}{rr|cccccccc}
&&&&&&j&&\\
&\binom{w_i}{w_j}&\varepsilon&1&10&100&101&1000&1001&1010\\
\hline
 &\varepsilon&1 & 0 & 0 & 0 & 0 & 0 & 0 & 0 \\
 &1&1 & 1 & 0 & 0 & 0 & 0 & 0 & 0 \\
 &10&1 & 1 & 1 & 0 & 0 & 0 & 0 & 0 \\
 i&100&1 & 1 & 2 & 1 & 0 & 0 & 0 & 0 \\
 &101&1 & 2 & 1 & 0 & 1 & 0 & 0 & 0 \\
 &1000&1 & 1 & 3 & 3 & 0 & 1 & 0 & 0 \\
 &1001&1 & 2 & 2 & 1 & 2 & 0 & 1 & 0 \\
 &1010&1 & 2 & 3 & 1 & 1 & 0 & 0 & 1 \\
\end{array}$$
\caption{The first few values in the generalized Pascal triangle $\mathrm{P}_{U_\varphi}$.}
    \label{tab:tpFib}
\end{table}
The sets $\mathcal{U}^\varphi_3$, $\mathcal{U}^\varphi_4$ and $\mathcal{U}^\varphi_5$ are depicted in Figure~\ref{fig:exemples-U-1}. 
The set $\mathcal{U}^\varphi_9$ is depicted in Figure~\ref{fig:exemples-U-2} given in the appendix.

\begin{center}
\begin{figure}
\begin{subfigure}[b]{1\textwidth}
\begin{center}
\begin{tikzpicture}
\draw[fill=black] (0,1) -- (1,1) -- (1,2) -- (0,2) -- cycle;
\draw[fill=white] (1,1) -- (2,1) -- (2,2) -- (1,2) -- cycle;
\draw[fill=white] (2,1) -- (3,1) -- (3,2) -- (2,2) -- cycle;
\draw[fill=white] (3,1) -- (4,1) -- (4,2) -- (3,2) -- cycle;
\draw[fill=white] (4,1) -- (5,1) -- (5,2) -- (4,2) -- cycle;

\draw[fill=black] (0,0) -- (1,0) -- (1,1) -- (0,1) -- cycle;
\draw[fill=black] (1,0) -- (2,0) -- (2,1) -- (1,1) -- cycle;
\draw[fill=white] (2,0) -- (3,0) -- (3,1) -- (2,1) -- cycle;
\draw[fill=white] (3,0) -- (4,0) -- (4,1) -- (3,1) -- cycle;
\draw[fill=white] (4,0) -- (5,0) -- (5,1) -- (4,1) -- cycle;

\draw[fill=black] (0,-1) -- (1,-1) -- (1,0) -- (0,0) -- cycle;
\draw[fill=black] (1,-1) -- (2,-1) -- (2,0) -- (1,0) -- cycle;
\draw[fill=black] (2,-1) -- (3,-1) -- (3,0) -- (2,0) -- cycle;
\draw[fill=white] (3,-1) -- (4,-1) -- (4,0) -- (3,0) -- cycle;
\draw[fill=white] (4,-1) -- (5,-1) -- (5,0) -- (4,0) -- cycle;

\draw[fill=black] (0,-2) -- (1,-2) -- (1,-1) -- (0,-1) -- cycle;
\draw[fill=black] (1,-2) -- (2,-2) -- (2,-1) -- (1,-1) -- cycle;
\draw[fill=white] (2,-2) -- (3,-2) -- (3,-1) -- (2,-1) -- cycle;
\draw[fill=black] (3,-2) -- (4,-2) -- (4,-1) -- (3,-1) -- cycle;
\draw[fill=white] (4,-2) -- (5,-2) -- (5,-1) -- (4,-1) -- cycle;

\draw[fill=black] (0,-3) -- (1,-3) -- (1,-2) -- (0,-2) -- cycle;
\draw[fill=white] (1,-3) -- (2,-3) -- (2,-2) -- (1,-2) -- cycle;
\draw[fill=black] (2,-3) -- (3,-3) -- (3,-2) -- (2,-2) -- cycle;
\draw[fill=white] (3,-3) -- (4,-3) -- (4,-2) -- (3,-2) -- cycle;
\draw[fill=black] (4,-3) -- (5,-3) -- (5,-2) -- (4,-2) -- cycle;

\draw (-0.8,1.7) node[anchor=north west] {$\varepsilon$};
\draw (-0.8,0.7) node[anchor=north west] {$1$};
\draw (-0.8,-0.3) node[anchor=north west] {$10$};
\draw (-0.8,-1.3) node[anchor=north west] {$100$};
\draw (-0.8,-2.3) node[anchor=north west] {$101$};

\draw (0.3,2.5) node[anchor=north west] {$\varepsilon$};
\draw (1.3,2.5) node[anchor=north west] {$1$};
\draw (2.2,2.5) node[anchor=north west] {$10$};
\draw (3.1,2.5) node[anchor=north west] {$100$};
\draw (4.1,2.5) node[anchor=north west] {$101$};
\end{tikzpicture}
\end{center}
\caption{The set $\mathcal{U}^\varphi_3$.}
\end{subfigure}

\begin{subfigure}[b]{1\textwidth}
\begin{center}
\begin{tikzpicture}
\draw[fill=black] (0,0.625) -- (0.625,0.625) -- (0.625,1.25) -- (0,1.25) -- cycle;
\draw[fill=white] (0.625,0.625) -- (1.25,0.625) -- (1.25,1.25) -- (0.625,1.25) -- cycle;
\draw[fill=white] (1.25,0.625) -- (1.875,0.625) -- (1.875,1.25) -- (1.25,1.25) -- cycle;
\draw[fill=white] (1.875,0.625) -- (2.5,0.625) -- (2.5,1.25) -- (1.875,1.25) -- cycle;
\draw[fill=white] (2.5,0.625) -- (3.125,0.625) -- (3.125,1.25) -- (2.5,1.25) -- cycle;
\draw[fill=white] (3.125,0.625) -- (3.75,0.625) -- (3.75,1.25) -- (3.125,1.25) -- cycle;
\draw[fill=white] (3.75,0.625) -- (4.375,0.625) -- (4.375,1.25) -- (3.75,1.25) -- cycle;
\draw[fill=white] (4.375,0.625) -- (5,0.625) -- (5,1.25) -- (4.375,1.25) -- cycle;

\draw[fill=black] (0,0) -- (0.625,0) -- (0.625,0.625) -- (0,0.625) -- cycle;
\draw[fill=black] (0.625,0) -- (1.25,0) -- (1.25,0.625) -- (0.625,0.625) -- cycle;
\draw[fill=white] (1.25,0) -- (1.875,0) -- (1.875,0.625) -- (1.25,0.625) -- cycle;
\draw[fill=white] (1.875,0) -- (2.5,0) -- (2.5,0.625) -- (1.875,0.625) -- cycle;
\draw[fill=white] (2.5,0) -- (3.125,0) -- (3.125,0.625) -- (2.5,0.625) -- cycle;
\draw[fill=white] (3.125,0) -- (3.75,0) -- (3.75,0.625) -- (3.125,0.625) -- cycle;
\draw[fill=white] (3.75,0) -- (4.375,0) -- (4.375,0.625) -- (3.75,0.625) -- cycle;
\draw[fill=white] (4.375,0) -- (5,0) -- (5,0.625) -- (4.375,0.625) -- cycle;

\draw[fill=black] (0,-0.625) -- (0.625,-0.625) -- (0.625,0) -- (0,0) -- cycle;
\draw[fill=black] (0.625,-0.625) -- (1.25,-0.625) -- (1.25,0) -- (0.625,0) -- cycle;
\draw[fill=black] (1.25,-0.625) -- (1.875,-0.625) -- (1.875,0) -- (1.25,0) -- cycle;
\draw[fill=white] (1.875,-0.625) -- (2.5,-0.625) -- (2.5,0) -- (1.875,0) -- cycle;
\draw[fill=white] (2.5,-0.625) -- (3.125,-0.625) -- (3.125,0) -- (2.5,0) -- cycle;
\draw[fill=white] (3.125,-0.625) -- (3.75,-0.625) -- (3.75,0) -- (3.125,0) -- cycle;
\draw[fill=white] (3.75,-0.625) -- (4.375,-0.625) -- (4.375,0) -- (3.75,0) -- cycle;
\draw[fill=white] (4.375,-0.625) -- (5,-0.625) -- (5,0) -- (4.375,0) -- cycle;

\draw[fill=black] (0,-1.25) -- (0.625,-1.25) -- (0.625,-0.625) -- (0,-0.625) -- cycle;
\draw[fill=black] (0.625,-1.25) -- (1.25,-1.25) -- (1.25,-0.625) -- (0.625,-0.625) -- cycle;
\draw[fill=white] (1.25,-1.25) -- (1.875,-1.25) -- (1.875,-0.625) -- (1.25,-0.625) -- cycle;
\draw[fill=black] (1.875,-1.25) -- (2.5,-1.25) -- (2.5,-0.625) -- (1.875,-0.625) -- cycle;
\draw[fill=white] (2.5,-1.25) -- (3.125,-1.25) -- (3.125,-0.625) -- (2.5,-0.625) -- cycle;
\draw[fill=white] (3.125,-1.25) -- (3.75,-1.25) -- (3.75,-0.625) -- (3.125,-0.625) -- cycle;
\draw[fill=white] (3.75,-1.25) -- (4.375,-1.25) -- (4.375,-0.625) -- (3.75,-0.625) -- cycle;
\draw[fill=white] (4.375,-1.25) -- (5,-1.25) -- (5,-0.625) -- (4.375,-0.625) -- cycle;

\draw[fill=black] (0,-1.875) -- (0.625,-1.875) -- (0.625,-1.25) -- (0,-1.25) -- cycle;
\draw[fill=white] (0.625,-1.875) -- (1.25,-1.875) -- (1.25,-1.25) -- (0.625,-1.25) -- cycle;
\draw[fill=black] (1.25,-1.875) -- (1.875,-1.875) -- (1.875,-1.25) -- (1.25,-1.25) -- cycle;
\draw[fill=white] (1.875,-1.875) -- (2.5,-1.875) -- (2.5,-1.25) -- (1.875,-1.25) -- cycle;
\draw[fill=black] (2.5,-1.875) -- (3.125,-1.875) -- (3.125,-1.25) -- (2.5,-1.25) -- cycle;
\draw[fill=white] (3.125,-1.875) -- (3.75,-1.875) -- (3.75,-1.25) -- (3.125,-1.25) -- cycle;
\draw[fill=white] (3.75,-1.875) -- (4.375,-1.875) -- (4.375,-1.25) -- (3.75,-1.25) -- cycle;
\draw[fill=white] (4.375,-1.875) -- (5,-1.875) -- (5,-1.25) -- (4.375,-1.25) -- cycle;

\draw[fill=black] (0,-2.5) -- (0.625,-2.5) -- (0.625,-1.875) -- (0,-1.875) -- cycle;
\draw[fill=black] (0.625,-2.5) -- (1.25,-2.5) -- (1.25,-1.875) -- (0.625,-1.875) -- cycle;
\draw[fill=black] (1.25,-2.5) -- (1.875,-2.5) -- (1.875,-1.875) -- (1.25,-1.875) -- cycle;
\draw[fill=black] (1.875,-2.5) -- (2.5,-2.5) -- (2.5,-1.875) -- (1.875,-1.875) -- cycle;
\draw[fill=white] (2.5,-2.5) -- (3.125,-2.5) -- (3.125,-1.875) -- (2.5,-1.875) -- cycle;
\draw[fill=black] (3.125,-2.5) -- (3.75,-2.5) -- (3.75,-1.875) -- (3.125,-1.875) -- cycle;
\draw[fill=white] (3.75,-2.5) -- (4.375,-2.5) -- (4.375,-1.875) -- (3.75,-1.875) -- cycle;
\draw[fill=white] (4.375,-2.5) -- (5,-2.5) -- (5,-1.875) -- (4.375,-1.875) -- cycle;

\draw[fill=black] (0,-3.125) -- (0.625,-3.125) -- (0.625,-2.5) -- (0,-2.5) -- cycle;
\draw[fill=white] (0.625,-3.125) -- (1.25,-3.125) -- (1.25,-2.5) -- (0.625,-2.5) -- cycle;
\draw[fill=white] (1.25,-3.125) -- (1.875,-3.125) -- (1.875,-2.5) -- (1.25,-2.5) -- cycle;
\draw[fill=black] (1.875,-3.125) -- (2.5,-3.125) -- (2.5,-2.5) -- (1.875,-2.5) -- cycle;
\draw[fill=white] (2.5,-3.125) -- (3.125,-3.125) -- (3.125,-2.5) -- (2.5,-2.5) -- cycle;
\draw[fill=white] (3.125,-3.125) -- (3.75,-3.125) -- (3.75,-2.5) -- (3.125,-2.5) -- cycle;
\draw[fill=black] (3.75,-3.125) -- (4.375,-3.125) -- (4.375,-2.5) -- (3.75,-2.5) -- cycle;
\draw[fill=white] (4.375,-3.125) -- (5,-3.125) -- (5,-2.5) -- (4.375,-2.5) -- cycle;

\draw[fill=black] (0,-3.75) -- (0.625,-3.75) -- (0.625,-3.125) -- (0,-3.125) -- cycle;
\draw[fill=white] (0.625,-3.75) -- (1.25,-3.75) -- (1.25,-3.125) -- (0.625,-3.125) -- cycle;
\draw[fill=black] (1.25,-3.75) -- (1.875,-3.75) -- (1.875,-3.125) -- (1.25,-3.125) -- cycle;
\draw[fill=black] (1.875,-3.75) -- (2.5,-3.75) -- (2.5,-3.125) -- (1.875,-3.125) -- cycle;
\draw[fill=black] (2.5,-3.75) -- (3.125,-3.75) -- (3.125,-3.125) -- (2.5,-3.125) -- cycle;
\draw[fill=white] (3.125,-3.75) -- (3.75,-3.75) -- (3.75,-3.125) -- (3.125,-3.125) -- cycle;
\draw[fill=white] (3.75,-3.75) -- (4.375,-3.75) -- (4.375,-3.125) -- (3.75,-3.125) -- cycle;
\draw[fill=black] (4.375,-3.75) -- (5,-3.75) -- (5,-3.125) -- (4.375,-3.125) -- cycle;

\draw (-1,1.1) node[anchor=north west] {$\varepsilon$};
\draw (-1,0.5) node[anchor=north west] {$1$};
\draw (-1,-0.1) node[anchor=north west] {$10$};
\draw (-1,-0.7) node[anchor=north west] {$100$};
\draw (-1,-1.3) node[anchor=north west] {$101$};
\draw (-1,-1.9) node[anchor=north west] {$1000$};
\draw (-1,-2.6) node[anchor=north west] {$1001$};
\draw (-1,-3.2) node[anchor=north west] {$1010$};

\draw (0,1.5) node[anchor=north west, rotate=40] {$\varepsilon$};
\draw (0.6,1.5) node[anchor=north west, rotate=40] {$1$};
\draw (1.2,1.5) node[anchor=north west, rotate=40] {$10$};
\draw (1.8,1.5) node[anchor=north west, rotate=40] {$100$};
\draw (2.4,1.5) node[anchor=north west, rotate=40] {$101$};
\draw (3,1.5) node[anchor=north west, rotate=40] {$1000$};
\draw (3.6,1.5) node[anchor=north west, rotate=40] {$1001$};
\draw (4.2,1.5) node[anchor=north west, rotate=40] {$1010$};

\end{tikzpicture}
\end{center}
\caption{The set $\mathcal{U}^\varphi_4$.}
\end{subfigure}

\begin{subfigure}[b]{1\textwidth}
\begin{center}
\begin{tikzpicture}
\draw[fill=black] (0,0.3846) -- (0.3846,0.3846) -- (0.3846,0.7692) -- (0,0.7692) -- cycle;
\draw[fill=white] (0.3846,0.3846) -- (0.7692,0.3846) -- (0.7692,0.7692) -- (0.3846,0.7692) -- cycle;
\draw[fill=white] (0.7692,0.3846) -- (1.1538,0.3846) -- (1.1538,0.7692) -- (0.7692,0.7692) -- cycle;
\draw[fill=white] (1.1538,0.3846) -- (1.5384,0.3846) -- (1.5384,0.7692) -- (1.1538,0.7692) -- cycle;
\draw[fill=white] (1.5384,0.3846) -- (1.923,0.3846) -- (1.923,0.7692) -- (1.5384,0.7692) -- cycle;
\draw[fill=white] (1.923,0.3846) -- (2.3076,0.3846) -- (2.3076,0.7692) -- (1.923,0.7692) -- cycle;
\draw[fill=white] (2.3076,0.3846) -- (2.6922,0.3846) -- (2.6922,0.7692) -- (2.3076,0.7692) -- cycle;
\draw[fill=white] (2.6922,0.3846) -- (3.0768,0.3846) -- (3.0768,0.7692) -- (2.6922,0.7692) -- 
cycle;
\draw[fill=white] (3.0768,0.3846) -- (3.4614,0.3846) -- (3.4614,0.7692) -- (3.0768,0.7692) -- cycle;
\draw[fill=white] (3.4614,0.3846) -- (3.846,0.3846) -- (3.846,0.7692) -- (3.4614,0.7692) -- cycle;
\draw[fill=white] (3.846,0.3846) -- (4.2306,0.3846) -- (4.2306,0.7692) -- (3.846,0.7692) -- cycle;
\draw[fill=white] (4.2306,0.3846) -- (4.6152,0.3846) -- (4.6152,0.7692) -- (4.2306,0.7692) -- cycle;
\draw[fill=white] (4.6152,0.3846) -- (5,0.3846) -- (5,0.7692) -- (4.6152,0.7692) -- cycle;

\draw[fill=black] (0,0) -- (0.3846,0) -- (0.3846,0.3846) -- (0,0.3846) -- cycle;
\draw[fill=black] (0.3846,0) -- (0.7692,0) -- (0.7692,0.3846) -- (0.3846,0.3846) -- cycle;
\draw[fill=white] (0.7692,0) -- (1.1538,0) -- (1.1538,0.3846) -- (0.7692,0.3846) -- cycle;
\draw[fill=white] (1.1538,0) -- (1.5384,0) -- (1.5384,0.3846) -- (1.1538,0.3846) -- cycle;
\draw[fill=white] (1.5384,0) -- (1.923,0) -- (1.923,0.3846) -- (1.5384,0.3846) -- cycle;
\draw[fill=white] (1.923,0) -- (2.3076,0) -- (2.3076,0.3846) -- (1.923,0.3846) -- cycle;
\draw[fill=white] (2.3076,0) -- (2.6922,0) -- (2.6922,0.3846) -- (2.3076,0.3846) -- cycle;
\draw[fill=white] (2.6922,0) -- (3.0768,0) -- (3.0768,0.3846) -- (2.6922,0.3846) -- 
cycle;
\draw[fill=white] (3.0768,0) -- (3.4614,0) -- (3.4614,0.3846) -- (3.0768,0.3846) -- cycle;
\draw[fill=white] (3.4614,0) -- (3.846,0) -- (3.846,0.3846) -- (3.4614,0.3846) -- cycle;
\draw[fill=white] (3.846,0) -- (4.2306,0) -- (4.2306,0.3846) -- (3.846,0.3846) -- cycle;
\draw[fill=white] (4.2306,0) -- (4.6152,0) -- (4.6152,0.3846) -- (4.2306,0.3846) -- cycle;
\draw[fill=white] (4.6152,0) -- (5,0) -- (5,0.3846) -- (4.6152,0.3846) -- cycle;

\draw[fill=black] (0,-0.3846) -- (0.3846,-0.3846) -- (0.3846,0) -- (0,0) -- cycle;
\draw[fill=black] (0.3846,-0.3846) -- (0.7692,-0.3846) -- (0.7692,0) -- (0.3846,0) -- cycle;
\draw[fill=black] (0.7692,-0.3846) -- (1.1538,-0.3846) -- (1.1538,0) -- (0.7692,0) -- cycle;
\draw[fill=white] (1.1538,-0.3846) -- (1.5384,-0.3846) -- (1.5384,0) -- (1.1538,0) -- cycle;
\draw[fill=white] (1.5384,-0.3846) -- (1.923,-0.3846) -- (1.923,0) -- (1.5384,0) -- cycle;
\draw[fill=white] (1.923,-0.3846) -- (2.3076,-0.3846) -- (2.3076,0) -- (1.923,0) -- cycle;
\draw[fill=white] (2.3076,-0.3846) -- (2.6922,-0.3846) -- (2.6922,0) -- (2.3076,0) -- cycle;
\draw[fill=white] (2.6922,-0.3846) -- (3.0768,-0.3846) -- (3.0768,0) -- (2.6922,0) -- 
cycle;
\draw[fill=white] (3.0768,-0.3846) -- (3.4614,-0.3846) -- (3.4614,0) -- (3.0768,0) -- cycle;
\draw[fill=white] (3.4614,-0.3846) -- (3.846,-0.3846) -- (3.846,0) -- (3.4614,0) -- cycle;
\draw[fill=white] (3.846,-0.3846) -- (4.2306,-0.3846) -- (4.2306,0) -- (3.846,0) -- cycle;
\draw[fill=white] (4.2306,-0.3846) -- (4.6152,-0.3846) -- (4.6152,0) -- (4.2306,0) -- cycle;
\draw[fill=white] (4.6152,-0.3846) -- (5,-0.3846) -- (5,0) -- (4.6152,0) -- cycle;

\draw[fill=black] (0,-0.7692) -- (0.3846,-0.7692) -- (0.3846,-0.3846) -- (0,-0.3846) -- cycle;
\draw[fill=black] (0.3846,-0.7692) -- (0.7692,-0.7692) -- (0.7692,-0.3846) -- (0.3846,-0.3846) -- cycle;
\draw[fill=white] (0.7692,-0.7692) -- (1.1538,-0.7692) -- (1.1538,-0.3846) -- (0.7692,-0.3846) -- cycle;
\draw[fill=black] (1.1538,-0.7692) -- (1.5384,-0.7692) -- (1.5384,-0.3846) -- (1.1538,-0.3846) -- cycle;
\draw[fill=white] (1.5384,-0.7692) -- (1.923,-0.7692) -- (1.923,-0.3846) -- (1.5384,-0.3846) -- cycle;
\draw[fill=white] (1.923,-0.7692) -- (2.3076,-0.7692) -- (2.3076,-0.3846) -- (1.923,-0.3846) -- cycle;
\draw[fill=white] (2.3076,-0.7692) -- (2.6922,-0.7692) -- (2.6922,-0.3846) -- (2.3076,-0.3846) -- cycle;
\draw[fill=white] (2.6922,-0.7692) -- (3.0768,-0.7692) -- (3.0768,-0.3846) -- (2.6922,-0.3846) -- 
cycle;
\draw[fill=white] (3.0768,-0.7692) -- (3.4614,-0.7692) -- (3.4614,-0.3846) -- (3.0768,-0.3846) -- cycle;
\draw[fill=white] (3.4614,-0.7692) -- (3.846,-0.7692) -- (3.846,-0.3846) -- (3.4614,-0.3846) -- cycle;
\draw[fill=white] (3.846,-0.7692) -- (4.2306,-0.7692) -- (4.2306,-0.3846) -- (3.846,-0.3846) -- cycle;
\draw[fill=white] (4.2306,-0.7692) -- (4.6152,-0.7692) -- (4.6152,-0.3846) -- (4.2306,-0.3846) -- cycle;
\draw[fill=white] (4.6152,-0.7692) -- (5,-0.7692) -- (5,-0.3846) -- (4.6152,-0.3846) -- cycle;

\draw[fill=black] (0,-1.1538) -- (0.3846,-1.1538) -- (0.3846,-0.7692) -- (0,-0.7692) -- cycle;
\draw[fill=white] (0.3846,-1.1538) -- (0.7692,-1.1538) -- (0.7692,-0.7692) -- (0.3846,-0.7692) -- cycle;
\draw[fill=black] (0.7692,-1.1538) -- (1.1538,-1.1538) -- (1.1538,-0.7692) -- (0.7692,-0.7692) -- cycle;
\draw[fill=white] (1.1538,-1.1538) -- (1.5384,-1.1538) -- (1.5384,-0.7692) -- (1.1538,-0.7692) -- cycle;
\draw[fill=black] (1.5384,-1.1538) -- (1.923,-1.1538) -- (1.923,-0.7692) -- (1.5384,-0.7692) -- cycle;
\draw[fill=white] (1.923,-1.1538) -- (2.3076,-1.1538) -- (2.3076,-0.7692) -- (1.923,-0.7692) -- cycle;
\draw[fill=white] (2.3076,-1.1538) -- (2.6922,-1.1538) -- (2.6922,-0.7692) -- (2.3076,-0.7692) -- cycle;
\draw[fill=white] (2.6922,-1.1538) -- (3.0768,-1.1538) -- (3.0768,-0.7692) -- (2.6922,-0.7692) -- 
cycle;
\draw[fill=white] (3.0768,-1.1538) -- (3.4614,-1.1538) -- (3.4614,-0.7692) -- (3.0768,-0.7692) -- cycle;
\draw[fill=white] (3.4614,-1.1538) -- (3.846,-1.1538) -- (3.846,-0.7692) -- (3.4614,-0.7692) -- cycle;
\draw[fill=white] (3.846,-1.1538) -- (4.2306,-1.1538) -- (4.2306,-0.7692) -- (3.846,-0.7692) -- cycle;
\draw[fill=white] (4.2306,-1.1538) -- (4.6152,-1.1538) -- (4.6152,-0.7692) -- (4.2306,-0.7692) -- cycle;
\draw[fill=white] (4.6152,-1.1538) -- (5,-1.1538) -- (5,-0.7692) -- (4.6152,-0.7692) -- cycle;

\draw[fill=black] (0,-1.5384) -- (0.3846,-1.5384) -- (0.3846,-1.1538) -- (0,-1.1538) -- cycle;
\draw[fill=black] (0.3846,-1.5384) -- (0.7692,-1.5384) -- (0.7692,-1.1538) -- (0.3846,-1.1538) -- cycle;
\draw[fill=black] (0.7692,-1.5384) -- (1.1538,-1.5384) -- (1.1538,-1.1538) -- (0.7692,-1.1538) -- cycle;
\draw[fill=black] (1.1538,-1.5384) -- (1.5384,-1.5384) -- (1.5384,-1.1538) -- (1.1538,-1.1538) -- cycle;
\draw[fill=white] (1.5384,-1.5384) -- (1.923,-1.5384) -- (1.923,-1.1538) -- (1.5384,-1.1538) -- cycle;
\draw[fill=black] (1.923,-1.5384) -- (2.3076,-1.5384) -- (2.3076,-1.1538) -- (1.923,-1.1538) -- cycle;
\draw[fill=white] (2.3076,-1.5384) -- (2.6922,-1.5384) -- (2.6922,-1.1538) -- (2.3076,-1.1538) -- cycle;
\draw[fill=white] (2.6922,-1.5384) -- (3.0768,-1.5384) -- (3.0768,-1.1538) -- (2.6922,-1.1538) -- 
cycle;
\draw[fill=white] (3.0768,-1.5384) -- (3.4614,-1.5384) -- (3.4614,-1.1538) -- (3.0768,-1.1538) -- cycle;
\draw[fill=white] (3.4614,-1.5384) -- (3.846,-1.5384) -- (3.846,-1.1538) -- (3.4614,-1.1538) -- cycle;
\draw[fill=white] (3.846,-1.5384) -- (4.2306,-1.5384) -- (4.2306,-1.1538) -- (3.846,-1.1538) -- cycle;
\draw[fill=white] (4.2306,-1.5384) -- (4.6152,-1.5384) -- (4.6152,-1.1538) -- (4.2306,-1.1538) -- cycle;
\draw[fill=white] (4.6152,-1.5384) -- (5,-1.5384) -- (5,-1.1538) -- (4.6152,-1.1538) -- cycle;

\draw[fill=black] (0,-1.923) -- (0.3846,-1.923) -- (0.3846,-1.5384) -- (0,-1.5384) -- cycle;
\draw[fill=white] (0.3846,-1.923) -- (0.7692,-1.923) -- (0.7692,-1.5384) -- (0.3846,-1.5384) -- cycle;
\draw[fill=white] (0.7692,-1.923) -- (1.1538,-1.923) -- (1.1538,-1.5384) -- (0.7692,-1.5384) -- cycle;
\draw[fill=black] (1.1538,-1.923) -- (1.5384,-1.923) -- (1.5384,-1.5384) -- (1.1538,-1.5384) -- cycle;
\draw[fill=white] (1.5384,-1.923) -- (1.923,-1.923) -- (1.923,-1.5384) -- (1.5384,-1.5384) -- cycle;
\draw[fill=white] (1.923,-1.923) -- (2.3076,-1.923) -- (2.3076,-1.5384) -- (1.923,-1.5384) -- cycle;
\draw[fill=black] (2.3076,-1.923) -- (2.6922,-1.923) -- (2.6922,-1.5384) -- (2.3076,-1.5384) -- cycle;
\draw[fill=white] (2.6922,-1.923) -- (3.0768,-1.923) -- (3.0768,-1.5384) -- (2.6922,-1.5384) -- 
cycle;
\draw[fill=white] (3.0768,-1.923) -- (3.4614,-1.923) -- (3.4614,-1.5384) -- (3.0768,-1.5384) -- cycle;
\draw[fill=white] (3.4614,-1.923) -- (3.846,-1.923) -- (3.846,-1.5384) -- (3.4614,-1.5384) -- cycle;
\draw[fill=white] (3.846,-1.923) -- (4.2306,-1.923) -- (4.2306,-1.5384) -- (3.846,-1.5384) -- cycle;
\draw[fill=white] (4.2306,-1.923) -- (4.6152,-1.923) -- (4.6152,-1.5384) -- (4.2306,-1.5384) -- cycle;
\draw[fill=white] (4.6152,-1.923) -- (5,-1.923) -- (5,-1.5384) -- (4.6152,-1.5384) -- cycle;

\draw[fill=black] (0,-2.3076) -- (0.3846,-2.3076) -- (0.3846,-1.923) -- (0,-1.923) -- cycle;
\draw[fill=white] (0.3846,-2.3076) -- (0.7692,-2.3076) -- (0.7692,-1.923) -- (0.3846,-1.923) -- cycle;
\draw[fill=black] (0.7692,-2.3076) -- (1.1538,-2.3076) -- (1.1538,-1.923) -- (0.7692,-1.923) -- cycle;
\draw[fill=black] (1.1538,-2.3076) -- (1.5384,-2.3076) -- (1.5384,-1.923) -- (1.1538,-1.923) -- cycle;
\draw[fill=black] (1.5384,-2.3076) -- (1.923,-2.3076) -- (1.923,-1.923) -- (1.5384,-1.923) -- cycle;
\draw[fill=white] (1.923,-2.3076) -- (2.3076,-2.3076) -- (2.3076,-1.923) -- (1.923,-1.923) -- cycle;
\draw[fill=white] (2.3076,-2.3076) -- (2.6922,-2.3076) -- (2.6922,-1.923) -- (2.3076,-1.923) -- cycle;
\draw[fill=black] (2.6922,-2.3076) -- (3.0768,-2.3076) -- (3.0768,-1.923) -- (2.6922,-1.923) -- 
cycle;
\draw[fill=white] (3.0768,-2.3076) -- (3.4614,-2.3076) -- (3.4614,-1.923) -- (3.0768,-1.923) -- cycle;
\draw[fill=white] (3.4614,-2.3076) -- (3.846,-2.3076) -- (3.846,-1.923) -- (3.4614,-1.923) -- cycle;
\draw[fill=white] (3.846,-2.3076) -- (4.2306,-2.3076) -- (4.2306,-1.923) -- (3.846,-1.923) -- cycle;
\draw[fill=white] (4.2306,-2.3076) -- (4.6152,-2.3076) -- (4.6152,-1.923) -- (4.2306,-1.923) -- cycle;
\draw[fill=white] (4.6152,-2.3076) -- (5,-2.3076) -- (5,-1.923) -- (4.6152,-1.923) -- cycle;

\draw[fill=black] (0,-2.6922) -- (0.3846,-2.6922) -- (0.3846,-2.3076) -- (0,-2.3076) -- cycle;
\draw[fill=black] (0.3846,-2.6922) -- (0.7692,-2.6922) -- (0.7692,-2.3076) -- (0.3846,-2.3076) -- cycle;
\draw[fill=white] (0.7692,-2.6922) -- (1.1538,-2.6922) -- (1.1538,-2.3076) -- (0.7692,-2.3076) -- cycle;
\draw[fill=white] (1.1538,-2.6922) -- (1.5384,-2.6922) -- (1.5384,-2.3076) -- (1.1538,-2.3076) -- cycle;
\draw[fill=white] (1.5384,-2.6922) -- (1.923,-2.6922) -- (1.923,-2.3076) -- (1.5384,-2.3076) -- cycle;
\draw[fill=white] (1.923,-2.6922) -- (2.3076,-2.6922) -- (2.3076,-2.3076) -- (1.923,-2.3076) -- cycle;
\draw[fill=white] (2.3076,-2.6922) -- (2.6922,-2.6922) -- (2.6922,-2.3076) -- (2.3076,-2.3076) -- cycle;
\draw[fill=white] (2.6922,-2.6922) -- (3.0768,-2.6922) -- (3.0768,-2.3076) -- (2.6922,-2.3076) -- 
cycle;
\draw[fill=black] (3.0768,-2.6922) -- (3.4614,-2.6922) -- (3.4614,-2.3076) -- (3.0768,-2.3076) -- cycle;
\draw[fill=white] (3.4614,-2.6922) -- (3.846,-2.6922) -- (3.846,-2.3076) -- (3.4614,-2.3076) -- cycle;
\draw[fill=white] (3.846,-2.6922) -- (4.2306,-2.6922) -- (4.2306,-2.3076) -- (3.846,-2.3076) -- cycle;
\draw[fill=white] (4.2306,-2.6922) -- (4.6152,-2.6922) -- (4.6152,-2.3076) -- (4.2306,-2.3076) -- cycle;
\draw[fill=white] (4.6152,-2.6922) -- (5,-2.6922) -- (5,-2.3076) -- (4.6152,-2.3076) -- cycle;

\draw[fill=black] (0,-3.0768) -- (0.3846,-3.0768) -- (0.3846,-2.6922) -- (0,-2.6922) -- cycle;
\draw[fill=white] (0.3846,-3.0768) -- (0.7692,-3.0768) -- (0.7692,-2.6922) -- (0.3846,-2.6922) -- cycle;
\draw[fill=black] (0.7692,-3.0768) -- (1.1538,-3.0768) -- (1.1538,-2.6922) -- (0.7692,-2.6922) -- cycle;
\draw[fill=black] (1.1538,-3.0768) -- (1.5384,-3.0768) -- (1.5384,-2.6922) -- (1.1538,-2.6922) -- cycle;
\draw[fill=black] (1.5384,-3.0768) -- (1.923,-3.0768) -- (1.923,-2.6922) -- (1.5384,-2.6922) -- cycle;
\draw[fill=black] (1.923,-3.0768) -- (2.3076,-3.0768) -- (2.3076,-2.6922) -- (1.923,-2.6922) -- cycle;
\draw[fill=black] (2.3076,-3.0768) -- (2.6922,-3.0768) -- (2.6922,-2.6922) -- (2.3076,-2.6922) -- cycle;
\draw[fill=white] (2.6922,-3.0768) -- (3.0768,-3.0768) -- (3.0768,-2.6922) -- (2.6922,-2.6922) -- 
cycle;
\draw[fill=white] (3.0768,-3.0768) -- (3.4614,-3.0768) -- (3.4614,-2.6922) -- (3.0768,-2.6922) -- cycle;
\draw[fill=black] (3.4614,-3.0768) -- (3.846,-3.0768) -- (3.846,-2.6922) -- (3.4614,-2.6922) -- cycle;
\draw[fill=white] (3.846,-3.0768) -- (4.2306,-3.0768) -- (4.2306,-2.6922) -- (3.846,-2.6922) -- cycle;
\draw[fill=white] (4.2306,-3.0768) -- (4.6152,-3.0768) -- (4.6152,-2.6922) -- (4.2306,-2.6922) -- cycle;
\draw[fill=white] (4.6152,-3.0768) -- (5,-3.0768) -- (5,-2.6922) -- (4.6152,-2.6922) -- cycle;

\draw[fill=black] (0,-3.4614) -- (0.3846,-3.4614) -- (0.3846,-3.0768) -- (0,-3.0768) -- cycle;
\draw[fill=white] (0.3846,-3.4614) -- (0.7692,-3.4614) -- (0.7692,-3.0768) -- (0.3846,-3.0768) -- cycle;
\draw[fill=white] (0.7692,-3.4614) -- (1.1538,-3.4614) -- (1.1538,-3.0768) -- (0.7692,-3.0768) -- cycle;
\draw[fill=black] (1.1538,-3.4614) -- (1.5384,-3.4614) -- (1.5384,-3.0768) -- (1.1538,-3.0768) -- cycle;
\draw[fill=white] (1.5384,-3.4614) -- (1.923,-3.4614) -- (1.923,-3.0768) -- (1.5384,-3.0768) -- cycle;
\draw[fill=black] (1.923,-3.4614) -- (2.3076,-3.4614) -- (2.3076,-3.0768) -- (1.923,-3.0768) -- cycle;
\draw[fill=black] (2.3076,-3.4614) -- (2.6922,-3.4614) -- (2.6922,-3.0768) -- (2.3076,-3.0768) -- cycle;
\draw[fill=white] (2.6922,-3.4614) -- (3.0768,-3.4614) -- (3.0768,-3.0768) -- (2.6922,-3.0768) -- 
cycle;
\draw[fill=white] (3.0768,-3.4614) -- (3.4614,-3.4614) -- (3.4614,-3.0768) -- (3.0768,-3.0768) -- cycle;
\draw[fill=white] (3.4614,-3.4614) -- (3.846,-3.4614) -- (3.846,-3.0768) -- (3.4614,-3.0768) -- cycle;
\draw[fill=black] (3.846,-3.4614) -- (4.2306,-3.4614) -- (4.2306,-3.0768) -- (3.846,-3.0768) -- cycle;
\draw[fill=white] (4.2306,-3.4614) -- (4.6152,-3.4614) -- (4.6152,-3.0768) -- (4.2306,-3.0768) -- cycle;
\draw[fill=white] (4.6152,-3.4614) -- (5,-3.4614) -- (5,-3.0768) -- (4.6152,-3.0768) -- cycle;

\draw[fill=black] (0,-3.846) -- (0.3846,-3.846) -- (0.3846,-3.4614) -- (0,-3.4614) -- cycle;
\draw[fill=white] (0.3846,-3.846) -- (0.7692,-3.846) -- (0.7692,-3.4614) -- (0.3846,-3.4614) -- cycle;
\draw[fill=black] (0.7692,-3.846) -- (1.1538,-3.846) -- (1.1538,-3.4614) -- (0.7692,-3.4614) -- cycle;
\draw[fill=white] (1.1538,-3.846) -- (1.5384,-3.846) -- (1.5384,-3.4614) -- (1.1538,-3.4614) -- cycle;
\draw[fill=black] (1.5384,-3.846) -- (1.923,-3.846) -- (1.923,-3.4614) -- (1.5384,-3.4614) -- cycle;
\draw[fill=black] (1.923,-3.846) -- (2.3076,-3.846) -- (2.3076,-3.4614) -- (1.923,-3.4614) -- cycle;
\draw[fill=white] (2.3076,-3.846) -- (2.6922,-3.846) -- (2.6922,-3.4614) -- (2.3076,-3.4614) -- cycle;
\draw[fill=white] (2.6922,-3.846) -- (3.0768,-3.846) -- (3.0768,-3.4614) -- (2.6922,-3.4614) -- 
cycle;
\draw[fill=white] (3.0768,-3.846) -- (3.4614,-3.846) -- (3.4614,-3.4614) -- (3.0768,-3.4614) -- cycle;
\draw[fill=white] (3.4614,-3.846) -- (3.846,-3.846) -- (3.846,-3.4614) -- (3.4614,-3.4614) -- cycle;
\draw[fill=white] (3.846,-3.846) -- (4.2306,-3.846) -- (4.2306,-3.4614) -- (3.846,-3.4614) -- cycle;
\draw[fill=black] (4.2306,-3.846) -- (4.6152,-3.846) -- (4.6152,-3.4614) -- (4.2306,-3.4614) -- cycle;
\draw[fill=white] (4.6152,-3.846) -- (5,-3.846) -- (5,-3.4614) -- (4.6152,-3.4614) -- cycle;

\draw[fill=black] (0,-4.2306) -- (0.3846,-4.2306) -- (0.3846,-3.846) -- (0,-3.846) -- cycle;
\draw[fill=black] (0.3846,-4.2306) -- (0.7692,-4.2306) -- (0.7692,-3.846) -- (0.3846,-3.846) -- cycle;
\draw[fill=black] (0.7692,-4.2306) -- (1.1538,-4.2306) -- (1.1538,-3.846) -- (0.7692,-3.846) -- cycle;
\draw[fill=black] (1.1538,-4.2306) -- (1.5384,-4.2306) -- (1.5384,-3.846) -- (1.1538,-3.846) -- cycle;
\draw[fill=white] (1.5384,-4.2306) -- (1.923,-4.2306) -- (1.923,-3.846) -- (1.5384,-3.846) -- cycle;
\draw[fill=white] (1.923,-4.2306) -- (2.3076,-4.2306) -- (2.3076,-3.846) -- (1.923,-3.846) -- cycle;
\draw[fill=black] (2.3076,-4.2306) -- (2.6922,-4.2306) -- (2.6922,-3.846) -- (2.3076,-3.846) -- cycle;
\draw[fill=black] (2.6922,-4.2306) -- (3.0768,-4.2306) -- (3.0768,-3.846) -- (2.6922,-3.846) -- 
cycle;
\draw[fill=white] (3.0768,-4.2306) -- (3.4614,-4.2306) -- (3.4614,-3.846) -- (3.0768,-3.846) -- cycle;
\draw[fill=white] (3.4614,-4.2306) -- (3.846,-4.2306) -- (3.846,-3.846) -- (3.4614,-3.846) -- cycle;
\draw[fill=white] (3.846,-4.2306) -- (4.2306,-4.2306) -- (4.2306,-3.846) -- (3.846,-3.846) -- cycle;
\draw[fill=white] (4.2306,-4.2306) -- (4.6152,-4.2306) -- (4.6152,-3.846) -- (4.2306,-3.846) -- cycle;
\draw[fill=black] (4.6152,-4.2306) -- (5,-4.2306) -- (5,-3.846) -- (4.6152,-3.846) -- cycle;

\draw (-1.2,0.8) node[anchor=north west] {$\varepsilon$};
\draw (-1.2,0.5) node[anchor=north west] {$1$};
\draw (-1.2,0.1) node[anchor=north west] {$10$};
\draw (-1.2,-0.3) node[anchor=north west] {$100$};
\draw (-1.2,-0.7) node[anchor=north west] {$101$};
\draw (-1.2,-1.1) node[anchor=north west] {$1000$};
\draw (-1.2,-1.5) node[anchor=north west] {$1001$};
\draw (-1.2,-1.9) node[anchor=north west] {$1010$};
\draw (-1.2,-2.3) node[anchor=north west] {$10000$};
\draw (-1.2,-2.7) node[anchor=north west] {$10001$};
\draw (-1.2,-3.1) node[anchor=north west] {$10010$};
\draw (-1.2,-3.5) node[anchor=north west] {$10100$};
\draw (-1.2,-3.9) node[anchor=north west] {$10101$};

\draw (-0.1,1) node[anchor=north west, rotate=40] {$\varepsilon$};
\draw (0.2,1) node[anchor=north west, rotate=40] {$1$};
\draw (0.5,1) node[anchor=north west, rotate=40] {$10$};
\draw (0.9,1) node[anchor=north west, rotate=40] {$100$};
\draw (1.3,1) node[anchor=north west, rotate=40] {$101$};
\draw (1.7,1) node[anchor=north west, rotate=40] {$1000$};
\draw (2.1,1) node[anchor=north west, rotate=40] {$1001$};
\draw (2.5,1) node[anchor=north west, rotate=40] {$1010$};
\draw (2.9,1) node[anchor=north west, rotate=40] {$10000$};
\draw (3.3,1) node[anchor=north west, rotate=40] {$10001$};
\draw (3.7,1) node[anchor=north west, rotate=40] {$10010$};
\draw (4.1,1) node[anchor=north west, rotate=40] {$10100$};
\draw (4.5,1) node[anchor=north west, rotate=40] {$10101$};
\end{tikzpicture}
\end{center}
\caption{The set $\mathcal{U}^\varphi_5$.}
\end{subfigure}
\caption{The sets $\mathcal{U}^\varphi_3$, $\mathcal{U}^\varphi_4$ and $\mathcal{U}^\varphi_5$ when $\beta=\varphi$ is the golden ratio.}
\label{fig:exemples-U-1}
\end{figure}
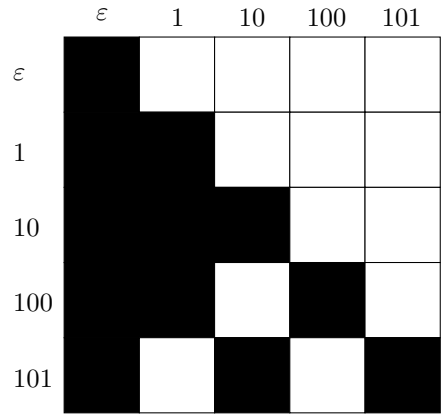
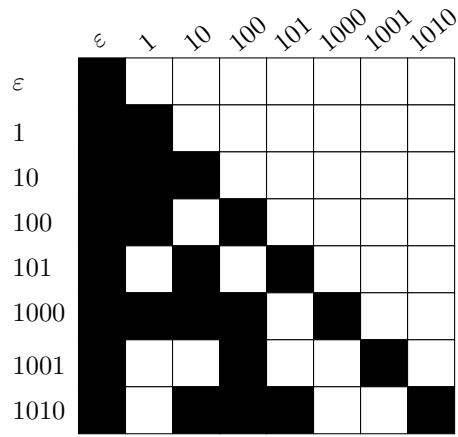
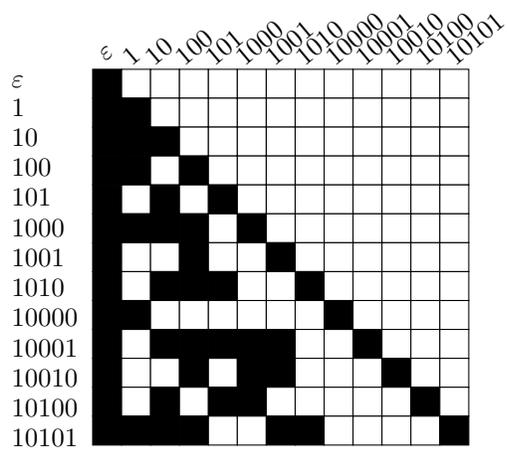
\end{center}
\end{example}

\begin{remark}\label{rem:pixel} 
Each pair $(u,v)$ of words of length at most $n$ with an odd binomial coefficient gives rise to a square region in $\U_n$. More precisely, we have the following situation. 
Let $n\ge 0$ and $u,v\in \Lang$ such that $0 \le |v|\le |u| \le n$ and $\binom{u}{v}\equiv 1\bmod{2}$. 
We have
$$((\val_{U_\beta}(v),\val_{U_\beta}(u)) + Q )/ U_\beta(n) \subset \U_n$$
as depicted in Figure~\ref{fig:vizu}.
\begin{figure}[h!tbp]
    \centering
\begin{tikzpicture}[line cap=round,line join=round,>=triangle 45,x=1.0cm,y=1.0cm, scale=.6]
\fill[fill=black,fill opacity=0.8] (-1,0) -- (-1,-1) -- (0,-1) -- (0,0) -- cycle;
\draw [->] (-3,5) -- (-3,-4);
\draw [->] (-3,5) -- (6,5);
\draw (-3.94,-3.66) node[anchor=north west] {$u$};
\draw (6,5.5) node[anchor=north west] {$v$};
\draw (-3.49,6) node[anchor=north west] {$0$};
\draw (-3.59,-2.7) node[anchor=north west] {$1$};
\draw (4.7,6) node[anchor=north west] {$1$};
\draw [dash pattern=on 5pt off 5pt] (-3,-3)-- (5,-3);
\draw [dash pattern=on 5pt off 5pt] (5,-3)-- (5,5);
\draw (-1,0)-- (-1,-1);
\draw (-1,-1)-- (0,-1);
\draw (0,-1)-- (0,0);
\draw (0,0)-- (-1,0);
\draw [dash pattern=on 5pt off 5pt] (-3,0)-- (-1,0);
\draw [dash pattern=on 5pt off 5pt] (-1,0)-- (-1,5);
\draw (-2,6.3) node[anchor=north west] {$\frac{\val_{U_\beta}(v)}{U_\beta(n)}$};
\draw (-5.6,0.75) node[anchor=north west] {$\frac{\val_{U_\beta}(u)}{U_\beta(n)}$};
\end{tikzpicture}
\caption{Visualization of a square region in $\U_n$.}
    \label{fig:vizu}
\end{figure}
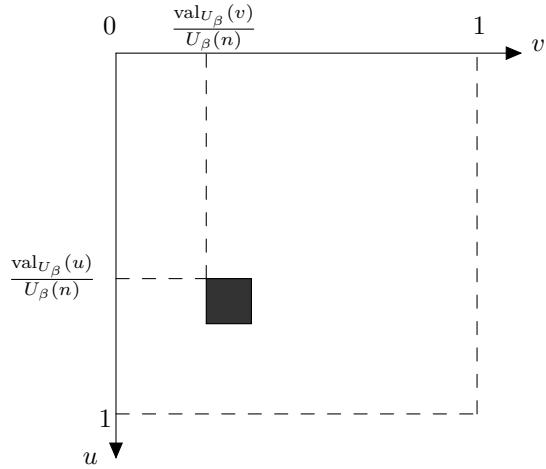
\end{remark}

We consider the space $(\mathcal{H}(\mathbb{R}^2),d_h)$ of the non-empty compact subsets of $\mathbb{R}^2$ equipped with the Hausdorff metric $d_h$ induced by the Euclidean distance $d$ on $\mathbb{R}^2$. 
It is well known that $(\mathcal{H}(\mathbb{R}^2),d_h)$ is complete~\cite{Falconer}. 
We let $B(x,\epsilon)$ denote the open ball of radius $\epsilon\ge 0$ centered at $x\in\mathbb{R}^2$ and, if $S\subset\mathbb{R}^2$, we let 
$$[S]_\epsilon:=\bigcup_{x\in S}B(x,\epsilon)$$
denote \emph{$\epsilon$-fattening} of $S$.

Our aim is to show that the sequence $(\U_n)_{n\ge 0}$ of compact subsets of $[0,1]\times[0,1]$ is converging and to provide an elementary description of its limit set. 
The idea is the following one. 
Let $(u,v)\in \Lang \times \Lang$ be a pair of words having an odd binomial coefficient. 
On the one hand, some of those pairs are such that $\binom{ua}{va} \equiv 0 \bmod{2}$ for all letters $a$ such that $ua,va \in \Lang$. 
In other words, those pairs of words create a black square region in $\U_{|u|}$ while the corresponding square region in $\U_{|u|+1}$ is white. 
As an example, take $\beta=\varphi$, $u=1010$ and $v=101$. 
We have $\binom{u0}{v0}=2$ (see Figure~\ref{fig:exemples-U-1}). 
On the other hand, some of those pairs create a more stable pattern, i.e., $\binom{uw}{vw} \equiv 1 \bmod{2}$ for all words $w$ such that $uw, vw \in \Lang$. 
Roughly, those pairs create a diagonal of square regions in $(\U_n)_{n\ge 0}$.
For instance, take $\beta=\varphi$, $u=101$ and $v=10$. In this case, $\binom{uw}{vw}\equiv 1 \bmod{2}$ for all admissible words $w$. In particular, the pairs of words $(u,v)$, $(u0,v0)$ and $(u00,v00)$, $(u01,v01)$ have odd binomial coefficients (see Figure~\ref{fig:exemples-U-1}) and create a diagonal of square regions. 
With the second type of pairs of words, we define a new sequence of compact subsets $(\AEns_n)_{n\ge 0}$ of $[0,1]\times[0,1]$ which converges to some well-defined limit set $\Llim$. 
Then, we show that the first sequence of compact sets $(\U_n)_{n\ge 0}$ also converges to this limit set. 
The remaining of this paper is dedicated to formalize and prove those statements. 
 
To reach that goal, for all non-empty words $u,v \in \Lang$, we first define the least integer $p$ such that $u0^pw, v0^pw$ belong to $\Lang$ for all words $w \in 0^*\Lang$.
In other terms, any word $w$ can be read after $u0^p$ and $v0^p$ in the automaton $\Aut$. 
Then, some pairs of words $(u,v)\in \Lang\times \Lang$ have the property that not only $\binom{u}{v}\equiv 1\bmod{2}$ but also $\binom{u0^pw}{v0^pw}\equiv 1\bmod{2}$ for all words $w\in 0^*\Lang$; see Corollary~\ref{cor:completion}. 
Such a property creates a particular pattern occurring in $\U_n$ for all sufficiently large $n$, as shown in Remark~\ref{rem:conv-diag}.

\begin{prop}\label{pro:def-p}
For all non-empty words $u,v\in \Lang$, there exists a smallest nonnegative integer $p(u,v)$ such that 
$$
(u0^{p(u,v)})^{-1} . \Lang = (v0^{p(u,v)})^{-1} . \Lang  = 0^* \Lang.
$$
\end{prop}
\begin{proof}
Using Proposition~\ref{pro:auto}, take $p(u,v)$ to be the least nonnegative integer $p$ such that $\delta( a_0, u0^{p}) = a_0 = \delta( a_0, v0^{p})$. 
Then, for any word $w \in 0^* \Lang$, the words $u0^{p(u,v)}w, v0^{p(u,v)}w$ are labels of paths in $\Aut$. Consequently, they are words in $\Lang$.
Conversely, if the words $u0^{p(u,v)}w, v0^{p(u,v)}w$ are labels of paths in $\Aut$, then $w \in 0^* \Lang$.
\end{proof}

In the following, we will be using $p(\varepsilon,\varepsilon)$. 
Observe that, using Proposition~\ref{pro:auto}, $\delta(a_0, \varepsilon) = a_0$. 
We naturally set $p(\varepsilon,\varepsilon):=0$ and we thus have $(\varepsilon0^{p(\varepsilon,\varepsilon)})^{-1} . \Lang  = \Lang$.

\begin{example}
If $\beta > 1$ is an integer, then $p(u,v)=0$ for all $u,v\in \Lang$. See Example~\ref{ex:integer}.
\end{example}

\begin{example}\label{ex:p-Fib}
If $\beta =\varphi$ is the golden ratio, then $p(u,v)=0$ if and only if $u$ and $v$ end with $0$ or $u=v=\varepsilon$, otherwise $p(u,v)=1$.
\end{example}

The integer of Proposition~\ref{pro:def-p} can be greater than $1$ as illustrated in the following example.

\begin{example}\label{ex:p-beta-plusgrand2}
Let $\beta$ be the dominant root of the polynomial $P(X)=X^4-2X^3-X^2-1$. 
Then $\beta \approx 2.47098$ is a Parry number with $d_\beta(1)=2101$ and $d_\beta^*(1)=(2100)^\omega$. 
The automaton $\Aut$ is depicted in Figure~\ref{fig:Aut-beta-ex-1}. 
For instance, $p(101,21)=2$.
\begin{figure}
\centering
\begin{tikzpicture}
\tikzstyle{every node}=[shape=circle,fill=none,draw=black,minimum size=20pt,inner sep=2pt]
\node[accepting](a) at (0,0) {$a_0$};
\node[accepting](b) at (2,0) {$a_1$};
\node[accepting](c) at (4,0) {$a_2$};
\node[accepting](d) at (6,0) {$a_3$};

\tikzstyle{every node}=[shape=circle,fill=none,draw=none,minimum size=10pt,inner sep=2pt]
\node(a0) at (-1,0) {};
\node(a1) at (0,1) {$0,1$};

\tikzstyle{every path}=[color =black, line width = 0.5 pt]
\tikzstyle{every node}=[shape=circle,minimum size=5pt,inner sep=2pt]
\draw [->] (a0) to [] node [] {}  (a);
\draw [->] (a) to [loop above] node [] {}  (a);

\draw [->] (a) to [bend left] node [above] {$2$}  (b);
\draw [->] (b) to [bend left] node [below] {$0$}  (a);
\draw [->] (b) to [] node [above] {$1$}  (c);
\draw [->] (c) to [] node [above] {$0$}  (d);

\draw [->] (d) to [bend right=50] node [above] {$0$}  (a);

;
\end{tikzpicture}
\caption{The automaton $\Aut$ for the dominant root $\beta$ of the polynomial $P(X)=X^4-2X^3-X^2-1$.}
\label{fig:Aut-beta-ex-1}
\end{figure}
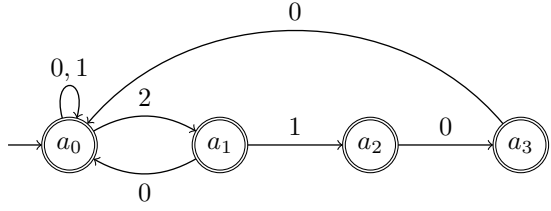
\end{example}

\begin{definition}\label{def:cond*}
Let $(u,v)\in \Lang\times \Lang$. 
We say that $(u,v)$ satisfies the $(\star)$ condition if either $u=v=\varepsilon$, or $|u|\ge|v|>0$ and
$$\binom{u0^{p(u,v)}}{v0^{p(u,v)}}\equiv 1\bmod{2} \quad \text{and} \quad \binom{u0^{p(u,v)}}{v0^{p(u,v)}a} = 0 \quad \forall \, a \in \Alp$$
where $p(u,v)$ is defined by Proposition~\ref{pro:def-p}. 
Observe that, if $(u,v)\neq (\varepsilon, \varepsilon)$, then $v0^{p(u,v)}a\in \Lang$ for all $a \in \Alp$.
\end{definition}

\begin{remark}
Observe that if only one of the two words $u$ or $v$ is empty, then the pair $(u,v)$ never satisfies $(\star)$.
\end{remark}

The next lemma shows that all diagonal elements of $\U_n$ satisfy $(\star)$.

\begin{lemma}\label{lem:u-u}
For any word $u\in \Lang$, the pair $(u,u)$ satisfies $(\star)$.
\end{lemma}
\begin{proof}
If $u=\varepsilon$, the result is clear using Definition~\ref{def:cond*}.
Suppose $u$ is non-empty and let $p:=p(u,u)$ denote the integer from Proposition~\ref{pro:def-p}. 
Then, for all $\, a\in \Alp$, we have
$$\binom{u0^{p}}{u0^{p}}=1 \equiv 1\bmod{2} \quad \text{and} \quad \binom{u0^{p}}{u0^{p}a} = 0$$
since $|u0^{p}a| > |u0^{p}|$.
\end{proof}

If a pair of words satisfies $(\star)$, it has the following two properties. 
First, its binomial coefficient is odd, as stated in the following proposition. 
Secondly, it creates a special pattern in $\U_n$ for all large enough $n$; see Proposition~\ref{pro:completion}, Corollary~\ref{cor:completion} and Remark~\ref{rem:conv-diag}.

\begin{prop}\label{pro:binom-equal-1}
Let $(u,v) \in \Lang\times \Lang$ satisfying $(\star)$. 
Then 
$$
\binom{u}{v} \equiv 1\bmod{2}. 
$$
\end{prop}
\begin{proof}
If $u=v=\varepsilon$, the result is clear by definition. 
Suppose that $u$ and $v$ are non-empty. 
Let us proceed by contradiction and suppose that $\binom{u}{v}$ is even. 
Let us set $p:=p(u,v)$ from Proposition~\ref{pro:def-p}. 
On the one hand, by Definition~\ref{def:cond*}, we know that
$$
\binom{u0^p}{v0^p} \equiv 1\bmod{2}
$$
and, on the other hand, Lemma~\ref{lem:bin-zeroes} states that
$$
\binom{u0^p}{v0^p} = \sum_{j=1}^p \binom{p}{j} \binom{u}{v0^j} + \binom{u}{v}.
$$
Consequently, we have
$$
\sum_{j=1}^p \binom{p}{j} \binom{u}{v0^j} \equiv 1\bmod{2} > 0 
$$
and there must exist $i\in\{1,\ldots,p\}$ such that $\binom{u}{v0^i}>0$.  
Using again Lemma~\ref{lem:bin-zeroes}, we also have
$$
\binom{u0^p}{v0^p0} = \sum_{j=1}^{p+1} \binom{p}{j-1} \binom{u}{v0^j} \ge \binom{p}{i-1} \binom{u}{v0^i} > 0,
$$
which contradicts Definition~\ref{def:cond*}.
\end{proof}

\begin{prop}\label{pro:completion}
Let $u,v \in \Lang$ be two non-empty words such that $(u,v)$ satisfies $(\star)$. 
For any letter $a\in \Alp$, the pair of words $( u0^{p(u,v)}a, v0^{p(u,v)}a) \in \Lang \times \Lang$ satisfies $(\star)$.
\end{prop}
\begin{proof}
For the sake of clarity, set $p:=p(u,v)$.
Let $a$ be a letter in $\Alp$ and also set $p':= p(u0^{p}a,v0^{p}a)$.
By Lemma~\ref{lem:lothaire-bin} and Lemma~\ref{lem:bin-zeroes}, 
$$
\binom{u0^pa0^{p'}}{v0^pa0^{p'}} = \sum_{j=1}^{p'} \binom{p'}{j} \binom{u0^pa}{v0^pa0^{j}} + \binom{u0^p}{v0^pa} + \binom{u0^p}{v0^p}.
$$
Since $(u,v)$ satisfies $(\star)$, all the coefficients $\binom{u0^pa}{v0^pa0^{j}}$, for $j=1,\ldots,p'$, and $\binom{u0^p}{v0^pa}$ are equal to $0$. 
Otherwise, it means that the word $v0^pa$ appears as a subword of the word $u0^p$, which contradicts $(\star)$. 
Consequently, using Definition~\ref{def:cond*}, we get
$$
\binom{u0^pa0^{p'}}{v0^pa0^{p'}} = \binom{u0^p}{v0^p} \equiv 1\bmod{2}.
$$
Using the same argument, for any letter $b\in \Alp$, we have 
$$
\binom{u0^pa0^{p'}}{v0^pa0^{p'}b} = 0.
$$
\end{proof}

The next corollary extends Lemma~\ref{lem:u-u} when $(u,v)\neq (\varepsilon,\varepsilon)$.
Indeed, recall that $p(\varepsilon,\varepsilon)=0$.

\begin{corollary}\label{cor:completion}
Let $u,v \in \Lang$ be two non-empty words such that $(u,v)$ satisfies $(\star)$. 
Then
$$
\binom{u0^{p(u,v)} w}{v0^{p(u,v)} w} \equiv 1\bmod{2} \quad \forall \, w\in 0^*\Lang.
$$
\end{corollary}
\begin{proof}
Set $p:=p(u,v)$. 
From Proposition~\ref{pro:def-p}, $u0^{p}w, v0^{p}w$ belong to $\Lang$ for any word $w \in 0^*\Lang$. 
Now proceed by induction on the length of $w \in 0^* \Lang$. 
If $|w|=0$, then $w=\varepsilon$ is the empty word and the statement is true using Definition~\ref{def:cond*}.
If $|w|=1$, then $w=a$ is a letter belonging to $\Alp$. 
Then, by Proposition~\ref{pro:completion}, we know that $(u0^pa,v0^pa)$ satisfies $(\star)$. 
Using Proposition~\ref{pro:binom-equal-1}, we have
$$
\binom{u0^pa}{v0^pa} \equiv 1\bmod{2}.
$$
Now suppose that $|w|\ge 2$ and write $w=aw'b \in 0^* \Lang$ where $a,b$ are letters.
From Lemma~\ref{lem:lothaire-bin}, we deduce that
$$
\binom{u0^pw}{v0^pw}
= \binom{u0^paw'}{v0^paw'b} + \binom{u0^paw'}{v0^paw'}.
$$
By induction hypothesis, $\binom{u0^paw'}{v0^paw'}\equiv 1\bmod{2}$ since $aw' \in 0^*\Lang$ and $|aw'| < |w|$. 
Furthermore, $\binom{u0^paw'}{v0^paw'b}$ must be $0$, otherwise it means that the word $v0^pa$ occurs as a subword of the word $u0^p$, which contradicts the fact that $(u,v)$ satisfies $(\star)$. 
This ends the proof. 
\end{proof}

The next lemma is useful to characterize the pattern created in $\U_n$, for all sufficiently large $n$, by pairs of words satisfying $(\star)$, see Remark~\ref{rem:conv-diag}.
In this result, we make use of the convention given in Definition~\ref{def:beta-exp}.

\begin{lemma}\label{lem:origine-extremite}
Let $(u,v) \in \Lang \times \Lang$ satisfying $(\star)$. 
\begin{itemize}
\item[(a)]  The sequence 
$$
\left(\left( \frac{\val_{U_\beta}(v0^{p(u,v)+n})}{U_\beta(|u|+p(u,v)+n)},\frac{\val_{U_\beta}(u0^{p(u,v)+n})}{U_\beta(|u|+p(u,v)+n)}\right) \right)_{n\ge 0}
$$
converges to the pair of real numbers $( 0.0^{|u|-|v|} v , 0.u )$.
\item[(b)] For all $n\ge 0$, let $w=d_n$ denotes the prefix of length $n$ of $d_\beta^*(1)$. Then the sequence  
$$
\left(\left( \frac{\val_{U_\beta}(v0^{p(u,v)} d_n)}{U_\beta(|u|+p(u,v)+n)},\frac{\val_{U_\beta}(u0^{p(u,v)} d_n)}{U_\beta(|u|+p(u,v)+n)}\right)\right)_{n\ge 0}
$$
converges to the pair of real numbers $(0.0^{|u|-|v|} v 0^{p(u,v)} d_\beta^*(1), 0.u 0^{p(u,v)} d_\beta^*(1))$.
\end{itemize}
\end{lemma}
\begin{proof}
Let $(u,v) \in \Lang \times \Lang$ satisfying $(\star)$ and set $p:=p(u,v)$.
We prove the first item as the proof of the second one is similar. 
The result is trivial if $u=v=\varepsilon$.
Suppose that $u$ and $v$ are non-empty words.
Let us write $u=u_{|u|-1} u_{|u|-2} \cdots u_{0}$ where $u_i\in \Alp$ for all $i$. 
By definition, we have
$$
\frac{\val_{U_\beta}(u0^{p+n})}{U_\beta(|u|+p+n)} = \sum_{i=0}^{|u|-1} u_i \, \frac{U_\beta(i+p+n)}{U_\beta(|u|+p+n)}.
$$
Using~\eqref{eq:conv-beta-U}, $U_\beta(i+p+n) / U_\beta(|u|+p+n)$ tends to $ \beta^i / \beta^{|u|}$ when $n$ tend to infinity. 
Consequently, 
$$
\lim\limits_{n \rightarrow +\infty} \frac{\val_{U_\beta}(u0^{p+n})}{U_\beta(|u|+p+n)} = \sum_{i=0}^{|u|-1} u_i \beta^{i-|u|} = 0.u.
$$
Using the same reasoning on the word $v$, we conclude that the sequence 
$$
\left(\left( \frac{\val_{U_\beta}(v0^{p(u,v)+n})}{U_\beta(|u|+p(u,v)+n)},\frac{\val_{U_\beta}(u0^{p(u,v)+n})}{U_\beta(|u|+p(u,v)+n)}\right) \right)_{n\ge 0}
$$
converges to the pair of real numbers $( 0.0^{|u|-|v|} v , 0.u )$.
\end{proof}

\begin{remark}\label{rem:conv-diag}
Let $(u,v) \in \Lang \times \Lang$ satisfying $(\star)$ and set $p:=p(u,v)$.
Suppose that $u$ and $v$ are non-empty (the case when $u=v=\varepsilon$ is similar: in the following, replace $0^*\Lang$ by $\Lang$ where needed).
Using Corollary~\ref{cor:completion}, the pair of words $(u0^{p}w, v0^{p}w)$ has an odd binomial coefficient for any word $w\in 0^*\Lang$.
In particular, the pair of words $(u0^pw, v0^pw)$ corresponds to a square region in $\U_{|u|+p+n}$ for all $w\in 0^*\Lang$ such that $|w|=n\ge 0$. 
Using Remark~\ref{rem:pixel}, this region is
$$\left( \frac{\val_{U_\beta}(v0^{p}w)}{U_\beta(|u|+p+n)},\frac{\val_{U_\beta}(u0^{p}w)}{U_\beta(|u|+p+n)}\right) + \frac{Q}{U_\beta(|u|+p+n)} \subset \U_{|u|+p+n}.$$
Using Lemma~\ref{lem:origine-extremite}, when $w=0^n$ (the smallest word of length $n$ in $0^*\Lang$), the sequence 
$$
\left(\left( \frac{\val_{U_\beta}(v0^{p+n})}{U_\beta(|u|+p+n)},\frac{\val_{U_\beta}(u0^{p+n})}{U_\beta(|u|+p+n)}\right) \right)_{n\ge 0}
$$
converges to the pair of real numbers $( 0.0^{|u|-|v|} v , 0.u )$.
This point will be the first endpoint of a segment associated with $u$ and $v$. See Definition~\ref{def:seg}.
Analogously, using Lemma~\ref{lem:origine-extremite}, when $w=d_n$ is the prefix of length $n$ of $d_\beta^*(1)$ (the greatest word of length $n$ in $0^*\Lang$), then the sequence  
$$
\left(\left( \frac{\val_{U_\beta}(v0^p d_n)}{U_\beta(|u|+p+n)},\frac{\val_{U_\beta}(u0^p d_n)}{U_\beta(|u|+p+n)}\right)\right)_{n\ge 0}
$$
converges to the pair of real numbers $(0.0^{|u|-|v|} v 0^p d_\beta^*(1), 0.u 0^p d_\beta^*(1))$. This point will be the second endpoint of the same segment associated with $u$ and $v$. See Definition~\ref{def:seg}.
As a consequence, the sequence of sets whose $n$th term is defined by
\begin{align}
\bigcup_{\substack{|w| = n \\ w \in 0^*\Lang}} \left( \left( \frac{\val_{U_\beta}(v0^{p}w)}{U_\beta(|u|+p+n)},\frac{\val_{U_\beta}(u0^{p}w)}{U_\beta(|u|+p+n)}\right) + \frac{Q}{U_\beta(|u|+p+n)} \right) \label{eq:conv-diag}
\end{align}
converges, for the Hausdorff distance, to the diagonal of the square $(0.0^{|u|-|v|}v, 0.u) + Q/\beta^{|u|+p}$.
\end{remark}

\begin{example}
As a first example, when $\beta=2$, we find back the construction in~\cite{LRS1}.
As a second example, let us take $\beta=\varphi$ to be the golden ratio. 
Let $u=101$ and $v=10$ (resp., $u'=100=v'$). 
Then $p(u,v)=1$ (resp., $p(u',v')=0$); see Example~\ref{ex:p-Fib}.
Those pairs of words satisfy $(\star)$.  
The first few terms of the sequence of sets~\eqref{eq:conv-diag} are respectively depicted in Figure~\ref{fig:conv-diag1} and Figure~\ref{fig:conv-diag2}.
Observe that when $n$ tends to infinity, the union of black squares in $\mathcal{U}^\varphi_{n+4}$ (resp., $\mathcal{U}^\varphi_{n+3}$) converges to the diagonal of $(0.0v, 0.u) + Q/\varphi^{4}$ (resp., $(0.v', 0.u') + Q/\varphi^{3}$).
\begin{figure}[h!tbp]
\begin{subfigure}[b]{0.3\textwidth}
\begin{center}
\begin{tikzpicture}
\draw[fill=black] (0,0) -- (4,0) -- (4,4) -- (0,4) -- cycle;

\draw (-0.5,2.3) node[anchor=north west] {$u$};
\draw (1.7,4.5) node[anchor=north west] {$v$};
\draw (1.7,-0.5) node[anchor=north west] {$\mathcal{U}^\varphi_3$};
\end{tikzpicture}
\end{center}
\caption{A subset of $\mathcal{U}^\varphi_3$.}
\end{subfigure}
\begin{subfigure}[b]{0.3\textwidth}
\begin{center}
\begin{tikzpicture}
\draw[fill=black] (0,2) -- (2,2) -- (2,4) -- (0,4) -- cycle;

\draw (-0.7,3.3) node[anchor=north west] {$u0$};
\draw (0.7,4.5) node[anchor=north west] {$v0$};
\draw (0.7,-0.5) node[anchor=north west] {$\mathcal{U}^\varphi_4$};
\end{tikzpicture}
\end{center}
\caption{The element $n=0$ of~\eqref{eq:conv-diag}.}
\end{subfigure}
\begin{subfigure}[b]{0.3\textwidth}
\begin{center}
\begin{tikzpicture}
\draw[fill=black] (0,3) -- (1,3) -- (1,4) -- (0,4) -- cycle;

\draw[fill=black] (1,2) -- (2,2) -- (2,3) -- (1,3) -- cycle;

\draw[fill=white] (0,2) -- (1,2) -- (1,3) -- (0,3) -- cycle;

\draw[fill=white] (1,3) -- (2,3) -- (2,4) -- (1,4) -- cycle;

\draw (-0.8,3.7) node[anchor=north west] {$u00$};
\draw (-0.8,2.8) node[anchor=north west] {$u01$};
\draw (0,4.5) node[anchor=north west] {$v00$};
\draw (1,4.5) node[anchor=north west] {$v01$};
\draw (0.7,-0.5) node[anchor=north west] {$\mathcal{U}^\varphi_5$};
\end{tikzpicture}
\end{center}
\caption{The element $n=1$ of~\eqref{eq:conv-diag}.}
\end{subfigure} 

\begin{subfigure}[b]{0.3\textwidth}
\begin{center}
\begin{tikzpicture}
\draw[fill=black] (0,3.33) -- (0.66,3.33) -- (0.66,4) -- (0,4) -- cycle;
\draw[fill=white] (0.66,3.33) -- (1.33,3.33) -- (1.33,4) -- (0.66,4) -- cycle;
\draw[fill=white] (1.33,3.33) -- (2,3.33) -- (2,4) -- (1.33,4) -- cycle;

\draw[fill=white] (0,2.66) -- (0.66,2.66) -- (0.66,3.33) -- (0,3.33) -- cycle;
\draw[fill=black] (0.66,2.66) -- (1.33,2.66) -- (1.33,3.33) -- (0.66,3.33) -- cycle;
\draw[fill=white] (1.33,2.66) -- (2,2.66) -- (2,3.33) -- (1.33,3.33) -- cycle;

\draw[fill=white] (0,2) -- (0.66,2) -- (0.66,2.66) -- (0,2.66) -- cycle;
\draw[fill=white] (0.66,2) -- (1.33,2) -- (1.33,2.66) -- (0.66,2.66) -- cycle;
\draw[fill=black] (1.33,2) -- (2,2) -- (2,2.66) -- (1.33,2.66) -- cycle;

\draw (-1.1,3.9) node[anchor=north west] {$u000$};
\draw (-1.1,3.2) node[anchor=north west] {$u001$};
\draw (-1.1,2.6) node[anchor=north west] {$u010$};
\draw (-0.3,4.5) node[anchor=north west] {$v000$};
\draw (0.5,4.5) node[anchor=north west] {$v001$};
\draw (1.3,4.5) node[anchor=north west] {$v010$};
\draw (0.7,-0.5) node[anchor=north west] {$\mathcal{U}^\varphi_6$};
\end{tikzpicture}
\end{center}
\caption{The element $n=2$ of~\eqref{eq:conv-diag}.}
\end{subfigure}
\begin{subfigure}[b]{0.3\textwidth}
\begin{center}
\begin{tikzpicture}
\draw[fill=black] (0,3.6) -- (0.4,3.6) -- (0.4,4) -- (0,4) -- cycle;
\draw[fill=white] (0.4,3.6) -- (0.8,3.6) -- (0.8,4) -- (0.4,4) -- cycle;
\draw[fill=white] (0.8,3.6) -- (1.2,3.6) -- (1.2,4) -- (0.8,4) -- cycle;
\draw[fill=white] (1.2,3.6) -- (1.6,3.6) -- (1.6,4) -- (1.2,4) -- cycle;
\draw[fill=white] (1.6,3.6) -- (2,3.6) -- (2,4) -- (1.6,4) -- cycle;

\draw[fill=white] (0,3.2) -- (0.4,3.2) -- (0.4,3.6) -- (0,3.6) -- cycle;
\draw[fill=black] (0.4,3.2) -- (0.8,3.2) -- (0.8,3.6) -- (0.4,3.6) -- cycle;
\draw[fill=white] (0.8,3.2) -- (1.2,3.2) -- (1.2,3.6) -- (0.8,3.6) -- cycle;
\draw[fill=white] (1.2,3.2) -- (1.6,3.2) -- (1.6,3.6) -- (1.2,3.6) -- cycle;
\draw[fill=white] (1.6,3.2) -- (2,3.2) -- (2,3.6) -- (1.6,3.6) -- cycle;

\draw[fill=white] (0,2.8) -- (0.4,2.8) -- (0.4,3.2) -- (0,3.2) -- cycle;
\draw[fill=white] (0.4,2.8) -- (0.8,2.8) -- (0.8,3.2) -- (0.4,3.2) -- cycle;
\draw[fill=black] (0.8,2.8) -- (1.2,2.8) -- (1.2,3.2) -- (0.8,3.2) -- cycle;
\draw[fill=white] (1.2,2.8) -- (1.6,2.8) -- (1.6,3.2) -- (1.2,3.2) -- cycle;
\draw[fill=white] (1.6,2.8) -- (2,2.8) -- (2,3.2) -- (1.6,3.2) -- cycle;

\draw[fill=white] (0,2.4) -- (0.4,2.4) -- (0.4,2.8) -- (0,2.8) -- cycle;
\draw[fill=white] (0.4,2.4) -- (0.8,2.4) -- (0.8,2.8) -- (0.4,2.8) -- cycle;
\draw[fill=white] (0.8,2.4) -- (1.2,2.4) -- (1.2,2.8) -- (0.8,2.8) -- cycle;
\draw[fill=black] (1.2,2.4) -- (1.6,2.4) -- (1.6,2.8) -- (1.2,2.8) -- cycle;
\draw[fill=white] (1.6,2.4) -- (2,2.4) -- (2,2.8) -- (1.6,2.8) -- cycle;

\draw[fill=white] (0,2) -- (0.4,2) -- (0.4,2.4) -- (0,2.4) -- cycle;
\draw[fill=white] (0.4,2) -- (0.8,2) -- (0.8,2.4) -- (0.4,2.4) -- cycle;
\draw[fill=white] (0.8,2) -- (1.2,2) -- (1.2,2.4) -- (0.8,2.4) -- cycle;
\draw[fill=white] (1.2,2) -- (1.6,2) -- (1.6,2.4) -- (1.2,2.4) -- cycle;
\draw[fill=black] (1.6,2) -- (2,2) -- (2,2.4) -- (1.6,2.4) -- cycle;

\draw (-1.1,4) node[anchor=north west] {$u0000$};
\draw (-1.1,3.6) node[anchor=north west] {$u0001$};
\draw (-1.1,3.2) node[anchor=north west] {$u0010$};
\draw (-1.1,2.8) node[anchor=north west] {$u0100$};
\draw (-1.1,2.4) node[anchor=north west] {$u0101$};
\draw (-0.3,4.2) node[anchor=north west, rotate=40] {$v0000$};
\draw (0.2,4.2) node[anchor=north west, rotate=40] {$v0001$};
\draw (0.6,4.2) node[anchor=north west, rotate=40] {$v0010$};
\draw (1,4.2) node[anchor=north west, rotate=40] {$v0100$};
\draw (1.4,4.2) node[anchor=north west, rotate=40] {$v0101$};

\draw (0.7,-0.5) node[anchor=north west] {$\mathcal{U}^\varphi_7$};
\end{tikzpicture}
\end{center}
\caption{The element $n=3$ of~\eqref{eq:conv-diag}.}
\end{subfigure}
\caption{The first few terms of sequence of sets~\eqref{eq:conv-diag} converging to the diagonal of the square $(0.0v, 0.u) + Q/\varphi^{4}$ for $u=101$ and $v=10$.}
    \label{fig:conv-diag1}
\end{figure}

\begin{figure}[h!tbp]
\begin{subfigure}[b]{0.5\textwidth}
\begin{center}
\begin{tikzpicture}
\draw[fill=black] (0,0) -- (4,0) -- (4,4) -- (0,4) -- cycle;

\draw (-0.5,2.3) node[anchor=north west] {$u'$};
\draw (1.7,4.5) node[anchor=north west] {$v'$};
\draw (1.7,-0.5) node[anchor=north west] {$\mathcal{U}^\varphi_3$};
\end{tikzpicture}
\end{center}
\caption{The element $n=0$ of~\eqref{eq:conv-diag}.}
\end{subfigure}
\begin{subfigure}[b]{0.5\textwidth}
\begin{center}
\begin{tikzpicture}
\draw[fill=black] (0,2) -- (2,2) -- (2,4) -- (0,4) -- cycle;
\draw[fill=white] (2,2) -- (4,2) -- (4,4) -- (2,4) -- cycle;
\draw[fill=white] (0,0) -- (2,0) -- (2,2) -- (0,2) -- cycle;
\draw[fill=black] (2,0) -- (4,0) -- (4,2) -- (2,2) -- cycle;

\draw (-0.7,3.3) node[anchor=north west] {$u'0$};
\draw (-0.7,1.3) node[anchor=north west] {$u'1$};
\draw (0.7,4.5) node[anchor=north west] {$v'0$};
\draw (2.7,4.5) node[anchor=north west] {$v'1$};
\draw (1.8,-0.5) node[anchor=north west] {$\mathcal{U}^\varphi_4$};
\end{tikzpicture}
\end{center}
\caption{The element $n=1$ of~\eqref{eq:conv-diag}.}
\end{subfigure}
\begin{subfigure}[b]{0.5\textwidth}
\begin{center}
\begin{tikzpicture}
\draw[fill=black] (0,2.66) -- (1.33,2.66) -- (1.33,4) -- (0,4) -- cycle;
\draw[fill=white] (1.33,2.66) -- (2.66,2.66) -- (2.66,4) -- (1.33,4) -- cycle;
\draw[fill=white] (2.66,2.66) -- (4,2.66) -- (4,4) -- (2.66,4) -- cycle;

\draw[fill=white] (0,1.33) -- (1.33,1.33) -- (1.33,2.66) -- (0,2.66) -- cycle;
\draw[fill=black] (1.33,1.33) -- (2.66,1.33) -- (2.66,2.66) -- (1.33,2.66) -- cycle;
\draw[fill=white] (2.66,1.33) -- (4,1.33) -- (4,2.66) -- (2.66,2.66) -- cycle;

\draw[fill=white] (0,0) -- (1.33,0) -- (1.33,1.33) -- (0,1.33) -- cycle;
\draw[fill=white] (1.33,0) -- (2.66,0) -- (2.66,1.33) -- (1.33,1.33) -- cycle;
\draw[fill=black] (2.66,0) -- (4,0) -- (4,1.33) -- (2.66,1.33) -- cycle;

\draw (-0.8,3.7) node[anchor=north west] {$u'00$};
\draw (-0.8,2.3) node[anchor=north west] {$u'01$};
\draw (-0.8,0.9) node[anchor=north west] {$u'10$};
\draw (0,4.5) node[anchor=north west] {$v'00$};
\draw (1.5,4.5) node[anchor=north west] {$v'01$};
\draw (3,4.5) node[anchor=north west] {$v'10$};
\draw (1.8,-0.5) node[anchor=north west] {$\mathcal{U}^\varphi_5$};
\end{tikzpicture}
\end{center}
\caption{The element $n=2$ of~\eqref{eq:conv-diag}.}
\end{subfigure} 
\begin{subfigure}[b]{0.5\textwidth}
\begin{center}
\begin{tikzpicture}
\draw[fill=black] (0,3.2) -- (0.8,3.2) -- (0.8,4) -- (0,4) -- cycle;
\draw[fill=white] (0.8,3.2) -- (1.6,3.2) -- (1.6,4) -- (0.8,4) -- cycle;
\draw[fill=white] (1.6,3.2) -- (2.4,3.2) -- (2.4,4) -- (1.6,4) -- cycle;
\draw[fill=white] (2.4,3.2) -- (3.2,3.2) -- (3.2,4) -- (2.4,4) -- cycle;
\draw[fill=white] (3.2,3.2) -- (4,3.2) -- (4,4) -- (3.2,4) -- cycle;

\draw[fill=white] (0,2.4) -- (0.8,2.4) -- (0.8,3.2) -- (0,3.2) -- cycle;
\draw[fill=black] (0.8,2.4) -- (1.6,2.4) -- (1.6,3.2) -- (0.8,3.2) -- cycle;
\draw[fill=white] (1.6,2.4) -- (2.4,2.4) -- (2.4,3.2) -- (1.6,3.2) -- cycle;
\draw[fill=white] (2.4,2.4) -- (3.2,2.4) -- (3.2,3.2) -- (2.4,3.2) -- cycle;
\draw[fill=white] (3.2,2.4) -- (4,2.4) -- (4,3.2) -- (3.2,3.2) -- cycle;

\draw[fill=white] (0,1.6) -- (0.8,1.6) -- (0.8,2.4) -- (0,2.4) -- cycle;
\draw[fill=white] (0.8,1.6) -- (1.6,1.6) -- (1.6,2.4) -- (0.8,2.4) -- cycle;
\draw[fill=black] (1.6,1.6) -- (2.4,1.6) -- (2.4,2.4) -- (1.6,2.4) -- cycle;
\draw[fill=white] (2.4,1.6) -- (3.2,1.6) -- (3.2,2.4) -- (2.4,2.4) -- cycle;
\draw[fill=white] (3.2,1.6) -- (4,1.6) -- (4,2.4) -- (3.2,2.4) -- cycle;

\draw[fill=white] (0,0.8) -- (0.8,0.8) -- (0.8,1.6) -- (0,1.6) -- cycle;
\draw[fill=white] (0.8,0.8) -- (1.6,0.8) -- (1.6,1.6) -- (0.8,1.6) -- cycle;
\draw[fill=white] (1.6,0.8) -- (2.4,0.8) -- (2.4,1.6) -- (1.6,1.6) -- cycle;
\draw[fill=black] (2.4,0.8) -- (3.2,0.8) -- (3.2,1.6) -- (2.4,1.6) -- cycle;
\draw[fill=white] (3.2,0.8) -- (4,0.8) -- (4,1.6) -- (3.2,1.6) -- cycle;

\draw[fill=white] (0,0) -- (0.8,0) -- (0.8,0.8) -- (0,0.8) -- cycle;
\draw[fill=white] (0.8,0) -- (1.6,0) -- (1.6,0.8) -- (0.8,0.8) -- cycle;
\draw[fill=white] (1.6,0) -- (2.4,0) -- (2.4,0.8) -- (1.6,0.8) -- cycle;
\draw[fill=white] (2.4,0) -- (3.2,0) -- (3.2,0.8) -- (2.4,0.8) -- cycle;
\draw[fill=black] (3.2,0) -- (4,0) -- (4,0.8) -- (3.2,0.8) -- cycle;

\draw (-1.1,3.9) node[anchor=north west] {$u'000$};
\draw (-1.1,3.1) node[anchor=north west] {$u'001$};
\draw (-1.1,2.3) node[anchor=north west] {$u'010$};
\draw (-1.1,1.5) node[anchor=north west] {$u'100$};
\draw (-1.1,0.7) node[anchor=north west] {$u'101$};

\draw (-0,4.3) node[anchor=north west,rotate=40] {$v'000$};
\draw (0.8,4.3) node[anchor=north west,rotate=40] {$v'001$};
\draw (1.6,4.3) node[anchor=north west,rotate=40] {$v'010$};
\draw (2.4,4.3) node[anchor=north west,rotate=40] {$v'100$};
\draw (3.2,4.3) node[anchor=north west,rotate=40] {$v'101$};

\draw (1.8,-0.5) node[anchor=north west,rotate=40] {$\mathcal{U}^\varphi_6$};
\end{tikzpicture}
\end{center}
\caption{The element $n=3$ of~\eqref{eq:conv-diag}.}
\end{subfigure}
\caption{The first few terms of sequence of sets~\eqref{eq:conv-diag} converging to the diagonal of the square $(0.v', 0.u') + Q/\varphi^{3}$ for $u'=100$ and $v'=100$.}
    \label{fig:conv-diag2}
\end{figure}
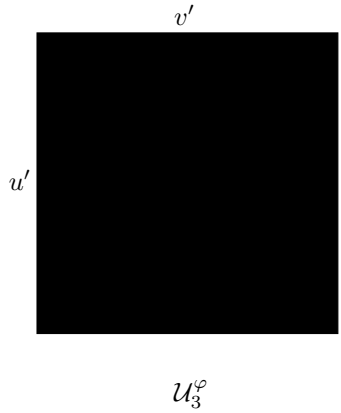
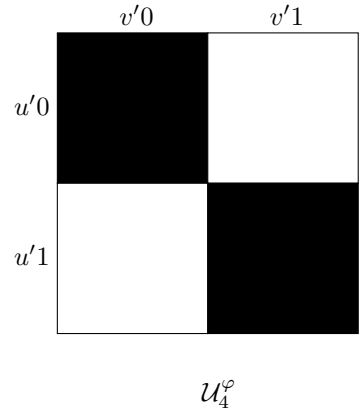
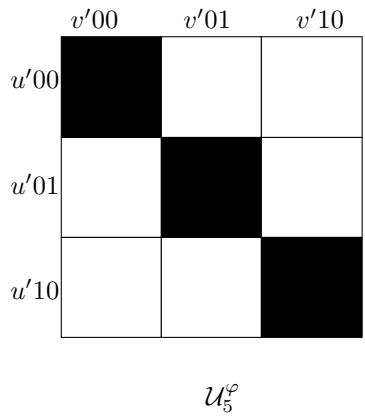
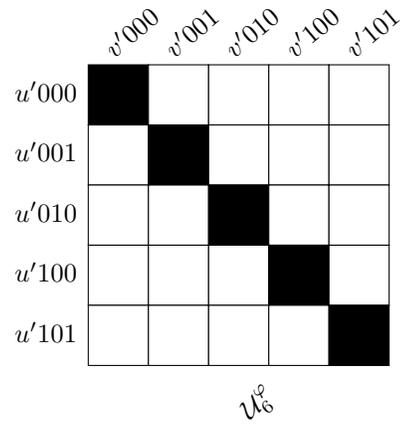
\end{example}


\section{The sequence of compact sets $(\AEns_n)_{n\ge 0}$}\label{sec:AEns}
The observation made in Remark~\ref{rem:conv-diag} leads to the definition of an initial set $\AEns_0$.
The same technique is applied in~\cite{LRS1}.
At first, let us define a segment associated with a pair of words. 

\begin{definition}\label{def:seg}
Let $(u,v)$ in $\Lang\times \Lang$ such that $|u|\ge|v|\ge 0$. 
We define a closed segment $S_{u,v}$ of slope $1$ and of length $\sqrt{2}\cdot \beta^{-|u|-p(u,v)}$ in $[0,1]\times [0,1]$. 
The endpoints of $S_{u,v}$ are given by 
$A_{u,v}:=(0.0^{|u|-|v|}v,0.u)$ and 
$$
B_{u,v}:= A_{u,v} + (\beta^{-|u|-p(u,v)},\beta^{-|u|-p(u,v)})
=(0.0^{|u|-|v|}v 0^{p(u,v)} d_\beta^*(1),0.u 0^{p(u,v)} d_\beta^*(1)).
$$
Observe that, if $u=v=\varepsilon$, the associated segment of slope $1$ has endpoints $(0,0)$ and $(1,1)$. Otherwise, the segment $S_{u,v}$ lies in $[0,1]\times [1/\beta,1]$.
\end{definition}

\begin{definition}\label{def:A0}
Let us define the following compact set which is the closure of a countable union of segments
$$\AEns_0:=\overline{\bigcup_{\substack{(u,v)\\ \text{satisfying} (\star)}} S_{u,v}}.$$
Notice that Definition~\ref{def:seg} implies that $\AEns_0 \subset [0,1]\times[0,1]$. More precisely, $\AEns_0\setminus S_{\varepsilon,\varepsilon} \subset [0,1]\times[1/\beta,1]$.
Furthermore, observe that we take the closure of a union to ensure the compactness of the set.
\end{definition}

\begin{example}
Let $\beta=\varphi$ be the golden ratio. 
In Figure~\ref{fig:A0}, the segment $S_{u,v}$ is represented for all $(u,v)$ satisfying $(\star)$ and such that $0 \le |v| \le |u| \le 10$. 
    \begin{figure}[h!tbp]
        \centering
        \includegraphics[scale=0.65]{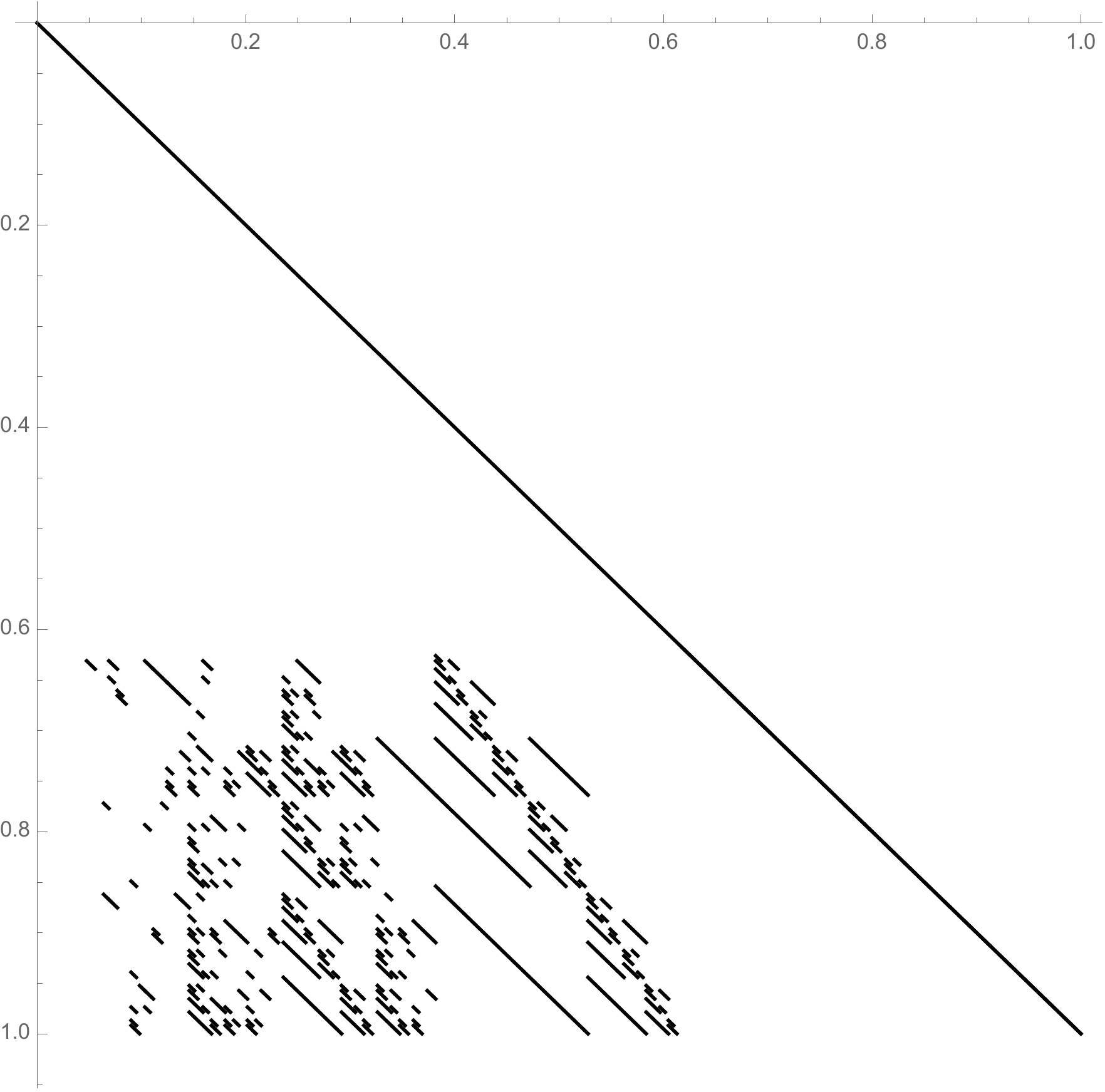}
        \caption{An approximation of $\mathcal{A}^\varphi_0$ computed with words of length $\le 10$.}
        \label{fig:A0}
    \end{figure}
\end{example}

In the following definition, we introduce another sequence of compact sets obtained by transforming the initial set $\AEns_0$ under iterations of two maps. 
This new sequence, which is shown to be a Cauchy sequence in Proposition~\ref{prop:cauchy_seq}, allows us to define properly the limit set $\Llim$. 

\begin{definition}\label{def:An}
We let $c$ denote the homothety of center $(0,0)$ and ratio $1/\beta$ and we consider the map $h:(x,y)\mapsto (x,\beta  y)$. 
We define a sequence of compact sets by setting, for all $n\ge 0$,
$$\AEns_n:=\bigcup_{\substack{0\le i\le n\\ 0\le j\le i}} h^j(c^i(\AEns_0)).$$
\end{definition}

In Figure~\ref{fig:hc}, we apply $c$ and $h$ at most twice from $\AEns_0\setminus S_{\varepsilon,\varepsilon}$. 
Let $m,n$ with $m\leq n$. 
Using Figure~\ref{fig:hc}, observe that
\begin{equation}
    \label{eq:stabilisation}
    \AEns_m\cap ([1/\beta^{m+1},1]\times [0,1])=\AEns_n\cap ([1/\beta^{m+1},1]\times [0,1]).
\end{equation}
\begin{figure}
\vspace{-2cm}
\centering
\begin{tikzpicture}
\tikzstyle{every node}=[shape=circle,fill=none,draw=none,minimum size=10pt,inner sep=2pt]
\node(a0) at (-0.1,0.3) {$0$};
\node(a1) at (9.9,0.3) {$1$};
\node(a2) at (-0.2,-9.9) {$1$};

\draw[fill=white] (0,-6.18) -- (10,-6.18) -- (10,-10) --(0,-10) -- cycle;

\node(a2) at (9.2,-6.5) {$\AEns_0\setminus S_{\varepsilon,\varepsilon}$};

\draw[color=red, fill=red, fill opacity=0.2] (0,-3.81) -- (6.18,-3.81) -- (6.18,-10) --(0,-10) -- cycle;

\node(a3) at (5.2,-4.1) {$c(\AEns_0\setminus S_{\varepsilon,\varepsilon})$};
\node(a4) at (5,-6.5) {$h(c(\AEns_0\setminus S_{\varepsilon,\varepsilon}))$};

\draw[color=blue, fill=blue, fill opacity=0.2] (0,-2.36) -- (3.81,-2.36) -- (3.81,-10) --(0,-10) -- cycle;

\node(a5) at (2.8,-2.7) {$c^2(\AEns_0\setminus S_{\varepsilon,\varepsilon})$};
\node(a6) at (2.5,-4.1) {$h(c^2(\AEns_0\setminus S_{\varepsilon,\varepsilon}))$};
\node(a7) at (2.5,-6.5) {$h^2(c(\AEns_0\setminus S_{\varepsilon,\varepsilon}))$};

\draw [->] (0,0) to [] node [] {}  (10.1,0);
\draw [->] (0,0) to [] node [] {}  (0,-10.1);

\draw [->] (8,-6.5) to [bend right] node [] {}  (5.8,-5);
\node(a8) at (7.3,-5.3) {$c$};

\draw [->] (5,-3.9) to [bend right] node [] {}  (3.5,-3.2);
\node(a9) at (4.5,-3.2) {$c$};

\draw [->] (0.5,-3) to [] node [] {}  (0.5,-8);
\node(a10) at (0.7,-5) {$h$};

\node(a11) at (-0.2, -6.18) {$\frac{1}{\beta}$};
\node(a12) at (-0.2,-3.81) {$\frac{1}{\beta^2}$};
\node(a13) at (-0.2,-2.36) {$\frac{1}{\beta^3}$};
\node(a14) at (6.18,0.3) {$\frac{1}{\beta}$};
\node(a15) at (3.81,0.3) {$\frac{1}{\beta^2}$};
;
\end{tikzpicture}
\caption{Two applications of $c$ and $h$ from $\AEns_0\setminus S_{\varepsilon,\varepsilon}$.}
\label{fig:hc}
\end{figure}
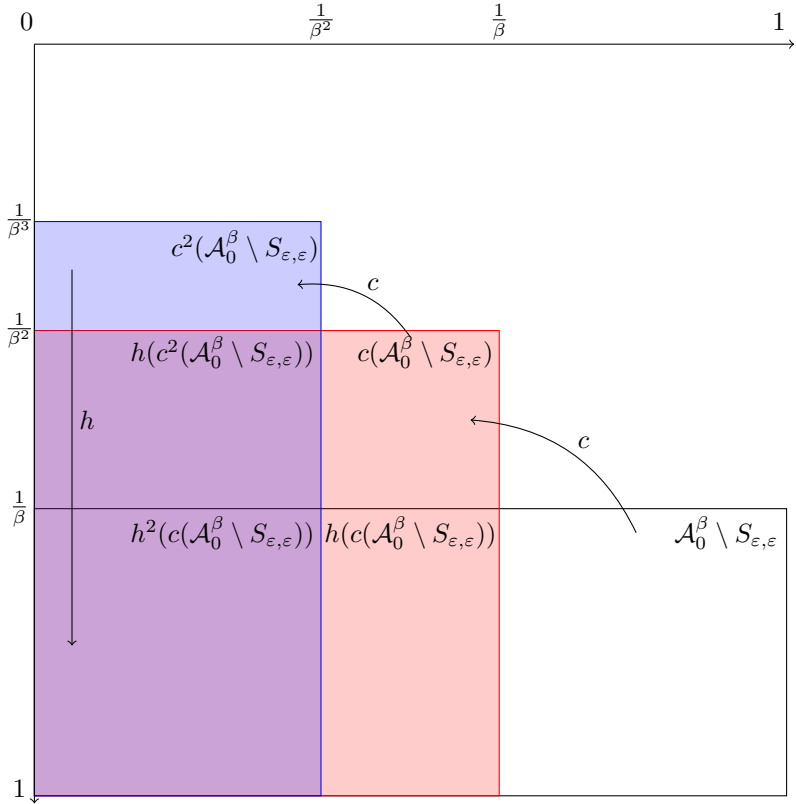

\begin{prop}\label{prop:cauchy_seq}
The sequence $(\AEns_n)_{n\ge 0}$ is a Cauchy sequence.
\end{prop}
\begin{proof}
Let $\epsilon>0$ and take $n>m$. 
We must show that $\AEns_m \subset [\AEns_n]_\epsilon$ and $\AEns_n \subset [\AEns_m]_\epsilon$.
The first inclusion is easy. 
Indeed, since $\AEns_m\subset\AEns_n$, we directly have that $[\AEns_n]_\epsilon$ contains $\AEns_m$.
Let us show the second inclusion. 
From \eqref{eq:stabilisation}, $\AEns_m$ and consequently $[\AEns_m]_\epsilon$ both contain $\AEns_n \cap ([1/\beta^{m+1},1]\times [0,1])$. 
Now we show that $[\AEns_m]_\epsilon$ contains $[0,1/\beta^{m+1}) \times [0,1]$ if $m$ is sufficiently large, which ends the proof.
By Definition~\ref{def:A0}, $\AEns_0$ contains the segment $S_{\varepsilon,\varepsilon}$ of slope $1$ with endpoints $(0,0)$ and $(1,1)$. 
Thus, by Definition~\ref{def:An}, $\AEns_m$ contains the segment $h^m(c^m(S_{\varepsilon,\varepsilon}))$ of slope $\beta^m$ with endpoints $(0,0)$ and $(1/\beta^m,1)$. 
Let $(x,y) \in [0,1/\beta^{m+1}) \times [0,1]$. 
Then $(y/\beta^m,y)$ belongs to $ h^m(c^m(S_{\varepsilon,\varepsilon})) \subset \AEns_m$. 
Consequently, 
$$
d((x,y),\AEns_m) \le d((x,y),(y/\beta^m,y)) \le x + y/\beta^m < \epsilon
$$
if $m$ is sufficiently large. 
\end{proof}

\begin{definition} 
Since the sequence $(\AEns_n)_{n\ge 0}$ is a Cauchy sequence in the complete metric space $(\mathcal{H}(\mathbb{R}^2),d_h)$, its limit is a well-defined compact set denoted by $\Llim$. 
\end{definition}

\begin{example}\label{exa:approx-pour-Fib}
Let $\varphi$ be the golden ratio.  
We have represented in Figure~\ref{fig:approx-pour-Fib} all the segments of $\mathcal{A}^\varphi_0$ for words of length at most $10$ and we have applied the maps $h^j(c^i(\cdot))$ to this set of segments for $0 \le j\le i\le 4$. 
Thus we have an approximation of $\mathcal{A}^\varphi_4$. 
\begin{figure}[h!tb]
   \centering
	\includegraphics[scale=0.4]{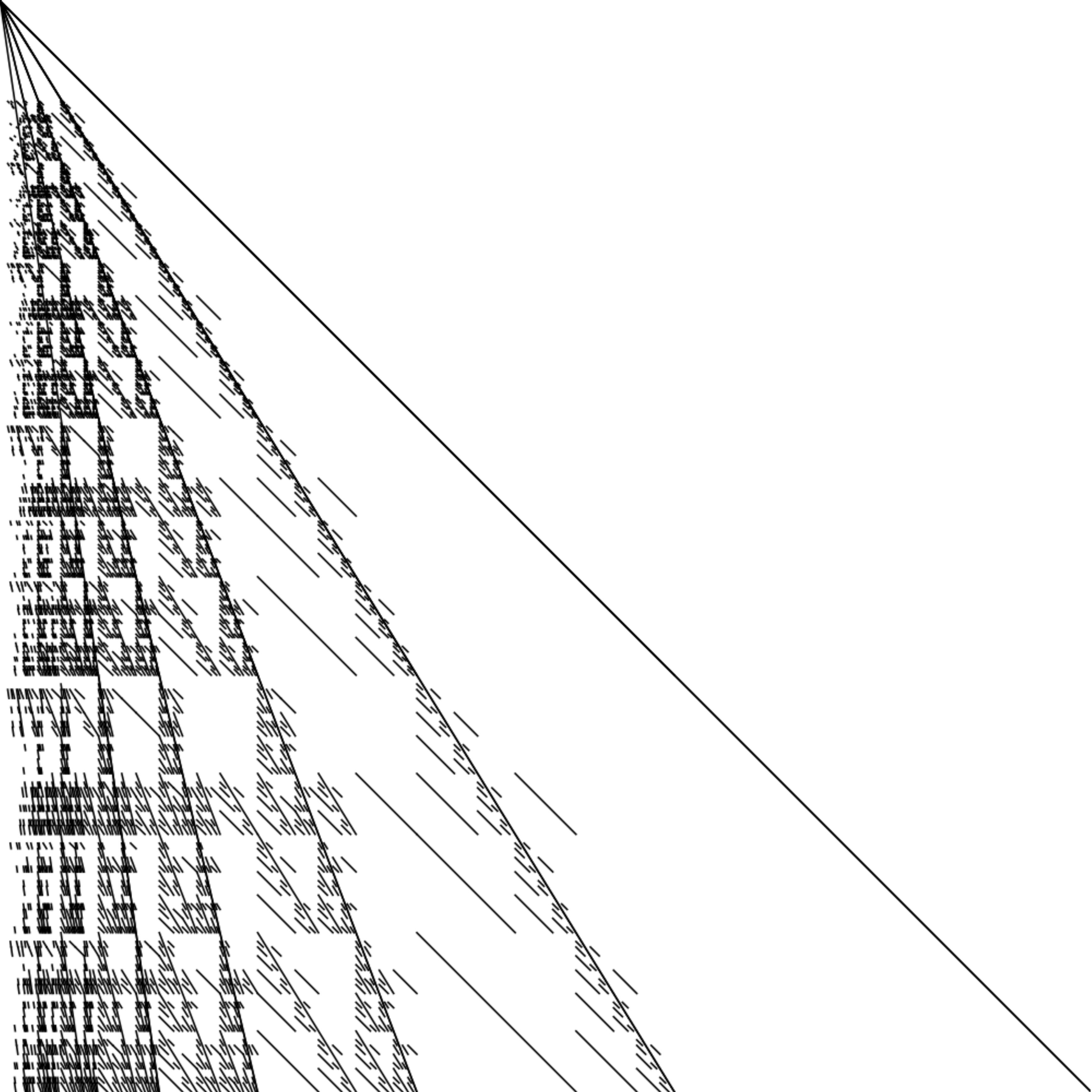}
   \caption{An approximation of the limit set $\mathcal{L}^\varphi$.}
   \label{fig:approx-pour-Fib}
\end{figure}
\end{example}


\section{The limit of the sequence of compact sets $(\U_n)_{n\ge 0}$}\label{sec:limite}

In this section, we show that the sequence $(\U_n)_{n\ge 0}$ of compact subsets of $[0,1]\times [0,1]$ also converges to $\Llim$. 
The proofs of Lemma~\ref{lem:first_inclusion}, Lemma~\ref{lem:Un} are essentially the same as the ones from~\cite{LRS1} (\cite[Lemma 27, Lemma 28, Theorem 29]{LRS1}).
However we recall them so that the paper is self-contained. 
The first part is to show that, when $\epsilon$ is a positive real number, then $\U_n\subset [\Llim]_\epsilon$ for all sufficiently large $n$. 

\begin{lemma}\label{lem:first_inclusion} 
Let $\epsilon > 0$. 
For all sufficiently large $n\in\mathbb{N}$, we have
$$
\U_n\subset [\Llim]_\epsilon.
$$
\end{lemma}
\begin{proof}
Let $\epsilon>0$. Take $n\in \mathbb{N}$ and let $(x,y)\in \U_n$. 
From Remark~\ref{rem:pixel}, there exists $(u,v)\in \Lang\times \Lang$ such that $\binom{u}{v}\equiv 1\bmod{2}$, $0 \le |v|\le|u|\le n$ and the point $(x,y)$ belongs to the square region
\begin{equation}
    \label{eq:particular-sq-region}
((\val_{U_\beta}(v),\val_{U_\beta}(u))+Q)/U_\beta(n) \subset \U_n.
\end{equation} 
Let us set
$$
A:= \left( \frac{\val_{U_\beta}(v)}{U_\beta(n)},\frac{\val_{U_\beta}(u)}{U_\beta(n)} \right)
$$
to be the upper-left corner of the square region~\eqref{eq:particular-sq-region} in $\U_n$.

Assume first that $(u,v)$ satisfies $(\star)$. 
The segment $S_{u,v}$ of length $\sqrt{2}\cdot \beta^{-|u|-p(u,v)}$ having $A_{u,v}=(0.0^{|u|-|v|}v,0.u)$ as endpoint belongs to $\AEns_0$. 
Now apply $n-|u|$ times the homothety $c$ to this segment. 
So the segment $c^{n-|u|}(S_{u,v})$ of length $\sqrt{2}\cdot  \beta^{-n-p(u,v)}$ of endpoint $B_1:=(0.0^{n-|v|}v,0.0^{n-|u|}u)$ belongs to $\AEns_{n-|u|}$ and thus to $\Llim$. 
Using~\eqref{eq:conv-beta-U} (the reasoning is similar to the one developed in the proof of Lemma~\ref{lem:origine-extremite}), there exists $N_1\in \mathbb{N}$ such that, for all $n\ge N_1$, $d(A,B_1)<\epsilon/2$.
Hence, for all $n\ge N_1$ such that $\sqrt{2}/U_\beta(n) < \epsilon/2$, we have 
$$
d((x,y),\Llim)
\le d((x,y),B_1)
\le d((x,y),A) + d(A,B_1) \le \sqrt{2}/U_\beta(n) + d(A,B_1)
<\epsilon.
$$

Now assume that $(u,v)$ does not satisfy $(\star)$.
Since $\binom{u}{v}\equiv 1 \bmod{2}$, then either $u$ and $v$ are non-empty words, or $u$ is non-empty and $v=\varepsilon$.
Suppose that $u$ and $v$ are non-empty. 
By assumption, we have an odd number $r$ of occurrences of $v$ in $u$. 
For each occurrence of $v$ in $u$, we count the total number of zeroes after it. 
We thus define a sequence of non-negative integer indices 
$$|u| \geq i_1 \geq i_2 \geq \cdots \geq i_r \geq 0$$
corresponding to the number of zeroes following the first, the second, ..., the $r$th occurrence of $v$ in $u$. 
Now let $k$ be a non-negative integer such that $k > \lceil \log_2 |u|\rceil$ and $2^k > p(u,v)$.
By definition of $p(u,v)$, the words $u0^{2^k}1$ and $v0^{2^k}1$ belong to $\Lang$. 
We get 
$$
\binom{u0^{2^k}1}{v0^{2^k}1} = \sum_{\ell=1}^r \binom{2^k+i_\ell}{2^k}.
$$
Indeed, for each $\ell\in\{1,\ldots,r\}$, consider the $\ell$th occurrence of $v$ in $u$: we have the factorization $u=pw$ where the last letter of $p$ is the last letter of the $\ell$th occurrence of $v$ and $|w|_0=i_\ell$. 
With this particular occurrence of $v$, we obtain occurrences of $v0^{2^k}1$ in $u0^{2^k}1$ by choosing $2^k$ zeroes among the $2^k+i_\ell$ zeroes available in $w0^{2^k}1$. 
Moreover, with the long block of $2^k$ zeroes, it is not possible to have any other occurrence of $v0^{2^k}1$ than those obtained from occurrences of $v$ in $u$.

Then, for each $\ell\in\{1,\ldots,r\}$, we have
$$
\binom{2^k+i_\ell}{2^k} \equiv 1 \bmod{2}
$$
from Theorem~\ref{thm:lucas}. 
Since $r$ is odd, we get
$$
\binom{u0^{2^k}1}{v0^{2^k}1} \equiv 1 \bmod{2}.
$$
Now, for all $k\in\mathbb{N}$ such that $k > \lceil \log_2 |u|\rceil$ and $2^k > p(u,v)$, it is easy to check that the pair of words $(u0^{2^k}1,v0^{2^k}1)$ satisfies $(\star)$. 
For the sake of simplicity, define $u_k :=  u0^{2^k}1$, $v_k := v0^{2^k}1$ and $p_k:=p(u_k,v_k)$. 
As in the first part of the proof, the segment $S_{u_k, v_k}$ of length $\sqrt{2}\cdot \beta^{-|u|-2^k-1-p_k}$ having $A_{u_k,v_k}=(0.0^{|u|-|v|}v0^{2^k}1,0.u0^{2^k}1)$ as endpoint belongs to $\AEns_0$. 
Now apply $n-|u|$ times the homothety $c$ to this segment. 
So the segment $c^{n-|u|}(S_{u_k,v_k})$ of length $\sqrt{2}\cdot \beta^{-n-2^k-1-p_k}$ of endpoint $B_2:=(0.0^{n-|v|}v0^{2^k}1,0.0^{n-|u|}u0^{2^k}1)$ belongs to $\AEns_{n-|u|}$ and thus to $\Llim$. 
Using again~\eqref{eq:conv-beta-U} and a reasoning similar to the one from the proof of Lemma~\ref{lem:origine-extremite}, there exists $N_2\in \mathbb{N}$ such that, for all $n\ge N_2$, $d(A,B_2)<\epsilon/2$.
Hence, for all $n\ge N_2$ such that $\sqrt{2}/U_\beta(n) < \epsilon/2$, we have 
$$
d((x,y),\Llim)
\le d((x,y),B_2)
\le d((x,y),A) + d(A,B_2) \le \sqrt{2}/U_\beta(n) + d(A,B_2)
<\epsilon.
$$

Assume now that $u$ is non-empty and $v=\varepsilon$.
In this case, the point $A$ is on the vertical line of equation $x=0$.
By Definition~\ref{def:A0}, $\AEns_0$ contains the segment $S_{\varepsilon,\varepsilon}$ of slope $1$ with endpoints $(0,0)$ and $(1,1)$. 
Thus, by Definition~\ref{def:An}, $\AEns_n$ contains the segment $h^n(c^n(S_{\varepsilon,\varepsilon}))$ of slope $\beta^n$ with endpoints $(0,0)$ and $(1/\beta^n,1)$. 
This segment also lies in $\Llim$. 
There exists $N_3\in \mathbb{N}$ such that, for all $n\ge N_3$, $d(A,h^n(c^n(S_{\varepsilon,\varepsilon}))) \le 1/\beta^n <\epsilon/2$.  
Consequently, for all $n\ge N_3$ such that $\sqrt{2}/U_\beta(n) < \epsilon/2$, we have 
\begin{eqnarray*}
d((x,y),\Llim)
&\le& d((x,y),h^n(c^n(S_{\varepsilon,\varepsilon})))
\le d((x,y),A) + d(A,h^n(c^n(S_{\varepsilon,\varepsilon}))) \\
&\le& \sqrt{2}/U_\beta(n) + d(A,h^n(c^n(S_{\varepsilon,\varepsilon}))) 
< \epsilon.
\end{eqnarray*}

In each of the three cases, we conclude that $(x,y)\in [\Llim]_\epsilon$, which proves that $\U_n \subset [\Llim]_\epsilon$ for all sufficiently large $n$. 
\end{proof}

If $\epsilon >0$, it remains to show that $\Llim \subset [U_n]_\epsilon$ for all sufficiently large $n \in \mathbb{N}$.
To that aim, we need to bound the number of consecutive words, in the genealogical order, that end with $0$ in $\Lang$. 

\begin{definition}\label{def:C-beta}
We let $C_\beta\in\mathbb{N}$ denote the maximal number of consecutive $0$ in $d_\beta^*(1)$, i.e.,
$$
C_\beta := \max \{ n\in \mathbb{N} \mid 0^n \text{ is a factor of } d_\beta^*(1)  \}.
$$
\end{definition}

In the next proposition, we show that the maximal number of consecutive words ending with $0$ in $\Lang$ is $C_\beta+1$.

\begin{prop}\label{pro:nbr-mots-0}
If we order the words in $\Lang$ by the genealogical order, the maximal number of consecutive words ending with $0$ in $\Lang$, i.e., the maximal number of consecutive normal $U_\beta$-representations ending with $0$, is $C_\beta +1$.
\end{prop}
\begin{proof}
Let $n\in \mathbb{N}$ be such that $\rep_{U_\beta}(n)$ ends with $0$. 
We can suppose that $\rep_{U_\beta}(n-1)$ does not end with $0$, otherwise we translate $n$. 
If $|\rep_{U_\beta}(n+1)|=|\rep_{U_\beta}(n)|$, then $\rep_{U_\beta}(n+1)$ does not end with $0$ because $U_\beta(m)\ge 2$ for all $m \ge 1$.
Indeed, if a single digit (not the least significant one) is changed, then the value is increased by at least $2$. 
Let $C\ge 1$ be such that, for all $k\in\{0,\ldots,C\}$, $|\rep_{U_\beta}(n+k)|=|\rep_{U_\beta}(n)|+k$ and $|\rep_{U_\beta}(n+C+1)|=|\rep_{U_\beta}(n+C)|$.
The normal-$U$ representation preserves the order, i.e., for all integers $m_1$ and $m_2$, $m_1 \le m_2$ if and only if $\rep_{U_\beta}(m_1) \le \rep_{U_\beta}(m_2)$ (see, for instance,~\cite{BR}).
Thus, the words $\rep_{U_\beta}(n), \ldots, \rep_{U_\beta}(n+C-1)$ are prefixes of $d_\beta^*(1)$, respectively of length $|\rep_{U_\beta}(n)|, |\rep_{U_\beta}(n)|+1, \ldots, |\rep_{U_\beta}(n)|+C-1$ (the prefixes of $d_\beta^*(1)$ are the maximal words of different length in $\Lang$).
By Definition~\ref{def:C-beta}, we deduce that $C\le C_\beta$. 
Consequently, there are at most $C_\beta +1$ consecutive words ending with $0$ in $\Lang$. 
\end{proof}

Let us illustrate the previous proposition. 

\begin{example}
Let $\varphi$ be the golden ratio. Then $C_\varphi=1$ since $d_\varphi^*(1)=(10)^\omega$. The first few words of $L_{U_\varphi}$ are
$\varepsilon, 1, 10, 100, 101, 1000, 1001, 1010, 10000, 10001, \ldots$.
The maximal number of consecutive words ending with $0$ in $L_{U_\varphi}$ is $C_\varphi+1=2$.
\end{example}

\begin{example}\label{ex:Parry2}
Let $\beta$ be the dominant root of the polynomial $P(X)=X^4-X^3-1$. 
Then $\beta \approx 1.38028$ is a Parry number with $d_\beta(1)=1001$ and $d_\beta^*(1)=(1000)^\omega$. 
The automaton $\Aut$ is depicted in Figure~\ref{fig:Aut-beta-ex-2}. 
In this example, $C_\beta = 3$. 
The first few words of $\Lang$ are $\varepsilon, 1, 10, 100, 1000, 10000, 10001, \ldots$. 
The maximal number of consecutive words ending with $0$ in $\Lang$ is $C_\beta+1=4$.
 
\begin{figure}
\centering
\begin{tikzpicture}
\tikzstyle{every node}=[shape=circle,fill=none,draw=black,minimum size=20pt,inner sep=2pt]
\node[accepting](a) at (0,0) {$a_0$};
\node[accepting](b) at (2,0) {$a_1$};
\node[accepting](c) at (4,0) {$a_2$};
\node[accepting](d) at (6,0) {$a_3$};

\tikzstyle{every node}=[shape=circle,fill=none,draw=none,minimum size=10pt,inner sep=2pt]
\node(a0) at (-1,0) {};
\node(a1) at (0,1) {$0$};

\tikzstyle{every path}=[color =black, line width = 0.5 pt]
\tikzstyle{every node}=[shape=circle,minimum size=5pt,inner sep=2pt]
\draw [->] (a0) to [] node [] {}  (a);
\draw [->] (a) to [loop above] node [] {}  (a);

\draw [->] (a) to [] node [above] {$1$}  (b);
\draw [->] (b) to [] node [above] {$0$}  (c);
\draw [->] (c) to [] node [above] {$0$}  (d);

\draw [->] (d) to [bend right=50] node [above] {$0$}  (a);

;
\end{tikzpicture}
\caption{The automaton $\Aut$ for the dominant root $\beta$ of the polynomial $P(X)=X^4-X^3-1$.}
\label{fig:Aut-beta-ex-2}
\end{figure}
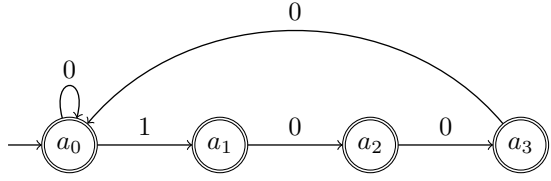
\end{example}

\begin{lemma}\label{lem:Un}
Let $\epsilon>0$. For all $(x,y)\in\Llim$, $d((x,y),\U_n)<\epsilon$ for all sufficiently large $n$.
\end{lemma}
\begin{proof}
Let $\epsilon>0$ and let $(x,y)\in\Llim$. 
Since $(\AEns_n)_{n\ge 0}$ converges to $\Llim$, there exists $N_1$ and $(x',y')\in\AEns_{N_1}$ such that,  
$$d((x,y),(x',y'))<\epsilon/4.$$
By definition of $\AEns_{N_1}$, there exist $i,j$ such that $0\le j\le i\le N_1$ and $(x_0',y_0')\in\AEns_0$ such that
$$h^j(c^i((x_0',y_0')))=(x',y').$$
By definition of $\AEns_0$, there exists a pair $(u,v)\in \Lang\times \Lang$ satisfying $(\star)$ and $(x_0'',y_0'')\in S_{u,v}$ such that
 $$d((x_0',y_0'),(x_0'',y_0''))<\epsilon/4.$$
Notice that, since $j\le i$,  
\begin{eqnarray*}
d((x',y'),h^j(c^i((x_0'',y_0''))))&=&d(h^j(c^i((x_0',y_0'))),h^j(c^i((x_0'',y_0''))))\\
&\le& d((x_0',y_0'),(x_0'',y_0''))<\epsilon/4.
\end{eqnarray*}
Consequently, we get that
$$d((x,y),h^j(c^i( (x_0'',y_0''))))<\epsilon/2.$$

In the second part of the proof, we will show that $d(h^j(c^i((x_0'',y_0''))),\U_n)<\epsilon/2$ for all sufficiently large $n$. 
We will make use of the constants $i,j$, the words $u,v$ given above and the integer $p:=p(u,v)$.

Set 
$$
L_{u,v}:=
\left\{
    \begin{array}{cl}
        \Lang, & \mbox{if } u=v=\varepsilon; \\
         0^*\Lang, & \mbox{otherwise.}
    \end{array}
\right.
$$ 
Since $(u,v)\in \Lang\times \Lang$ satisfies $(\star)$, the pair of words $(u0^pw,v0^pw)$ has an odd binomial coefficient, for all words $w \in L_{u,v}$, using Lemma~\ref{lem:u-u} and Corollary~\ref{cor:completion}. 
In particular, this is the case when $w \in L_{u,v}$ is of length $n$. 
We can choose $n$ sufficiently large such that $U_\beta(n) \ge C_\beta+3$ using Proposition~\ref{pro:nbr-mots-0}.
In this case, there exist at least two words $w\in L_{u,v}$ with $|w|=n$ and not ending with $0$. 
Furthermore, as soon as $w$ does not end with $0$, Lemma~\ref{lem:lothaire-bin} shows that
$$
\binom{u0^pw0^k}{v0^pw}=\binom{u0^pw}{v0^pw}\equiv 1\bmod{2}
$$
for all $k\ge 0$.
By definition of the sequence $U_\beta$, we also have
$$
\# \{ z \in 0^*\Lang \mid u0^pwz \in \Lang \text{ and } |z|= k \} \le U_\beta(k).
$$
Thus, for all $j\le i$, we conclude that at least one of the $U_\beta(j)$ binomial coefficients of the form $\binom{u0^pwz}{v0^pw}$ with $w$ not ending with $0$ and $|z|=j$ is odd (indeed, choose $z=0^j$ for instance). 
Otherwise stated, at least one of the square regions
\begin{equation}\label{eq:square-region}
\left( \frac{\val_{U_\beta}(v0^pw)}{U_\beta(n+i+|u|+p)}, \frac{\val_{U_\beta}(u0^pwz)}{U_\beta(n+i+|u|+p)} \right) + \frac{Q}{U_\beta(n+i+|u|+p)}, \text{ with }|z|=j,
\end{equation}
is a subset of $\U_{n+i+|u|+p}$, since $|v0^pw|,|u0^pwz| \le n+i+|u|+p$. 
This can be visualized in Figure~\ref{fig:square-regions-w-fixed}.

\begin{figure}
\vspace{-2cm}
\centering
\begin{tikzpicture}
\tikzstyle{every node}=[shape=circle,fill=none,draw=none,minimum size=10pt,inner sep=2pt]
\node(a0) at (-1,0.5) {$u0^pw$};
\node(a1) at (-4,0.5) {$j=0$};
\node(a2) at (0.5,1.5) {$v0^pw$};

\node(a3) at (-4,-1.5) {$j=1$};
\node(a4) at (-1,-1.5) {$u0^pw0$};
\node(a5) at (-1,-2.5) {$\vdots$};
\node(a6) at (2.7,-2) {$\le U_\beta(1)$};

\node(a7) at (-4,-4.5) {$j=2$};
\node(a8) at (-1,-4.5) {$u0^pw00$};
\node(a9) at (-1,-5.5) {$\vdots$};
\node(a10) at (2.7,-5.5) {$\le U_\beta(2)$};

\node(a7) at (-4,-8.5) {$j=3$};
\node(a8) at (-1,-8.5) {$u0^pw000$};
\node(a9) at (-1,-9.5) {$\vdots$};
\node(a10) at (2.7,-10.5) {$\le U_\beta(3)$};

\node(a11) at (3.5,1.5) {$\U_{n+i+|u|+p}$};

\draw[fill=black] (0,0) -- (1,0) -- (1,1) --(0,1) -- cycle;

\draw[fill=black] (0,-2) -- (1,-2) -- (1,-1) --(0,-1) -- cycle;
\draw[fill=white] (0,-3) -- (1,-3) -- (1,-2) --(0,-2) -- cycle;
\draw [->] (2,-1) to [] node [] {}  (2,-3);
\draw [->] (2,-3) to [] node [] {}  (2,-1);

\draw[fill=black] (0,-5) -- (1,-5) -- (1,-4) --(0,-4) -- cycle;
\draw[fill=white] (0,-6) -- (1,-6) -- (1,-5) --(0,-5) -- cycle;
\draw[fill=white] (0,-7) -- (1,-7) -- (1,-6) --(0,-6) -- cycle;
\draw [->] (2,-4) to [] node [] {}  (2,-7);
\draw [->] (2,-7) to [] node [] {}  (2,-4);

\draw[fill=black] (0,-9) -- (1,-9) -- (1,-8) --(0,-8) -- cycle;
\draw[fill=white] (0,-10) -- (1,-10) -- (1,-9) --(0,-9) -- cycle;
\draw[fill=white] (0,-11) -- (1,-11) -- (1,-10) --(0,-10) -- cycle;
\draw[fill=white] (0,-12) -- (1,-12) -- (1,-11) --(0,-11) -- cycle;
\draw[fill=white] (0,-13) -- (1,-13) -- (1,-12) --(0,-12) -- cycle;
\draw [->] (2,-8) to [] node [] {}  (2,-13);
\draw [->] (2,-13) to [] node [] {}  (2,-8);
;
\end{tikzpicture}
\caption{If $w$ does not end with $0$ and is such that $|w|=n$, then $\binom{u0^pw0^j}{v0^pw}$ being odd creates a square region in $\U_{n+i+|u|+p}$.}
\label{fig:square-regions-w-fixed}
\end{figure}
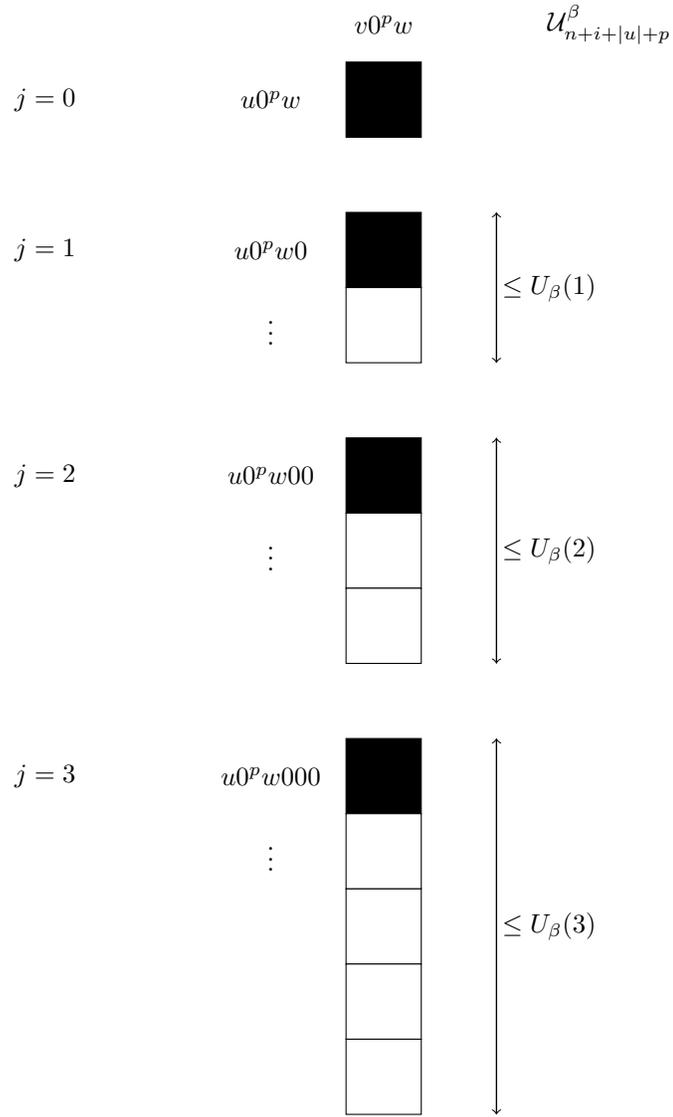

Now observe that, for any word $w\in L_{u,v}$, each square region of the form~\eqref{eq:square-region} is intersected by $h^j(c^i(S_{u,v}))$. 
Indeed, the latter segment has $A:=(0.0^{i+|u|-|v|}v,0.0^{i-j}u)$ and $B:=(0.0^{i+|u|-|v|}v0^pd_\beta^*(1),0.0^{i-j}u0^pd_\beta^*(1))$ as endpoints and slope $\beta^j$. 
Using~\eqref{eq:conv-beta-U}, if $n$ is sufficiently large, the points
$$
\left( \frac{\val_{U_\beta}(v0^p0^n)}{U_\beta(n+i+|u|+p)}, \frac{\val_{U_\beta}(u0^p0^{n+j})}{U_\beta(n+i+|u|+p)} \right) 
\left( \text{resp.,} \left( \frac{\val_{U_\beta}(v0^pd_n)}{U_\beta(n+i+|u|+p)}, \frac{\val_{U_\beta}(u0^pd_{n+j})}{U_\beta(n+i+|u|+p)} \right) \right)
$$
and $A$ (resp., $B$) are close for all $j\le i$, where $d_n$ denotes the prefix of length $n$ of $d_\beta^*(1)$ for all $n\ge 0$.
When $u$ and $v$ are non-empty, this can be seen in Figure~\ref{fig:segments-and-square-regions} where each rectangular gray region contains at least one square region from $\U_{n+i+|u|+p}$ (to draw this picture, we take the particular case of the golden ratio $\varphi$ and $i=2$).
When $u=v=\varepsilon$, Figure~\ref{fig:segments-and-square-regions} is modified in the following way: simply replace each word of the forms $u0^\ell$, $v0^\ell$ by $\varepsilon$. 

\begin{figure}
\vspace{-2cm}
\begin{center}
\begin{tikzpicture}
\node() at (-2.1,3.2) {$0$};
\node() at (15.8,3.2) {$1$};
\node() at (-1.1,-15.8) {$1$};
\node() at (10.5,1.5) {$\U_{n+2+|u|+p}$};
\node() at (2.3,1.2) {$c^2(S_{u,v})$};
\node() at (2.5,-2.8) {$h(c^2(S_{u,v}))$};
\node() at (2.6,-9.3) {$h^2(c^2(S_{u,v}))$};
\node() at (6.2,-2.8) {$c(S_{u,v})$};
\node() at (6.4,-9.3) {$h(c(S_{u,v}))$};
\node() at (12.4,-9.3) {$S_{u,v}$};

\node() at (0.3,2.7) {$v0^p0^n$};
\node() at (2.5,2.7) {$v0^pd_n$};
\node() at (4,-0.8) {$v0^p0^{n+1}$};
\node() at (7.5,-0.8) {$v0^pd_{n+1}$};
\node() at (9,-5.8) {$v0^p0^{n+2}$};
\node() at (15,-5.8) {$v0^pd_{n+2}$};

\node() at (-0.5,2.3) {$u0^p0^n$};
\node() at (-0.5,0.2) {$u0^pd_n$};
\node() at (-0.7,-1.2) {$u0^p0^{n+1}$};
\node() at (-0.7,-4.8) {$u0^pd_{n+1}$};
\node() at (-0.7,-6.2) {$u0^p0^{n+2}$};
\node() at (-0.7,-12.2) {$u0^pd_{n+2}$};

\draw [->] (-2,3) to [] node [] {}  (16,3);
\draw [->] (-2,3) to [] node [] {}  (-2,-16);

\fill[gray] (0,2.5) rectangle (0.5,2);
\fill[gray] (0.5,2) rectangle (1,1.5);
\fill[gray] (1,1.5) rectangle (1.5,1);
\fill[gray] (1.5,1) rectangle (2,0.5);
\fill[gray] (2,0.5) rectangle (2.5,0);

\fill[gray] (3.5,-1) rectangle (4,-1.5);
\fill[gray] (4,-1.5) rectangle (4.5,-2);
\fill[gray] (4.5,-2) rectangle (5,-2.5);
\fill[gray] (5,-2.5) rectangle (5.5,-3);
\fill[gray] (5.5,-3) rectangle (6,-3.5);
\fill[gray] (6,-3.5) rectangle (6.5,-4);
\fill[gray] (6.5,-4) rectangle (7,-4.5);
\fill[gray] (7,-4.5) rectangle (7.5,-5);

\fill[gray] (8.5,-6) rectangle (9,-6.5);
\fill[gray] (9,-6.5) rectangle (9.5,-7);
\fill[gray] (9.5,-7) rectangle (10,-7.5);
\fill[gray] (10,-7.5) rectangle (10.5,-8);
\fill[gray] (10.5,-8) rectangle (11,-8.5);
\fill[gray] (11,-8.5) rectangle (11.5,-9);
\fill[gray] (11.5,-9) rectangle (12,-9.5);
\fill[gray] (12,-9.5) rectangle (12.5,-10);
\fill[gray] (12.5,-10) rectangle (13,-10.5);
\fill[gray] (13,-10.5) rectangle (13.5,-11);
\fill[gray] (13.5,-11) rectangle (14,-11.5);
\fill[gray] (14,-11.5) rectangle (14.5,-12);
\fill[gray] (14.5,-12) rectangle (15,-12.5);

\fill[white] (0,-1) rectangle (0.5,-2);
\fill[gray] (0.5,-2) rectangle (1,-2.5);
\fill[white] (1,-2.5) rectangle (1.5,-3.5);
\fill[white] (1.5,-3.5) rectangle (2,-4.5);
\fill[gray] (2,-4.5) rectangle (2.5,-5);

\fill[white] (0,-6) rectangle (0.5,-7.5);
\fill[gray] (0.5,-7.5) rectangle (1,-8.5);
\fill[white] (1,-8.5) rectangle (1.5,-10);
\fill[white] (1.5,-10) rectangle (2,-11.5);
\fill[gray] (2,-11.5) rectangle (2.5,-12.5);

\fill[white] (3.5,-6) rectangle (4,-7);
\fill[gray] (4,-7) rectangle (4.5,-7.5);
\fill[white] (4.5,-7.5) rectangle (5,-8.5);
\fill[white] (5,-8.5) rectangle (5.5,-9.5);
\fill[gray] (5.5,-9.5) rectangle (6,-10);
\fill[white] (6,-10) rectangle (6.5,-11);
\fill[gray] (6.5,-11) rectangle (7,-11.5);
\fill[white] (7,-11.5) rectangle (7.5,-12.5);

\draw [very thin, gray] (0,0) grid[step=0.5] (2.5,2.5);
\draw [black] (0,2.5) -- (2.5,0);
\draw [very thin, gray] (3.5,-1) grid[step=0.5] (7.5,-5);
\draw [very thin, gray] (3.5,-1) -- (3.5,-5);
\draw [black] (3.5,-1) -- (7.5,-5);
\draw [very thin, gray] (8.5,-6) grid[step=0.5] (15,-12.5);
\draw [very thin, gray] (8.5,-6) -- (8.5,-12.5);
\draw [black] (8.5,-6) -- (15,-12.5);

\draw [very thin, gray] (0,-5) grid[step=0.5] (2.5,-1);
\draw [black] (0,-1) -- (2.5,-5);
\draw [very thin, gray] (0,-12.5) grid[step=0.5] (2.5,-6);
\draw [black] (0,-6) -- (2.5,-12.5);
\draw [very thin, gray] (3.5,-12.5) grid[step=0.5] (7.5,-6);
\draw [very thin, gray] (3.5,-12.5) -- (3.5,-6);
\draw [black] (3.5,-6) -- (7.5,-12.5);

\draw [black] (9.5,-12.5) -- (9,-13);
\draw [black] (10,-12.5) -- (10.5,-13);
\draw [->,>=stealth] (9,-13) to [] node [] {}  (10.5,-13);
\draw [->,>=stealth] (10.5,-13) to [] node [] {}  (9,-13);
\node() at (9.7,-13.4) {$\frac{1}{U_{n+2+|u|+p}}$};

\draw[very thin, black] (0,-1) rectangle (0.5,-2);
\draw[very thin, black] (1,-2.5) rectangle (1.5,-3.5);
\draw[very thin, black] (1.5,-3.5) rectangle (2,-4.5);

\draw[very thin, black] (0,-6) rectangle (0.5,-7.5);
\draw[very thin, black] (1,-8.5) rectangle (1.5,-10);
\draw[very thin, black] (1.5,-10) rectangle (2,-11.5);

\draw[very thin, black] (3.5,-6) rectangle (4,-7);
\draw[very thin, black] (4.5,-7.5) rectangle (5,-8.5);
\draw[very thin, black] (5,-8.5) rectangle (5.5,-9.5);
\draw[very thin, black] (6,-10) rectangle (6.5,-11);
\draw[very thin, black] (7,-11.5) rectangle (7.5,-12.5);
;
\end{tikzpicture}
\end{center}
\caption{The situation occurring in the proof of Lemma~\ref{lem:Un}, where we choose $\beta$ to be the golden ratio.}
\label{fig:segments-and-square-regions}
\end{figure}
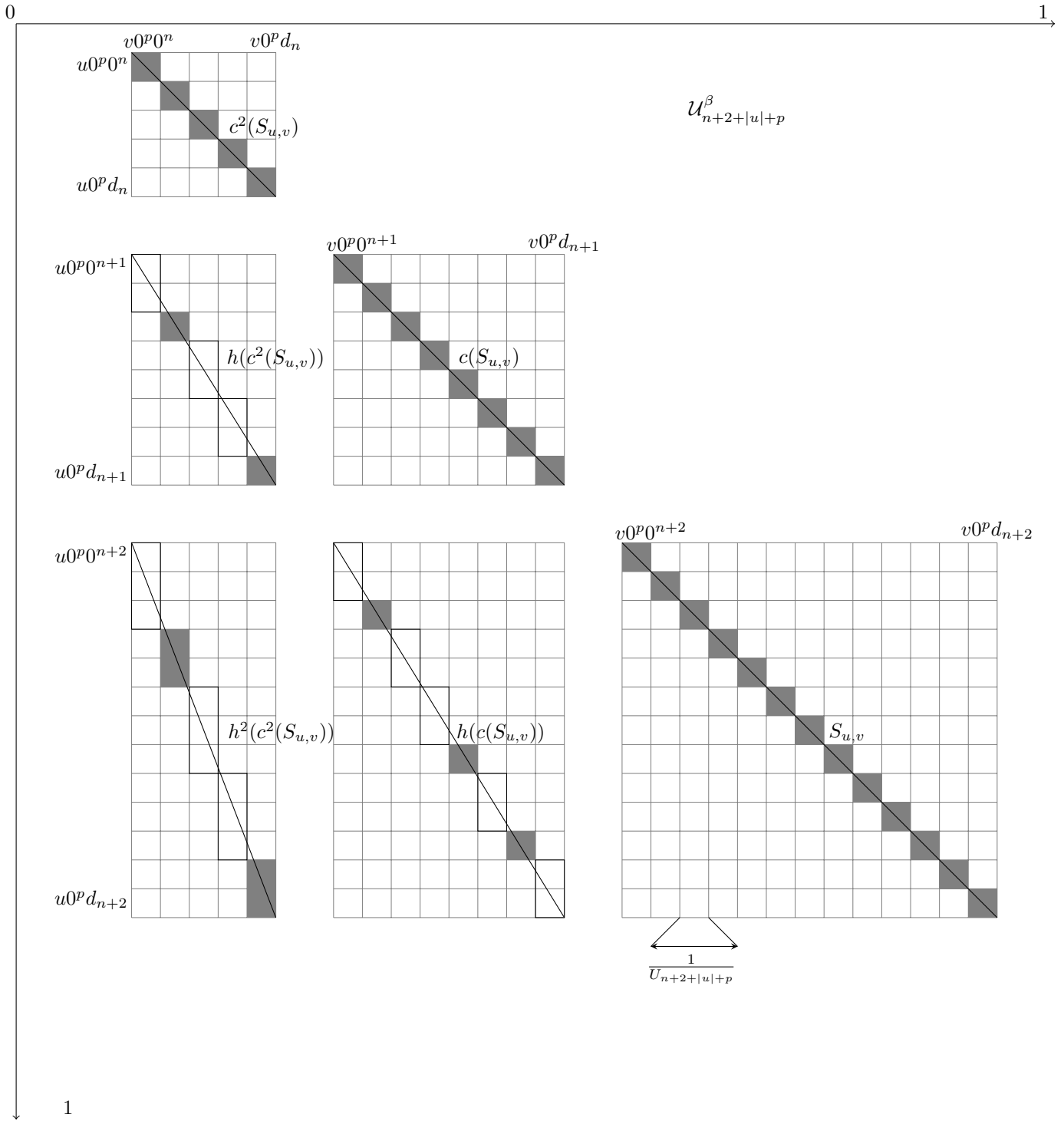

Consequently, every point of $h^j(c^i(S_{u,v}))$ is at distance at most 
$$
\frac{2\cdot (C_\beta +2) \cdot U_\beta(j)}{U_\beta(n+i+|u|+p)}
$$ 
from a point in $\U_{n+i+|u|+p}$ when $n$ is sufficiently large. 
Indeed, either the point falls into a gray region from Figure~\ref{fig:segments-and-square-regions}, or not.
In the first case, the point is at distance at most $U_\beta(j) / U_\beta(n+i+|u|+p)$ from a square region in $\U_{n+i+|u|+p}$; see Figure~\ref{fig:square-regions-w-fixed}.
Observe that this square region is of the form~\eqref{eq:square-region} where $w$ does not end with $0$.
Otherwise, the point falls into a (white) square region of the form
$$
\left( \frac{\val_{U_\beta}(v0^pw)}{U_\beta(n+i+|u|+p)}, \frac{\val_{U_\beta}(u0^pw'z)}{U_\beta(n+i+|u|+p)} \right) + \frac{Q}{U_\beta(n+i+|u|+p)}, \text{ with }|w|=|w'|=n, |z|=j.
$$
Since $n$ is large enough, there exists a word $w''$ not ending with $0$ with $|w''|=n$, which is within a distance of $2\cdot (C_\beta + 2)$ of $w$ and $w'$. 
Then, applying the argument from the previous case proves the statement. 

In particular, the result holds for the point $h^j(c^i((x_0'',y_0'')))$ belonging to $h^j(c^i(S_{u,v}))$. 
Hence, for all sufficiently large $n$, 
$$d(h^j(c^i((x_0'',y_0''))),\U_{n})<\epsilon/2.$$
The conclusion follows. 
\end{proof}

\begin{corollary}\label{cor:suite-Cauchy-point-segment}
Let $(u,v) \in \Lang \times \Lang$ satisfying $(\star)$ and let $0\le j \le i$. 
For every point $(f,g)$ of the segment $h^j(c^i(S_{u,v}))$, there exists a sequence $((f_n,g_n))_{n\ge 0}$ converging to $(f,g)$ and such that $(f_n,g_n)\in \U_n$ for all $n$.
\end{corollary}
\begin{proof}
Let $(f,g)$ be a point of the segment $h^j(c^i(S_{u,v}))$.
From the proof of Lemma~\ref{lem:Un}, we have
$$
d((f,g), \U_m) \le  \frac{2\cdot (C_\beta +2) \cdot U_\beta(j)}{U_\beta(m)}
$$ 
for all sufficiently large $m$. 
Consequently, there exists a sequence $((f_n,g_n))_{n\ge 0}$ converging to $(f,g)$ and such that $(f_n,g_n)\in \U_n$ for all $n$.
\end{proof}

We are now ready to prove the main result of this paper.

\begin{theorem}\label{thm:convergence-U-vers-L}
The sequence $(\U_n)_{n\ge 0}$ converges to $\Llim$.
\end{theorem}
\begin{proof}
Let $\epsilon>0$. 
From Lemma~\ref{lem:first_inclusion}, it suffices to show that $\Llim\subset [\U_n]_\epsilon$ for all sufficiently large $n\in\mathbb{N}$. 
For all $(x,y)\in\Llim$, using Corollary~\ref{cor:suite-Cauchy-point-segment}, there exists a (Cauchy) sequence $((f_i(x,y),g_i(x,y))_{i\ge 0}$ such that $(f_i(x,y),g_i(x,y))\in \U_i$ for all $i$, and there exists $N_{(x,y)}$ such that, for all $i,j\ge N_{(x,y)}$,
\begin{equation}
    \label{eq:1ein}
d((f_i(x,y),g_i(x,y)),(f_j(x,y),g_j(x,y)))<\epsilon/2    
\end{equation}
and $$d((f_i(x,y),g_i(x,y)),(x,y))<\epsilon/2.$$
We trivially have
$$\Llim\subset\bigcup_{(x,y)\in\Llim}B((f_{N_{(x,y)}}(x,y),g_{N_{(x,y)}}(x,y)),\epsilon/2).$$
Since $\Llim$ is compact, we can extract a finite covering: there exist $(x_1,y_1),\ldots,(x_k,y_k)$ in $\Llim$ such that
$$\Llim\subset\bigcup_{j=1}^kB((f_{N_{(x_j,y_j)}}(x_j,y_j),g_{N_{(x_j,y_j)}}(x_j,y_j)),\epsilon/2).$$
Let $N=\max_{j=1,\ldots,k} N_{(x_j,y_j)}$. 
From \eqref{eq:1ein}, we deduce that, for all $j\in\{1,\ldots,k\}$ and all $n\ge N$, 
$$B((f_{N_{(x_j,y_j)}}(x_j,y_j),g_{N_{(x_j,y_j)}}(x_j,y_j)),\epsilon/2) \subset B((f_{n}(x_j,y_j),g_{n}(x_j,y_j),\epsilon)$$
and therefore
$$\Llim\subset\bigcup_{j=1}^kB((f_{n}(x_j,y_j),g_{n}(x_j,y_j)),\epsilon)\subset [\U_n]_\epsilon.$$
\end{proof}

\begin{remark}
As in~\cite{LRS1}, the results mentioned above can be extended to any prime number. 
Let $q$ be a prime number and $r$ be a positive residue in $\{1,\ldots,q-1\}$. 
We can extend Definition~\ref{def:Un} to 
$$
\U_  {n,r}:= \frac{1}{U_\beta(n)} \bigcup \left\{(\val_{U_\beta}(v),\val_{U_\beta}(u))+Q\mid u,v\in \Lang, \binom{u}{v}\equiv r\bmod{q}\right\}\subset [0,1]\times [0,1].
$$
Since we make use of Lucas' theorem, we limit ourselves to congruences modulo a prime number. 
We just sketch the main differences with the case $q=2$. 

See, for instance, Figure~\ref{fig:FibMod3} for the case $\beta=\varphi$, $q=3$ and $r=2$.

The $(\star)$ condition from Definition~\ref{def:cond*} becomes $(\star)_r$. 
We say that $(u,v) \in \Lang \times \Lang$ satisfies the $(\star)_{r}$ condition if either $u=v=\varepsilon$ and $\binom{u}{v}\equiv r\bmod{q}$, or $|u| \ge |v| >0$ and  
$$\binom{u0^{p(u,v)}}{v0^{p(u,v)}}\equiv r\bmod{q} \quad \text{and} \quad \binom{u0^{p(u,v)}}{v0^{p(u,v)}a} = 0 \quad \forall \, a \in \Alp $$
where $p(u,v)$ is defined using Proposition~\ref{pro:def-p}. 
In this extended context, Proposition~\ref{pro:binom-equal-1}, Proposition~\ref{pro:completion}, Corollary~\ref{cor:completion}, Lemma~\ref{lem:origine-extremite} and Remark~\ref{rem:conv-diag} are easy to adapt. 
Note that the pairs $(u,v)$ satisfying this condition depend on the choice of $q$ and $r$. 
The sets $\AEns_n$ are defined as before.
The pair $(u,u)$ satisfies $(\star)_r$ if and only if $r=1$; see Lemma~\ref{lem:u-u}. 
Thus, the segment of slope $1$ with endpoints $(0,0)$ and $(1,1)$ belongs to $\AEns_0$ if and only if $r=1$. 
An alternative proof of Proposition~\ref{prop:cauchy_seq} follows the same lines as in~\cite{LRS1}.
\end{remark}

\section{Appendix}\label{sec:appendix}

\begin{example}
We have represented the set $\mathcal{U}^\varphi_9$ in Figure~\ref{fig:exemples-U-2}.
\begin{figure}[h!tb]
   \centering
   \includegraphics[scale=0.4]{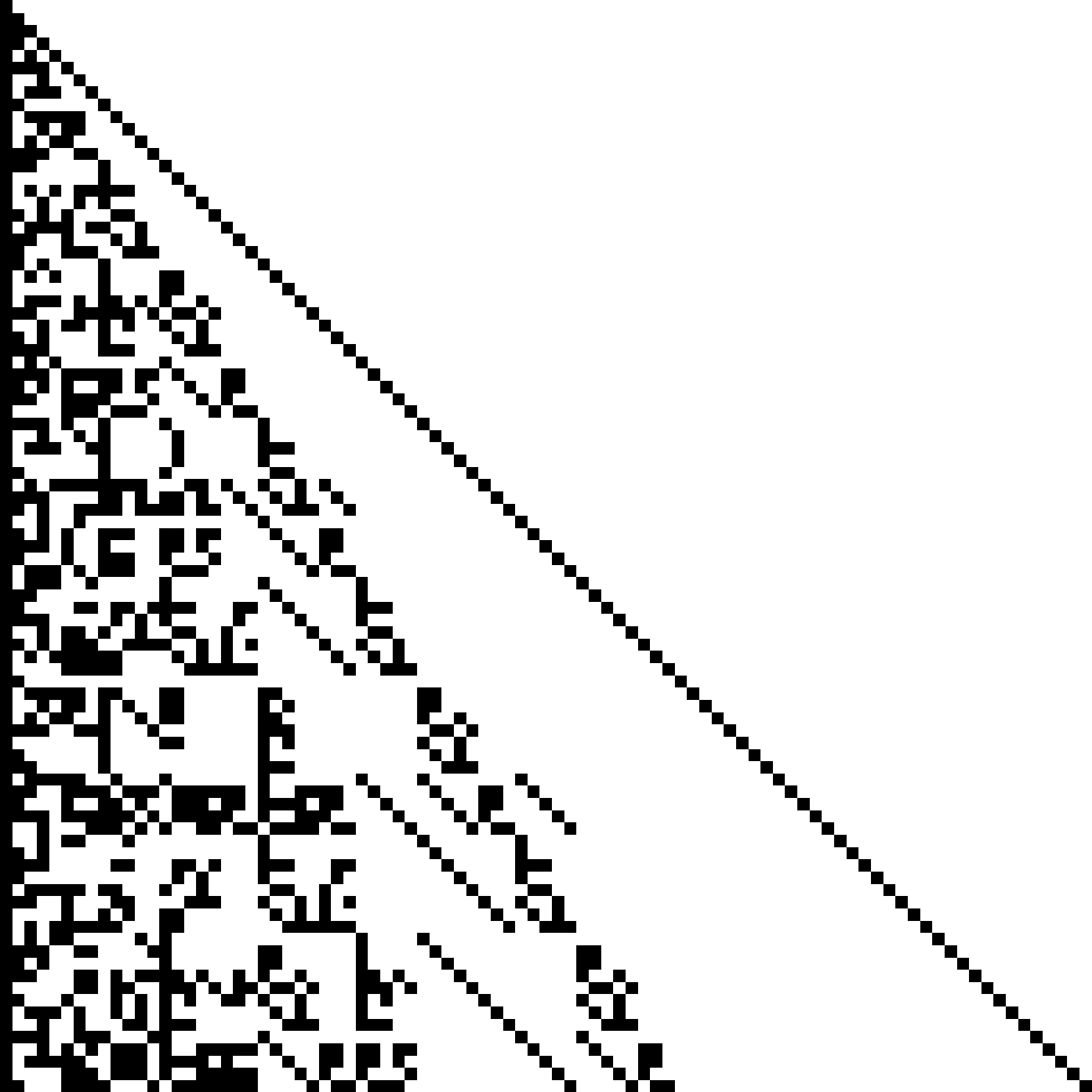}
   \caption{The set $\mathcal{U}^\varphi_9$.}
   \label{fig:exemples-U-2}
\end{figure}
\end{example}

\begin{example}
Let us consider the case when $\beta=\varphi$ is the golden ratio. We have represented in Figure~\ref{fig:FibMod3} the set $\mathcal{U}^\varphi_{9,2}$ when considering binomial coefficients congruent to $2$ modulo~$3$ and an approximation of the limit set $\mathcal{L}^\varphi$ proceeding as in Example~\ref{exa:approx-pour-Fib}.
\begin{figure}[h!tb]
   \centering
\begin{subfigure}[b]{0.45\textwidth}
\begin{center}
       \scalebox{.4}{\includegraphics{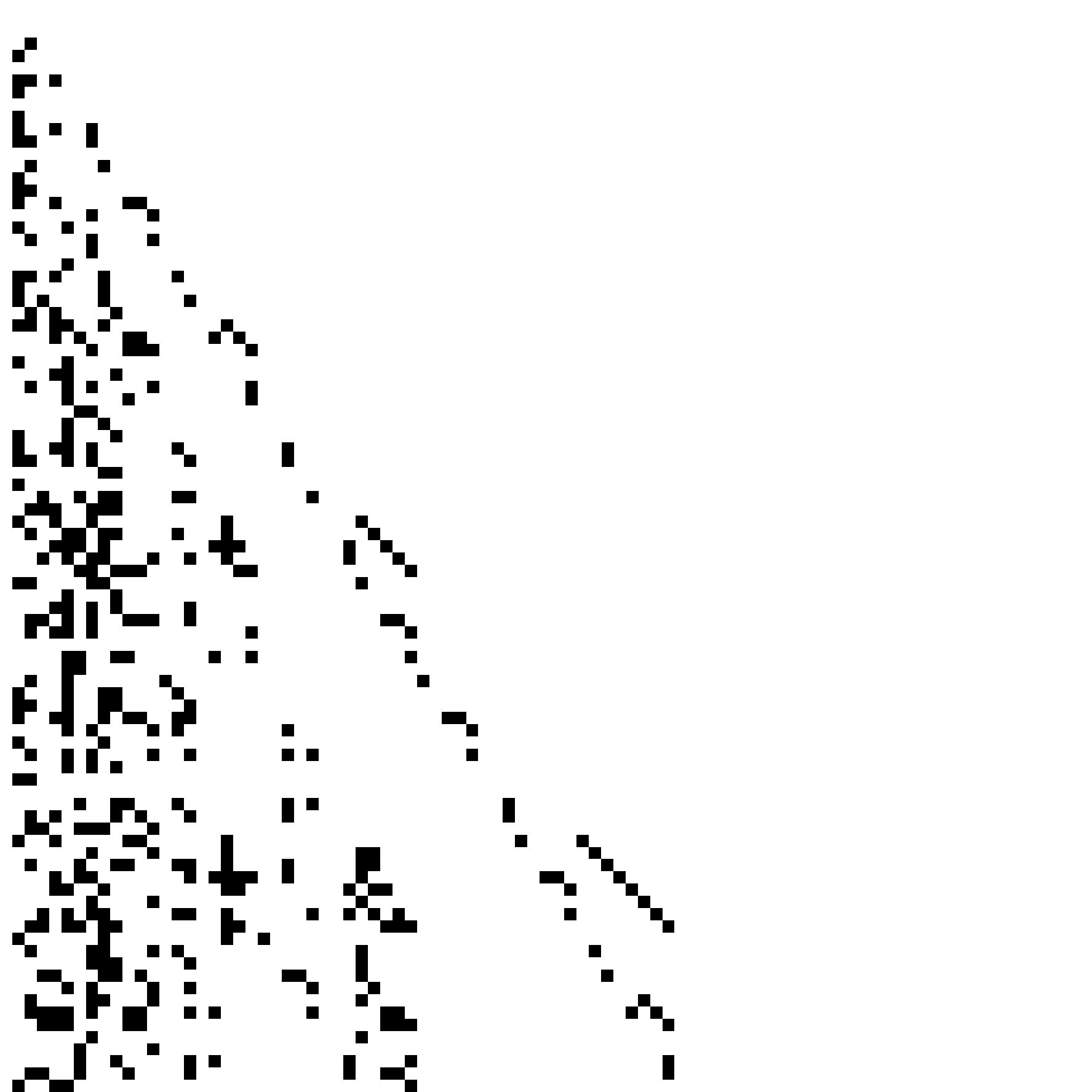}}
   \end{center}
\end{subfigure}  
\begin{subfigure}[b]{0.45\textwidth}
\begin{center}
       \scalebox{.225}{\includegraphics{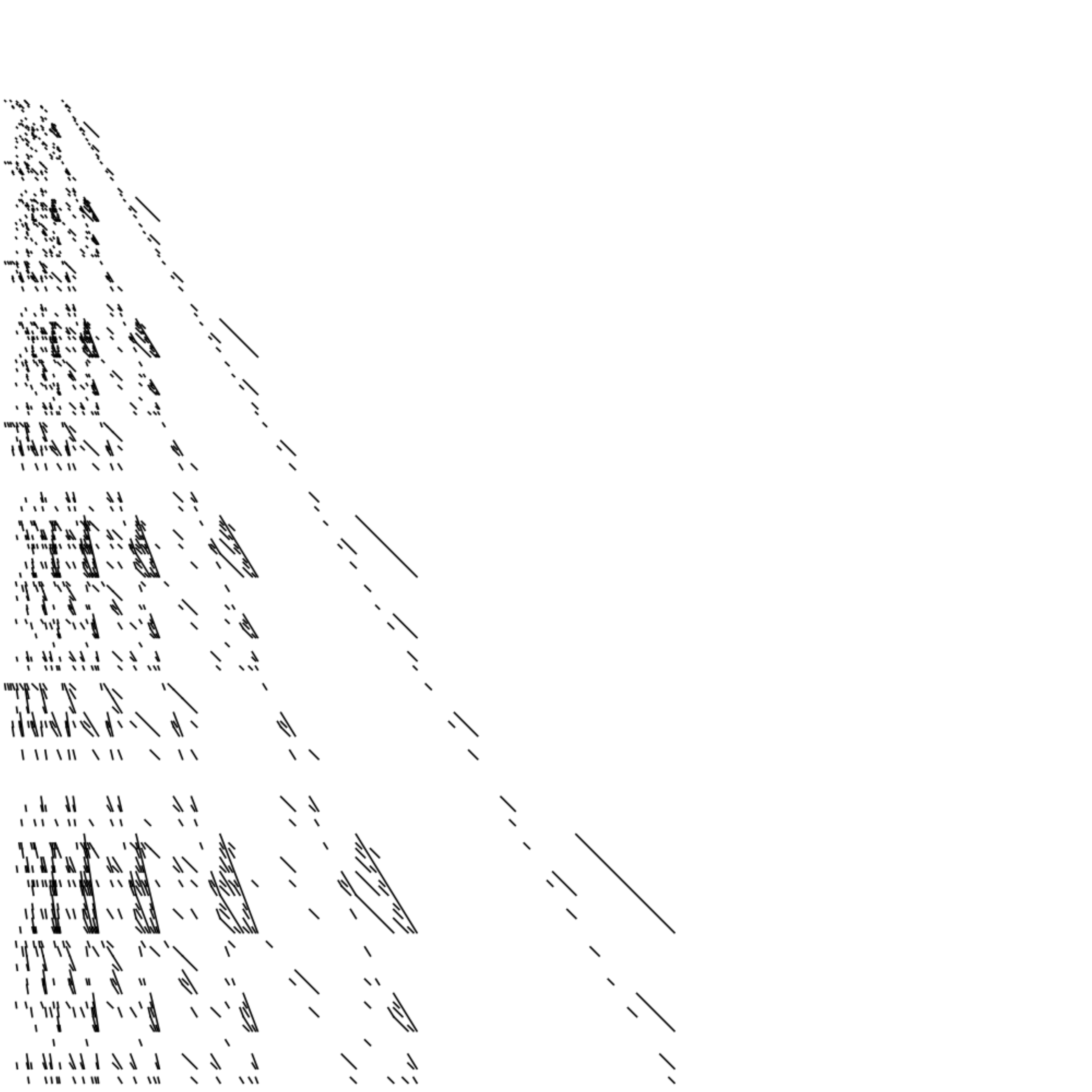}}
   \end{center}
\end{subfigure}
   \caption{The set $\mathcal{U}^\varphi_{9,2}$ (on the left) and an approximation of the corresponding limit set $\mathcal{L}^\varphi$ (on the right).}
   \label{fig:FibMod3}
\end{figure}
\end{example}

In this last example, we give an approximation of the limit object $\Llim$ for several different values of $\beta$. 
A real number $\beta >1$ is a \textit{Pisot number} if it is an algebraic integer whose conjugates have modulus less than $1$.

\begin{example}
Let us define several Parry numbers. 
Let $\beta_1\approx 2.47098$ be the dominant root of the polynomial $P(X)=X^4-2X^3-X^2-1$, which is a Parry and Pisot number; see Example~\ref{ex:p-beta-plusgrand2}.
Let $\beta_2\approx 2.47098$ be the dominant root of the polynomial $P(X)=X^4-X^3-1$, which is a Parry and Pisot number; see Example~\ref{ex:Parry2}.
Let $\beta_3\approx 2.80399$ be the dominant root of the polynomial $P(X)=X^4-2X^3-2X^2-2$. 
We can show that $\beta_3$ is a Parry number, but not a Pisot number.
Let $\beta_4\approx 1.32472$ be the dominant root of the polynomial $P(X)=X^5-X^4-1$. 
We can show that $\beta_4$ is a Parry number and also the smallest Pisot number.
In Figure~\ref{fig:SeveralApproxLimitObject}, we depict an approximation of $\Llim$ for $\beta$ in $\{\varphi^2, \beta_1, \ldots, \beta_4\}$.
 
\begin{figure}[h!tb]
   \centering
\begin{subfigure}[b]{0.45\textwidth}
\begin{center}
       \scalebox{.55}{\includegraphics{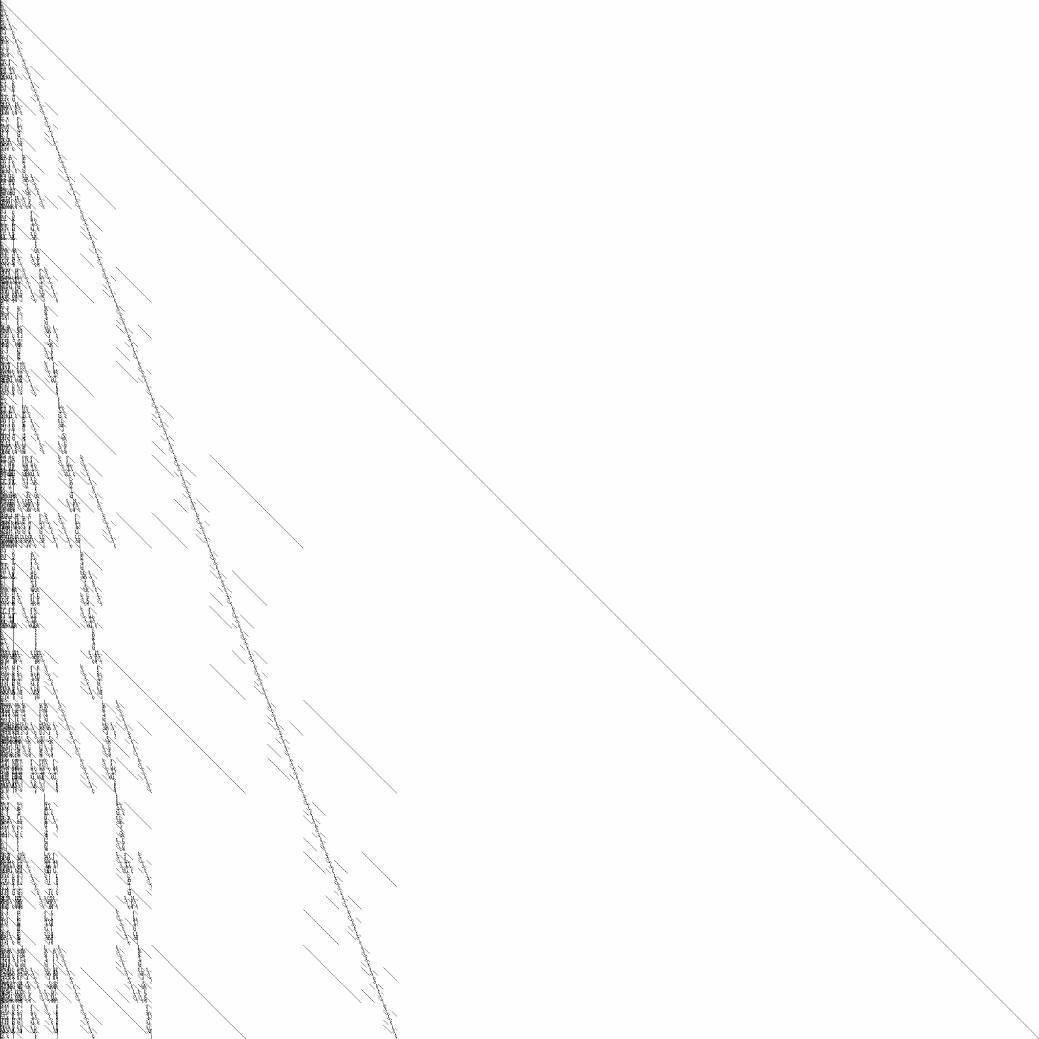}}
   \end{center}
   \caption{An approximation of $\mathcal{L}^{\varphi^2}$.}
\end{subfigure}  
\begin{subfigure}[b]{0.45\textwidth}
\begin{center}
       \scalebox{.24}{\includegraphics{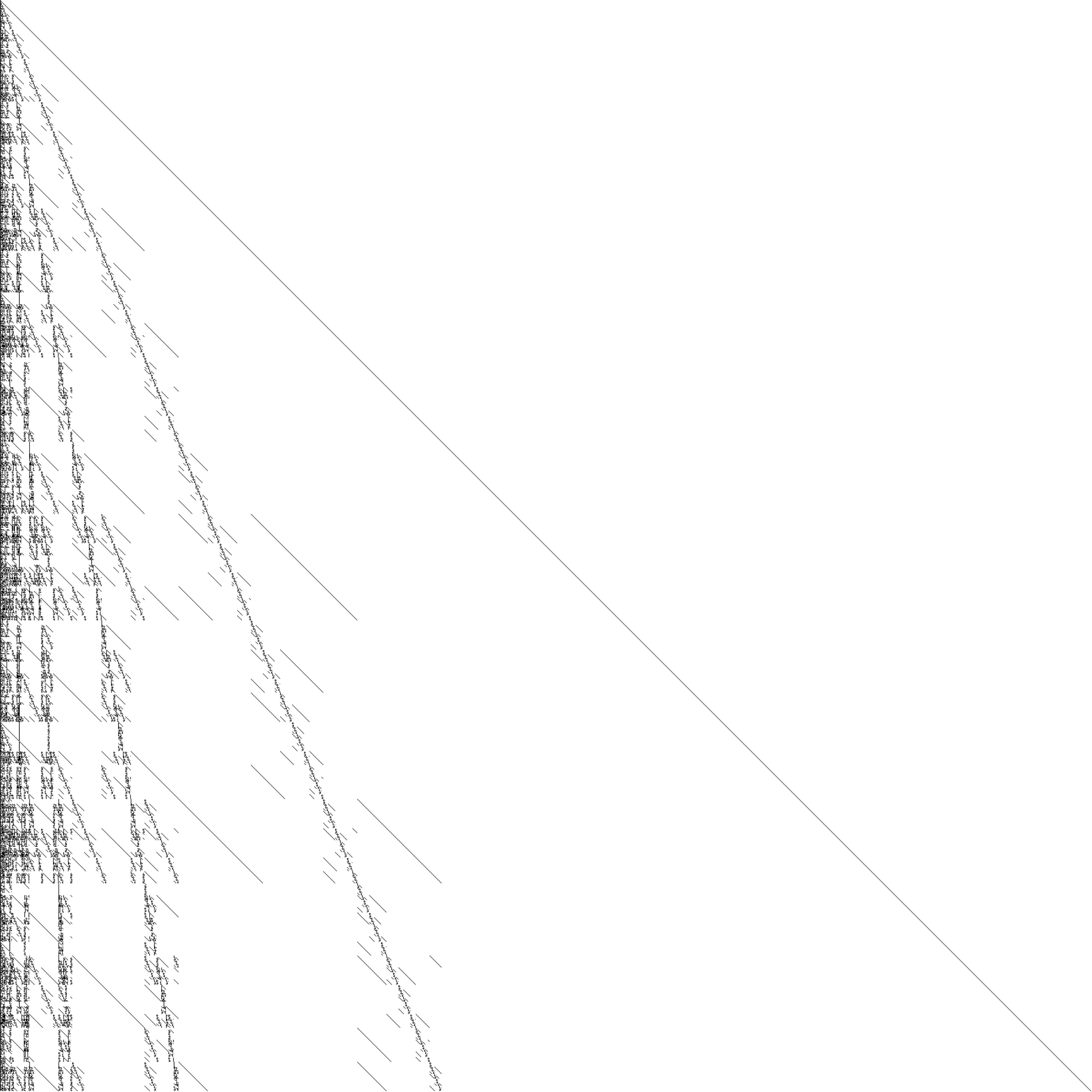}}
   \end{center}
   \caption{An approximation of $\mathcal{L}^{\beta_1}$.}
\end{subfigure}
\begin{subfigure}[b]{0.45\textwidth}
\begin{center}
       \scalebox{.245}{\includegraphics{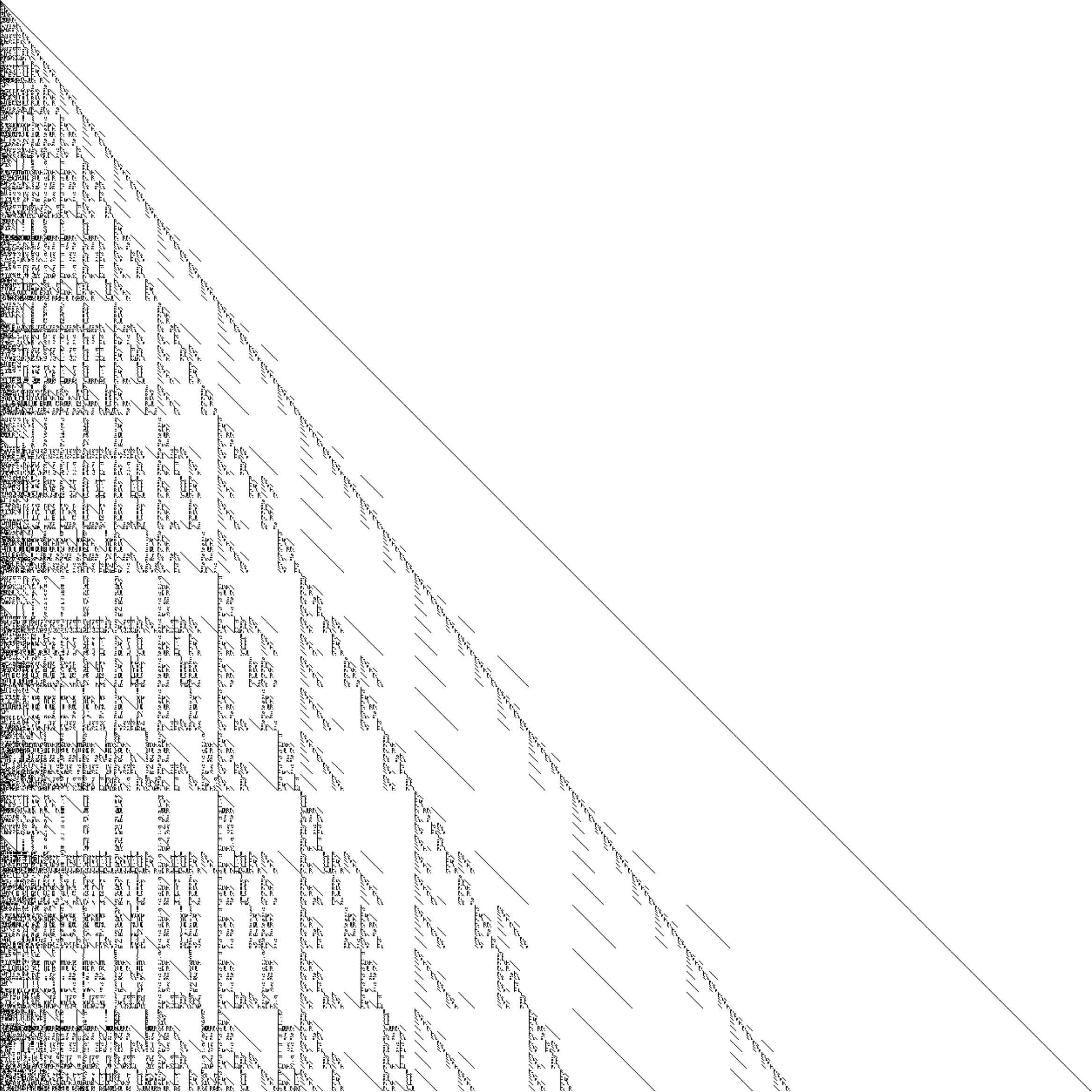}}
   \end{center}
      \caption{An approximation of $\mathcal{L}^{\beta_2}$.}
\end{subfigure}
\begin{subfigure}[b]{0.45\textwidth}
\begin{center}
       \scalebox{.25}{\includegraphics{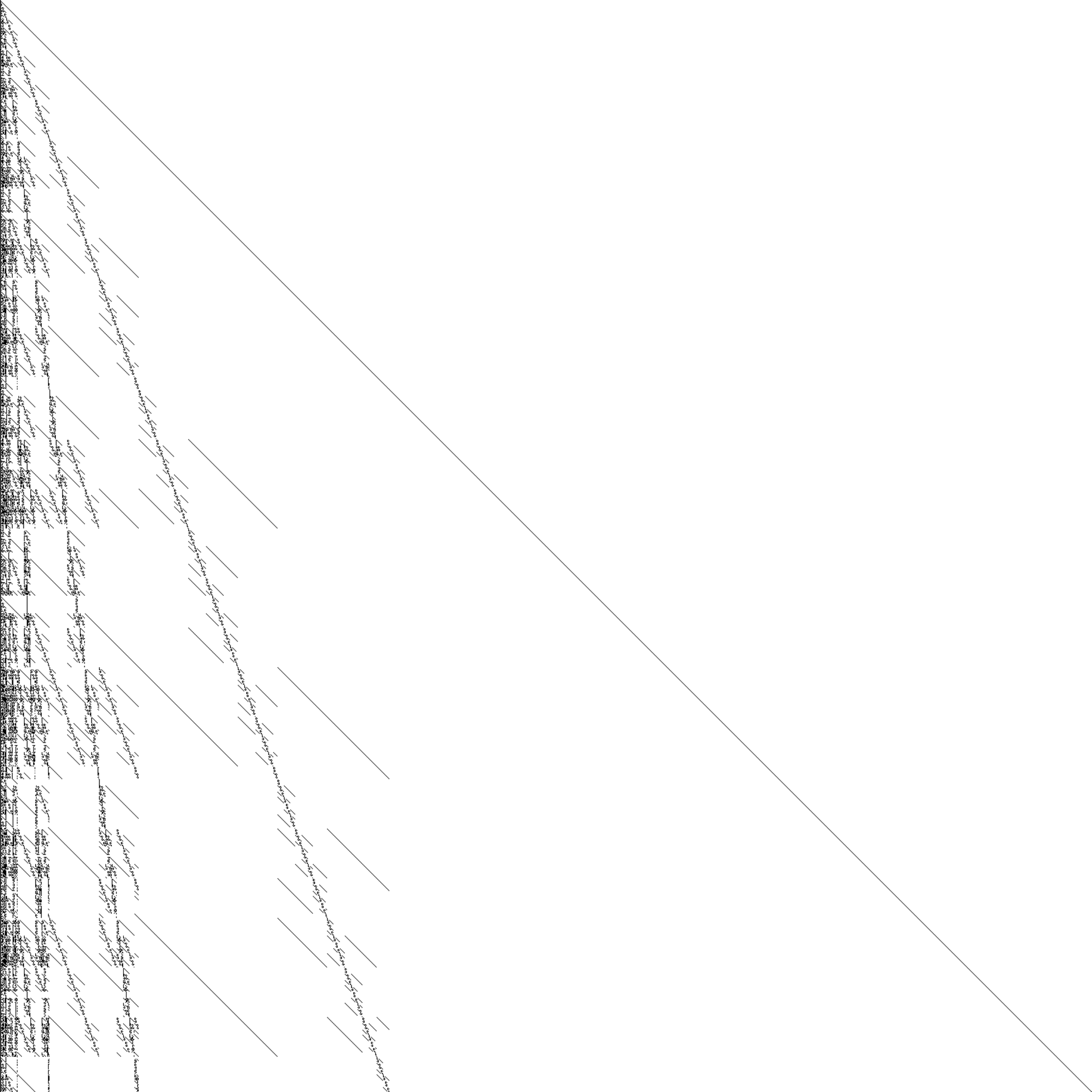}}
   \end{center}
      \caption{An approximation of $\mathcal{L}^{\beta_3}$.}
\end{subfigure}
\begin{subfigure}[b]{0.45\textwidth}
\begin{center}
       \scalebox{.215}{\includegraphics{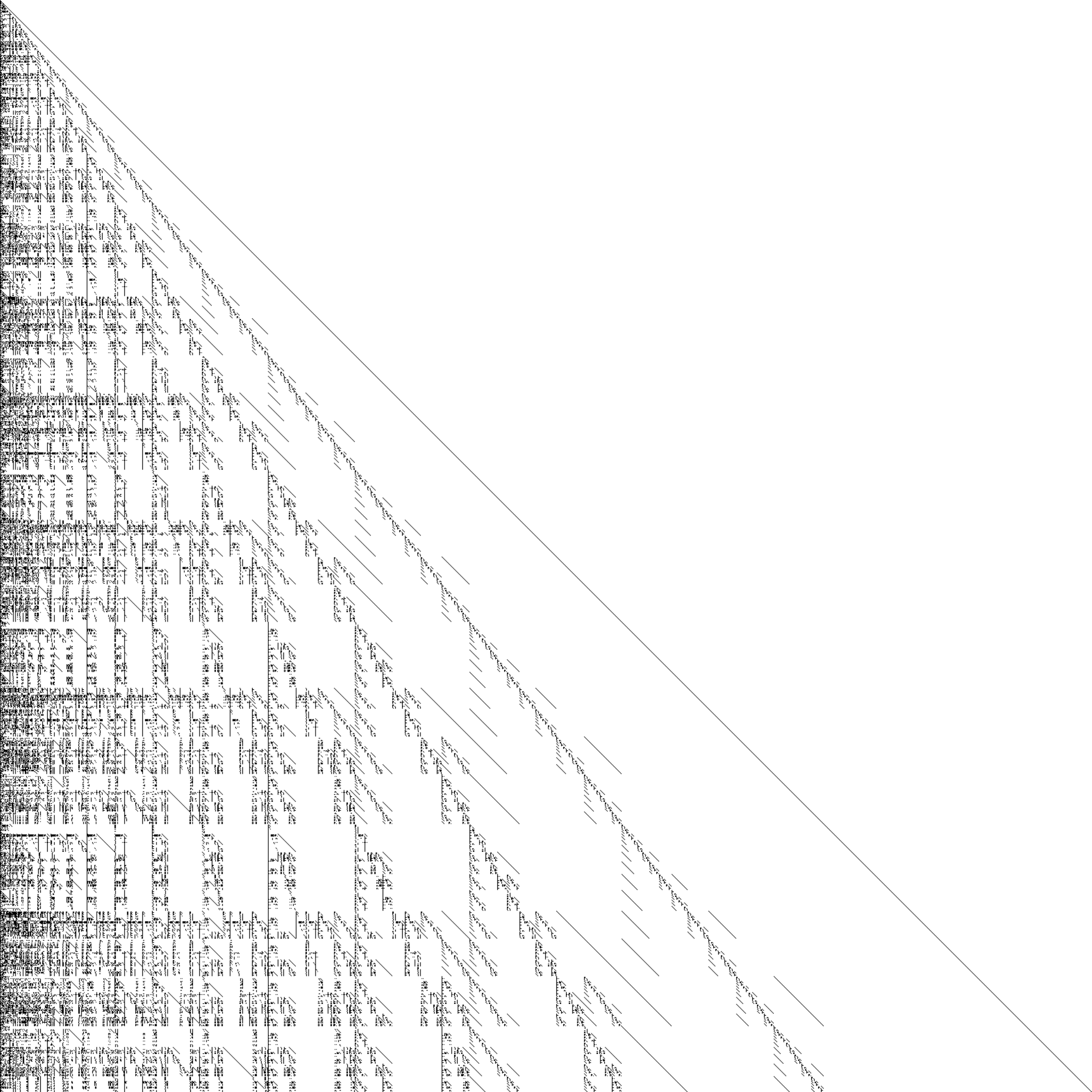}}
   \end{center}
      \caption{An approximation of $\mathcal{L}^{\beta_4}$.}
\end{subfigure}
   \caption{An approximation of the limit object $\Llim$ for different values of $\beta$.}
   \label{fig:SeveralApproxLimitObject}
\end{figure}
\end{example}

\section*{Acknowledgments}
This work was supported by a FRIA grant [grant number 1.E030.16].

The author wants to thank her advisor, Michel Rigo, and her colleague, Julien Leroy, for interesting scientific conversations, very useful comments on and improvements to a first draft of this paper.  



\begin{thebibliography}{99}



\bibitem{BNS} H. Belbachir, L. N\'emeth, and L. Szalay, Hyperbolic Pascal triangles, \emph{Appl. Math. Comput.} \textbf{273} (2016), 453--464.

\bibitem{BS} H. Belbachir and L. Szalay, On the arithmetic triangles, \emph{\v Siauliai Math. Semin.} \textbf{9} (2014), no. 17, 15--26.  

    \bibitem{BR} V. Berth\'e, M. Rigo (Eds.), {\em Combinatorics, automata and number theory},  Encycl. of Math. and its Appl. {\bf 135}, Cambridge University Press, 2010.

\bibitem{Bertrand} A. Bertrand-Mathis, Comment \'ecrire les nombres entiers dans une base qui n'est pas enti\`ere, \emph{Acta Math. Hungar} 54 (1989), 237--241.

\bibitem{CRRW} \'E. Charlier, N. Rampersad, M. Rigo, L. Waxweiler, The minimal automaton recognizing $m\mathbb{N}$ in a linear numeration system, \emph{Integers} \textbf{11B} (2011), Paper No. A4, 24 pp. 

  \bibitem{Falconer} K. Falconer, \emph{The Geometry of Fractal Sets}, Cambridge University Press, New York, 1985.
 
   \bibitem{Fine} N. Fine, Binomial coefficients modulo a prime, \emph{Amer. Math. Monthly} \textbf{54} (1947), 589--592.

  \bibitem{vonHPS}  F. von Haeseler, H.-O. Peitgen, and G. Skordev, Pascal's triangle, dynamical systems and attractors, {\em Ergod. Th. \& Dynam. Sys.} {\bf 12} (1992), 479--486.

\bibitem{JRV} \'E. Janvresse, T. de la Rue, and Y. Velenik, Self-similar corrections to the ergodic theorem for the Pascal-adic transformation, \emph{Stoch. Dyn.} \textbf{5} (2005), no. 1, 1--25.

  \bibitem{LRS1} J. Leroy, M. Rigo, and M. Stipulanti, Generalized Pascal triangle for binomial coefficients of words, {\em Adv. in Appl. Math.} {\bf 80} (2016), 24--47. 

    \bibitem{LRS2} J. Leroy, M. Rigo, and M. Stipulanti, Counting the number of non-zero coefficients in rows of generalized Pascal triangles, \emph{Discrete Math.} \textbf{340} (2017), 862--881.

      \bibitem{LRS3} J. Leroy, M. Rigo, and M. Stipulanti, Behavior of digital sequences through exotic numeration systems, \emph{Electron. J. Combin.} \textbf{24} (2017), no. 1, Paper 1.44, 36 pp.

      \bibitem{LRS4} J. Leroy, M. Rigo, and M. Stipulanti, Counting Subword Occurrences in Base-$b$ Expansions, \textit{to appear in Integers}.

  \bibitem{Lot} M. Lothaire, \emph{Combinatorics on Words}, Cambridge Mathematical Library, Cambridge University Press, 1997.
  
  \bibitem{Lot2} M. Lothaire, \emph{Algebraic Combinatorics on Words}, Encyclopedia of Mathematics and Its Applications, Cambridge
University Press, vol. 90, 2002.

  \bibitem{Lucas} \'E. Lucas, Th\'eorie des fonctions num\'eriques simplement p\'eriodiques, Amer. J. Math. 1 (1878) 197--240.      

\bibitem{Parry} W. Parry, On the $\beta$-expansions of real numbers, \emph{Acta
Math. Acad. Sci. Hungar.} \textbf{11} (1960), 401--416.      
      
   \bibitem{Rigo1} M. Rigo, \emph{Formal languages, automata and numeration systems. 1. Introduction to combinatorics on words}, ISTE, London; John Wiley \& Sons, Inc., Hoboken, NJ, 2014. 
   
    \bibitem{Rigo2} M. Rigo, \emph{Formal languages, automata and numeration systems. 2. Applications to recognizability and decidability},  ISTE, London; John Wiley \& Sons, Inc., Hoboken, NJ, 2014.  
      
            
      \bibitem{Sloane} N. J. A. Sloane, The On-Line Encyclopedia of Integer Sequences. Published electronically
at \url{http://oeis.org}, 2017.
\end{thebibliography}
\end{document}